\documentclass{amsart}
\usepackage{amssymb}

\newcommand{\bP}{\mathbf P}

\newcommand{\bbC}{\mathbb C}
\newcommand{\bbK}{\mathbb K}
\newcommand{\bbN}{\mathbb N}
\newcommand{\bbQ}{\mathbb Q}
\newcommand{\bbR}{\mathbb R}
\newcommand{\bbZ}{\mathbb Z}

\newcommand{\cD}{\mathcal D}
\newcommand{\cG}{\mathcal G}
\newcommand{\cI}{\mathcal I}
\newcommand{\cM}{\mathcal M}
\newcommand{\cN}{\mathcal N}
\newcommand{\cV}{\mathcal V}

\newcommand{\fS}{\mathfrak S}

\newcommand{\uX}{\underline{X}}

\newcommand{\bzero}{{\mathbf 0}}

\newcommand{\ie}{{\it i.e}}

\DeclareMathOperator{\card}{Card}
\DeclareMathOperator{\h}{h}
\DeclareMathOperator{\Imag}{Im}
\DeclareMathOperator{\ord}{ord}
\DeclareMathOperator{\Real}{Re}

\begin{document}
\newtheorem{thm}{Theorem}[section]
\newtheorem{alg}[thm]{Algorithm}
\newtheorem{conj}[thm]{Conjecture}
\newtheorem{cor}[thm]{Corollary}
\newtheorem{lem}[thm]{Lemma}
\newtheorem{prop}[thm]{Proposition}

\newtheorem*{defn}{Definition}

\theoremstyle{definition}
\newtheorem{ex}{Example}

\theoremstyle{remark}
\newtheorem*{ack}{Acknowledgement}
\newtheorem*{rem-nonum}{Remark}
\newtheorem{rem}[thm]{Remark}

\numberwithin{equation}{section}

\makeatletter\let\@wraptoccontribs\wraptoccontribs\makeatother

\title[A kit for linear forms in three logarithms]{A kit for linear forms in three logarithms}
\author{Maurice Mignotte}
\address{Maurice Mignotte\\
        Universit\'{e} Louis Pasteur\\
        U. F. R. de math\'{e}matiques\\
        7, rue Ren\'{e} Descartes\\
        67084 Strasbourg Cedex\\
        France}
\email{mignotte@math.unistra.fr}
\author{Paul Voutier}
\address{Paul Voutier\\
        London\\
        UK}
\email{paul.voutier@gmail.com}
\contrib[with an appendix by]{Michel Laurent}
\address{Michel Laurent\\
		Aix-Marseille Universit\'{e}\\
		Institut de Math\'{e}matiques de Marseille\\
		163 Avenue de Luminy, Case 907\\
		13288 Marseille C\'{e}dex 9\\
        France}
\email{michel-julien.laurent@univ-amu.fr}

\date{\today}


\begin{abstract}
We provide a technique to obtain explicit bounds for problems that can be reduced
to linear forms in three complex logarithms of algebraic numbers.
This technique can produce bounds significantly better than general results on
lower bounds for linear forms in logarithms.
We give worked examples to demonstrate both the use of our technique and the
improvements it provides. Publicly shared code is also available.
\end{abstract}

\maketitle

\section{Introduction}

\subsection{Background}

Many problems in number theory can be reduced to linear forms in the logarithms
of algebraic numbers which have a very small absolute value (exponentially small
in the coefficients of the linear form) (see \cite{Bug} for a broad selection
of examples). So, lower bounds for these linear forms that exceed the upper bounds
and with all the constants involved being explicit reduce such problems to
a finite amount of computation. For example, it is in this way (along with the
use of reduction techniques as in \cite{TW} and \cite{BH} to handle the remaining
computation) that the solution of Thue equations is
now routine, included as a function in PARI/GP \cite{Pari} and other mathematical
software.

Lower bounds for linear forms in two or three logarithms have proven to have
especially broad and important applications. In the case of linear forms in three
logarithms, such applications include Baker's solution \cite{Ba1} of the conjecture
of Gauss that there are only nine imaginary quadratic fields with class number $1$;
Tijdeman's proof \cite{Tij} that there are at most finitely many solutions of Catalan's
equation; and the result of Shorey \& Stewart \cite{SS}, and independently Peth\H{o}
\cite{Pet}, that there are only finitely many perfect powers in any binary
recurrence sequence. The use of effectively computable lower bounds for linear forms in
three logarithms gives rise to effectively computable upper bounds for each of these problems.

In this paper, we present a method, our ``kit'', that can be used to get good upper
bounds on quantities associated to such problems. The present paper has its origins
in earlier versions of our kit due to the first author in \cite{BMS1} and \cite{BMS2}.
In fact, \cite{BMS1} and \cite{BMS2} provide good examples of how somewhat weaker
versions of our kit were used to solve completely some important number theory
problems.

Our method is the method of interpolation determinants introduced by Michel Laurent
in \cite{L1}, \cite{L2} and \cite{L3}. In the case of three logarithms, this method
was used by C.D. Bennett {\it et al.} \cite{Ben-C}. But the present paper brings
some progress when compared to~\cite{Ben-C}: we treat the general case of algebraic
numbers (not only multiplicatively independent rational integers, as in \cite{Ben-C})
and many important technical details have been improved, including new zero lemmas.

Our aim, suggested by the title ``{\it A kit\ldots}'', is to explain how to obtain
results for problems that reduce to the study of linear forms in three logarithms
of algebraic numbers.

\subsection{Steps of the kit}

The process contains five steps.

\vspace*{1.0mm}

\noindent
(1) obtain an upper bound for a linear form in logs associated with our problem.

\vspace*{1.0mm}

\noindent
(2) combining the upper bound in step~(1) with a general estimate of Matveev, we obtain
an upper bound, $B_{1}$, for the maximum of the absolute values of the coefficients
of the linear form.

\vspace*{1.0mm}

\noindent
(3) supposing the linear form in three logs is {\it non-degenerate}, we use the upper
bound $B_{1}$ to obtain a second upper bound, $B_{2}$.
If $B_{2}$ is smaller than $B_{1}$ we proceed to step~(4).

\vspace*{1.0mm}

\noindent
(4) supposing the linear form in three logs is {\it degenerate}, we consider it as
a linear form in two logarithms and we apply the results of 
Laurent \cite{L5} to this linear form, along with the upper
bound $B_{1}$, to get a third upper bound $B_{3}$.\\
At this point the quantity we have bounded above by
$\min \left\{ B_{1}, \max \left\{ B_{2},B_{3} \right\} \right\}$.

\vspace*{1.0mm}

\noindent
(5) repeat steps~(3) and (4) as often as desired to make the upper bound as
small as possible.

\vspace*{1.0mm}

In our experience, there is very little further improvement after 3 iterations
(see the tables at the end of each example subsection in Section~\ref{sect:eg}
for details).

\subsection{Uses for the kit}

Our kit is most suited to the case when at least one of the algebraic numbers in
the linear form is a variable. If all three are fixed algebraic numbers, it is much better to first
use Matveev's result stated below and then apply a reduction technique like the
LLL-algorithm \cite{LLL} or variants of the Baker-Davenport reduction technique
\cite{BD}, like that of Dujella-Peth\H{o} \cite{DP}.

\subsection{Numerical results}


In the first example in Section~\ref{sect:eg}, we are able to reduce
the upper bound on the quantity $p$ from about $2 \cdot 10^{12}$ obtained by
Matveev's result to $18 \cdot 10^{6}$. In the second example there, we do even
better, reducing the upper bound on $p$ from about $3 \cdot 10^{13}$ to
$25 \cdot 10^{6}$. In our experience, these are typical of
the improvements that can be expected from our kit.

To help readers use the kit, code written in Pari,
along with examples for how to use it, is available from the authors at
\url{https://github.com/PV-314/lfl3-kit}. We encourage readers to use this code
for their applications of the kit, using the examples and documentation as a
guide. This code has now been applied to previously published uses of the
kit (\cite{BGMP,BDMS,BMS1,BMS2} -- the code for
these is available in the above github repository) and several new
problems shared with us by researchers. Support is available from the second
author and we warmly welcome questions and suggestions from users.

Another feature of our work, and the above code, is the quality of the results.
It is reasonable to believe that the degenerate case should play no part in
Theorem~\ref{thm:main} below and that only \eqref{eq:o} should matter (see the
proofs in Chapter~7 of \cite{W}, for example). In the case of ``imaginary'' linear forms in logs
(see their definition at the start of Section~\ref{sec:prelims} below)
we are able to attain such optimal bounds with our code above, while for
``real'' linear forms in logs, our code produces bounds that are at most 50\%
larger than the optimal bounds.

\subsection{Future work}

We highlight here three areas where further work would lead to significant improvements in the results,
as well as being of considerable theoretical interest for other diophantine and transcendence problems.

\vspace*{1.0mm}

\noindent
(1) Adopting Waldschmidt's approach for the degenerate case.
See Remark~\ref{rem:wald} for more information. This could reduce the bounds by a factor of
approximately $1.5$, but more importantly simplify the statement of Theorem~\ref{thm:main},
eliminating the need for conditions~\eqref{eq:thm41-i} and \eqref{eq:thm41-ii}.

\vspace*{1.0mm}

\noindent
(2) Improving the multiplicity estimates in Lemma~\ref{lem:7.2}. Calculations suggest that
the estimate there should be roughly $\Theta \left( 2K, |\cI| \right)$ instead of
$\Theta \left( K, |\cI| \right)$. The main term on the left-hand side of \eqref{eq:o} would then
become $KL$, rather than $KL/2$. This would improve the bounds obtained by a factor of roughly $5$.

\vspace*{1.0mm}

\noindent
(3) Improving the zero estimate in Proposition~\ref{prop:zero-est}. Conjecturally, the constants on the
right-hand sides of \eqref{eq:zero-est-i}--\eqref{eq:zero-est-iii} should all be $1$. The most important of these
inequalities for our work is \eqref{eq:zero-est-iii}. Replacing $3K^{2}L$ by $K^{2}L$ would lead to a further reduction
by a factor of roughly $2$ in the bounds obtained.

\subsection{Structure of this article}

In Section~\ref{sect:result}, we first provide some conventions and notations that will be used
throughout this paper and then present our main result for linear forms in three logs in Theorem~\ref{thm:main}.
Section~\ref{sec:prelims} contains the lemmas required to prove it, along with
Matveev's result which we use in step~(2). Section~\ref{sect:proof} contains the proof
of Theorem~\ref{thm:main}.

Section~\ref{sect:how-to}
provides information on the choice of the parameters in Theorem~\ref{thm:main}.
This simplifies the use of Theorem~\ref{thm:main}, reducing the selection of the
required parameters to the choice of four parameters. The best choice of these
four parameters can be found by a quick and easy brute force search.

To demonstrate both the usage of our kit and its benefits, we provide two
examples in Section~\ref{sect:eg}, revisiting the linear forms in \cite{BMS1}
and \cite{BMS2}. We obtain significant improvements in both examples. The
second example also corrects the use of the kit in \cite{BMS2}.

Lastly, we include a zero estimate due to Michel Laurent in Appendix~A. This is
the unpublished zero estimate \cite{L4} used in \cite{BMS2}, as well as in an
earlier version of this paper. In fact, Laurent's result was responsible for the
original kit, as it allowed improvements over \cite{Ben-C}. It is also applicable
more generally than our situation here, so it will be of interest to other
researchers of diophantine and transcendence problems.

\subsection{Acknowledgements}

Foremost, our thanks go to Michel Waldschmidt. He first proposed investigating
linear forms in three logarithms to the second author nearly 30 years ago. Since
then, he has been very supportive and encouraging to both authors in many ways.
Similarly, Michel Laurent has been very generous to this project and to both
authors over the years. Damien Roy and Patrice Philippon thoughtfully answered
our many questions about zero estimates. Mike Bennett
and Yann Bugeaud also deserve our thanks as they were instrumental
in bringing both authors together to complete this work. Lastly, we thank the
referee for their very careful reading of our paper and their helpful comments.

\section{Results}
\label{sect:result}

\subsection{Conventions}
\label{subsect:conventions}

We start by presenting the type of linear forms in three logarithms that we shall study.
We consider three distinct non-zero algebraic numbers $\alpha_{1}$, $\alpha_{2}$
and $\alpha_{3}$, positive rational integers $b_{1}$, $b_{2}$, $b_{3}$ with
$\gcd \left( b_{1},b_{2},b_{3} \right)=1$, and the linear form
\begin{equation}
\label{eq:lfl-form}
\Lambda = b_{1} \log \alpha_{1}+b_{2}\log \alpha_{2}-b_{3}\log \alpha_{3} \neq 0.
\end{equation}

We restrict our study to the following two cases:

\begin{itemize}
\item {\bf the real case}: $\alpha_{1}$, $\alpha_{2}$ and $\alpha_{3}$ are real
numbers greater than $1$, and the logarithms of the $\alpha_{i}$'s are all real
and positive. Furthermore, we assume that $\alpha_{1}$, $\alpha_{2}$ and $\alpha_{3}$
are multiplicatively independent over $\bbQ$. Of course, then the $\log \alpha_{j}$'s
are $\bbQ$-linearly independent. For many applications, this last assumption holds,
so in practice this should cause little restriction.

\item {\bf the imaginary case}: $\alpha_{1}$, $\alpha_{2}$ and $\alpha_{3}$ are complex
numbers $\neq 1$ of modulus one, and the logarithms of the $\alpha_{i}$ are arbitrary
determinations of the logarithm (then any of these determinations is purely imaginary).
Similar to the previous case, here we will assume that at least two of these
$\alpha$'s are multiplicatively independent over $\bbQ$ and the third one, if not
multiplicatively independent of the other two, is a root of unity.
We shall see later (see Lemma~\ref{lem:condM2})
that in this case, the $\log \alpha_{j}$'s are again $\bbQ$--linearly independent.
Once again, in practical examples, this last condition holds.
\end{itemize}

In practice, these restrictions do not cause any inconvenience since
\[
\left| \Lambda \right| \geq \max \left\{ \left| \Real(\Lambda) \right|, \left| \Imag(\Lambda) \right| \right\}.
\]

After possibly rearranging the terms and possibly replacing some logarithms by
their negatives in the imaginary case, we may assume that
\[
b_{3} \left| \log \alpha_{3} \right|
= b_{1} \left| \log \alpha_{1} \right| + b_{2} \left| \log \alpha_{2} \right|
\pm \left| \Lambda \right|.
\]

Notice that this introduces an important assymmetry between the
roles of the coefficients $b_{1}$, $b_{2}$ and $b_{3}$.

\vspace*{3.0mm}

Like the authors of \cite{Ben-C}, we use Laurent's method (see \cite{L1,L2}), and
consider a suitable interpolation determinant, $\Delta$. However, our
interpolation determinant differs from the one in \cite{Ben-C} (which was also
used in \cite{BMS1,BMS2}). We follow the construction of Waldschmidt in
Section~7.4 of \cite{W}. In examples, this change improves the
bounds we obtain by a factor of roughly $4$--$5$.

\subsection{Notation}
\label{subsect:notation}

We collect here some of the notation that we will use throughout this paper.

$\bbN$ will denote the set of non-negative rational integers.

$\bullet$ $K$, $L$, $R$, $S$, $T$ are positive rational integers
with $K \geq 3$ and $L \geq 5$.

$\bullet$ Put $N=K(K+1)L/2$ and we assume that $RST \geq N$.

$\bullet$ Let $i$ be an index from $1$ to $N$ such that $\left( k_{i},m_{i},\ell_{i} \right)$
runs through all triples of integers with $k_{i} \geq 0$, $m_{i} \geq 0$,
$k_{i}+m_{i} \leq K-1$ and $0 \leq \ell_{i} \leq L-1$. So each $0 \leq k_{i} \leq K-1$
occurs $\left( K-k_{i} \right) L$ times, and similarly each $m_{i}$ occurs
$\left( K-m_{i} \right) L$ times, and each number $0$, \dots, $L-1$
occurs $K(K+1)/2$ times as an $\ell_{i}$.

This is the main difference with the construction in \cite{Ben-C}, where
the conditions $0 \leq k_{i}, m_{i} \leq K-1$ are used instead.

$\bullet$ Put
\begin{equation}
\label{eq:g-defn}
g=\frac{1}{4}-\frac{N}{12RST}, \quad
G_{1}=\frac{NLR}{2}g, \quad
G_{2}=\frac{NLS}{2}g, \quad
G_{3}=\frac{NLT}{2}g.
\end{equation}

$\bullet$ With $d_{1}=\gcd \left( b_{1},b_{3} \right)$ and $d_{2}=\gcd \left( b_{2},b_{3} \right)$,
put
\begin{equation}
\label{eq:bi-rels}
b_{1} =d_{1} b_{1}',\ 
b_{2} =d_{2} b_{2}'', \ 
b_{3} =d_{1} b_{3}' =d_{2} b_{3}'', \ 
\beta_{1}=b_{1}/b_{3} =b_{1}'/b_{3}', \ 
\beta_{2}=b_{2}/b_{3} =b_{2}''/b_{3}''.
\end{equation}

$\bullet$ Let
\begin{equation}
\label{eq:lambda-defns}
\lambda_{i} = \ell_{i} - \frac{L-1}{2},\quad
\eta_{0} = \frac{R-1}{2}+\beta_{1} \frac{T-1}{2},\quad
\zeta_{0} = \frac{S-1}{2}+\beta_{2} \frac{T-1}{2}.
\end{equation}

$\bullet$ Let
\begin{equation}
\label{eq:b-defn}
b = \left( b_{3}'\eta_{0} \right) \left( b_{3}''\zeta_{0} \right)
\left( \prod_{k=1}^{K-1}(k!)^{K-k} \right)^{-\frac{12}{K(K-1)(K+1)}}.
\end{equation}
Similar to the $b$ in Th\'{e}or\`{e}me~1 of \cite{LMN}, this quantity arises
naturally in our proof -- see the end of the proof of Proposition~\ref{prop:M1}.

The expression involving the product of factorials here is also different
from that in \cite{Ben-C}, due to our different construction.

$\bullet$ Now we define the interpolation determinant that we shall use to
prove our results,
\begin{equation}
\label{eq:delta-defn}
\Delta = \det
\left( \binom{r_{j}b_{3}'+t_{j}b_{1}'}{k_{i}} \binom{s_{j}b_{3}''+t_{j}b_{2}''}{m_{i}}
	\alpha_{1}^{\ell_{i} r_{j}}\alpha_{2}^{\ell_{i} s_{j}}\alpha_{3}^{\ell_{i} t_{j}}
\right),
\end{equation}
where $1 \leq i, j \leq N$, $r_{j}$, $s_{j}$ and $t_{j}$ are non-negative integers
less than $R$, $S$ and $T$, respectively, such that $\left( r_{j},s_{j},t_{j} \right)$
runs over $N$ distinct triples.

$\bullet$ Lastly, with $r_{j}$, $s_{j}$ and $t_{j}$ as above in the definition
of our interpolation determinant, we let
\[
M_{1}=\frac{L-1}{2}\sum_{j=1}^{N} r_{j}, \qquad
M_{2}=\frac{L-1}{2}\sum_{j=1}^{N} s_{j}, \qquad
M_{3}=\frac{L-1}{2}\sum_{j=1}^{N} t_{j}.
\]

Here, and throughout, by $\alpha^{\beta}$, we mean $\exp \left( \beta \log \alpha \right)$
for any complex numbers $\alpha$ and $\beta$ with $\alpha \neq 0$ and some
determination of the logarithm.

\subsection{Main Theorem}
\label{subsect:main-thm}

With the above conventions and notation, we can present our main result.

\begin{thm}
\label{thm:main}
Let $\alpha_{1}$, $\alpha_{2}$ and $\alpha_{3}$ be three distinct non-zero algebraic
numbers which, along with their logarithms, satisfy one of the two conditions at
the start of this section.
Also let $b_{1}$, $b_{2}$, $b_{3}$ and $\Lambda$ be as there.
Assume that
\[
0 < \left| \Lambda \right| < 2\pi/w,
\]
where $w$ is the maximal order of a root of unity belonging to the number field
$\bbQ \left( \alpha_{1},\alpha_{2},\alpha_{3} \right)$
\footnote{\ If $D$ is the
degree of this number field, then $\varphi(w)\leq D$, where $\varphi$ is the Euler
totient function. Using \cite[Theorem~15]{RS} and some calculation for small $w$,
we see that $\varphi(w)\geq (w/2)^{0.63}$, which implies $w<2D^{1.6}$. Hence
$0 < \left| \Lambda \right| < 2\pi/w$ is satisfied if
$0<\left| \Lambda \right| \leq \pi D^{-1.6}$ and then $\Lambda \not\in i\pi\bbQ$.
Obviously, $\Lambda \not\in i\pi\bbQ$ is also satisfied when $\Lambda$ is real and non-zero.}.

Let $R_{1}$, $R_{2}$, $R_{3}$, $S_{1}$, $S_{2}$,
$S_{3}$, $T_{1}$, $T_{2}$, $T_{3}$ be positive rational integers with
\begin{equation}
\label{eq:RST-lb}
R>R_{1}+R_{2}+R_{3},\quad S>S_{1}+S_{2}+S_{3}
\text{ and } T>T_{1}+T_{2}+T_{3}.
\end{equation}

Let $\rho \geq 2$ be a real number. Suppose that
\begin{equation}
\label{eq:o}
\left( \frac{KL}{2} + \frac{L}{2} - 0.37K-2 \right) \log \rho
\geq (\cD+1)\log N + gL \left( a_{1}R+a_{2}S+a_{3}T \right) + \frac{2\cD (K-1)\log b}{3}.
\end{equation}
where
\[
a_{i} \geq \rho \left| \log \alpha_{i} \right|
    - \log \left| \alpha_{i} \right| +2 \cD \h \left( \alpha_{i} \right) \quad \text{ for } \quad
    i=1,~2,~3.
\]

Put
$\cV = \sqrt{\left( R_{1}+1 \right) \left( S_{1}+1 \right) \left( T_{1}+1 \right)}$.
If, for some positive real number $\chi$,
\begin{align}
\label{eq:thm41-i}
& \left( R_{1}+1 \right) \left( S_{1}+1 \right) \left( T_{1}+1 \right)
> K \max \left\{ R_{1}+S_{1}+1, S_{1}+T_{1}+1, R_{1}+T_{1}+1, \chi \cV \right\}, \\
\label{eq:thm41-ii}
& \card \,\left\{ \alpha_{1}^{r}\alpha_{2}^{s}\alpha_{3}^{t} :
    0 \leq r \leq R_{1}, \,0 \leq s \leq S_{1}, \,0 \leq t \leq T_{1} \right\} > L, \\
\label{eq:thm41-iii}
& \card \,\left\{ \alpha_{1}^{r}\alpha_{2}^{s}\alpha_{3}^{t} :
    0 \leq r \leq R_{2}, 0 \leq s \leq S_{2}, 0 \leq t \leq T_{2} \right\} > 2KL, \\
\label{eq:thm41-iv}
& \left( R_{2}+1 \right) \left( S_{2}+1 \right) \left( T_{2}+1 \right) > K^{2} \text{ and} \\
\label{eq:thm41-v}
& \left( R_{3}+1 \right) \left( S_{3}+1 \right) \left( T_{3}+1 \right) > 3K^{2}L
\end{align}
all hold, then either
\[
\Lambda':=\left| \Lambda \right| \cdot
\frac{LTe^{LT \left| \Lambda \right|/(2b_{3})}}{2 \left| b_{3} \right|} 
> \rho^{-KL}
\]
or
at least one of the following conditions~\eqref{eq:C1} or \eqref{eq:C2} holds:
\begin{equation}
\label{eq:C1}
\left| b_{1} \right| \leq \max \left\{ R_{1}, R_{2} \right\} \quad \text{and} \quad
\left| b_{2} \right| \leq \max \left\{ S_{1}, S_{2} \right\} \quad \text{and} \quad
\left| b_{3} \right| \leq \max \left\{ T_{1}, T_{2} \right\},
\end{equation}
\begin{equation}
\label{eq:C2}
\text{there exist $u_{1}, u_{2}, u_{3} \in \bbZ$ such that
$u_{1}b_{1}+u_{2}b_{2}+u_{3}b_{3}=0$,
with $\gcd \left( u_{1}, u_{2}, u_{3} \right)=1$},
\end{equation}
\[
\left| u_{1} \right|
\leq \frac{(S_{1}+1)( T_{1}+1)}{\cM-\max\{S_{1},T_{1}\} }, \quad
\left| u_{2} \right|
\leq \frac{(R_{1}+1)( T_{1}+1)}{\cM-\max\{R_{1},T_{1}\} }
\quad \text{and} \quad
\left| u_{3} \right|
\leq \frac{(R_{1}+1)( S_{1}+1)}{\cM-\max\{R_{1},S_{1}\}},
\]
where
$
\cM = \max \left\{ R_{1}+S_{1}+1, S_{1}+T_{1}+1, R_{1}+T_{1}+1, \chi \cV \right\}$.
\end{thm}

\section{Preliminaries}
\label{sec:prelims}

\subsection{Matveev's theorem for three logarithms}
\label{subsect:matveev}

We will need the special case of three logarithms of the theorem
of E. M. Matveev. So we quote his result in this case here.

\begin{thm}[Matveev]
\label{thm:Mat}
Let $\alpha_{1}$, $\alpha_{2}$ and $\alpha_{3}$ be three distinct non-zero
algebraic numbers, let $\log \alpha_{1}$, $\log \alpha_{2}$ and $\log \alpha_{3}$
be $\bbQ$--linearly
independent logarithms of these algebraic numbers and let $b_{1}$, $b_{2}$ and
$b_{3}$ be rational integers with $b_{1} \neq 0$. Put
\[
\Lambda = b_{1} \log \alpha_{1} + b_{2} \log \alpha_{2} + b_{3} \log \alpha_{3}.
\]

Let
\[
D=\left[ \bbQ \left( \alpha_{1},\alpha_{2},\alpha_{3} \right): \bbQ \right]
\hspace*{3.0mm} \text{ and } \hspace*{3.0mm}
\chi = \left[ \bbR \left( \alpha_{1},\alpha_{2},\alpha_{3} \right) : \bbR \right].
\]

Let $A_{1}$, $A_{2}$ and $A_{3}$ be positive real numbers, which satisfy
\[
A_{j} \geq \max \left\{ D \h \left( \alpha_{j} \right), \left| \log \alpha_{j} \right| \right\}
\quad (1 \leq j \leq 3),
\]
where $\h$ is the absolute logarithmic Weil height.

Assume that
\[
B \geq \max \left\{ \left| b_{j} \right| A_{j}/A_{1} : 1 \leq j \leq 3 \right\}.
\]

Also define
\[
C_{1} = \frac{5 \cdot 16^{5}}{6\chi} e^{3} (7+2\chi)\left(\frac{3e}{2}\right)^{\chi}
        \left( 26.25 +\log \left( D^{2}\log (eD) \right) \right).
\]
Then
\[
\log \left| \Lambda \right| > -C_{1} D^{2} A_{1} A_{2} A_{3} \log \left( 1.5eDB\log(eD) \right).
\]

\end{thm}

\begin{proof}
This is a very slight simplification of Theorem~2.1 of \cite{Matveev} applied
with $n=3$. Our only change is to note that $\left| b_{j} \right| A_{j}/A_{1} \geq 1$
for $j=1$, so the outer max in Matveev's inequality
$B \geq \max \left\{ 1, \max \left\{ \left| b_{j} \right| A_{j}/A_{1} : 1 \leq j \leq 3 \right\} \right\}$
is not needed.
\end{proof}

It is because the $\log \alpha_{j}$'s are $\bbQ$--linearly independent in both of the cases
that we present in Subsection~\ref{subsect:conventions} that we can use Theorem~2.1 of \cite{Matveev}
in this work.

Note that it is also possible to use the results of Aleksentsev \cite{Al} in place
of Matveev's result. This would give a slightly smaller upper bound in Step~(2),
but make no difference to the final results obtained from the kit.

\subsection{Some combinatorial inequalities}

This subsection contains some results used in the
estimates of the interpolation determinant.

\begin{lem}
\label{lem:lemG}
Let $K$, $L$, $N$, $R$, $S$, $T$, $G_{1}$ and $M_{1}$ be as above. Put
\[
\ell_{n} = \left\lfloor \frac{2(n-1)}{K(K+1)} \right\rfloor, \quad 1 \leq n \leq N,
\]
and $\left( r_{1}, \ldots, r_{N} \right) \in \{0,1,\ldots,R-1\}^{N}$. Suppose that
for each $r \in \{ 0,1,\ldots,R-1 \}$ there are at most $ST$ indices $j$ such that
$r_{j}=r$. Then
\[
\left| \sum_{n=1}^{N} \ell_{n} r_{n}-M_{1} \right| \leq G_{1}.
\]
\end{lem}

\begin{proof}
Apply Lemme~4 in \cite{LMN} with $K$ there set to $K(K+1)/2$.
\end{proof}

As in \cite[Section~1.3]{Ben-C} or \cite[p.~192]{W}, for $(k,m)\in \bbN^{2}$,
we put $\Vert (k,m) \Vert=k+m$. And for any $I, K_{0} \in \bbN$, we put
\begin{equation}
\label{eq:Theta-defn}
\Theta \left( K_{0},I \right) = \min \left\{ \Vert \left( k_{1},m_{1} \right) \Vert
+ \cdots + \Vert \left( k_{I},m_{I} \right) \Vert \right\},
\end{equation}
where the minimum is taken over all the sets of $I$ pairs $\left( k_{1},m_{1} \right)$,
\dots, $\left( k_{I},m_{I} \right)\in \bbN^{2}$ which are pairwise distinct and
satisfy $m_{1}$, \dots, $m_{I} \leq K_{0}$. Then, we have

\begin{lem}
\label{lem:Theta}
Let $K_{0}$ and $I$ be positive integers with $I \geq K_{0} \left( K_{0}+1 \right)/2$. Then
\[
\Theta \left( K_{0},I \right) \geq 
\left( \frac{I^{2}}{2(K_{0}+1)} \right) 
\left( 1+ \frac{(K_{0}-1)(K_{0}+1)}{I} - \frac{K_{0}(K_{0}+2)(K_{0}+1)^{2}}{12I^{2}} \right).
\]
\end{lem}

\begin{rem-nonum}
This is an improvement of the Lemma~1.4 of \cite{Ben-C}. If $I \equiv 0 \bmod \left( K_{0}+1 \right)$
and $K_{0}$ even, then this result is best possible. In the worst cases,
the difference between the left and right sides is at most roughly $K_{0}/8$.
\end{rem-nonum}

\begin{proof}
We follow more or less the proof of Lemma~1.4 of \cite{Ben-C}, the main difference
being the introduction of the term $r$ in the expression for $I$ below.

The smallest value for the sum 
$\Vert \left( k_{1},m_{1} \right) \Vert + \cdots + \Vert \left( k_{I},m_{I} \right)\Vert$
is reached when we choose successively, for each integer $n=0$, 1, \dots\ all
the points in the domain
\[
D_{n} = \left\{ (k,m)\in \bbN^{2} : m \leq K_{0} \text{ and } k+m=n \right\},
\]
and stop when the total number of points is $I$. Moreover,
\[
\card \left( D_{n} \right) =\begin{cases}
n+1,     & \text{if $n \leq K_{0}$,} \\
K_{0}+1, & \text{if \ $n \geq K_{0}$.}
\end{cases}
\]

Hence, for $A \geq K_{0}$, the number of points in $D_{0} \cup \cdots \cup D_{A-1}$ is
\[
\sum_{n=0}^{K_{0}-1} (n+1)+\sum_{n=K_{0}}^{A-1}(K_{0}+1)
= \frac{K_{0} \left( K_{0}+1 \right)}{2}+\left( A-K_{0} \right)\left( K_{0}+1 \right)
=\left( A - \frac{K_{0}}{2} \right)(K_{0}+1).
\]

Letting $A$ be the largest integer such that $\card \left( D_{0} \cup \cdots \cup D_{A-1} \right) \leq I$,
we can write
\[
I=\left( A - \frac{K_{0}}{2} \right) \left( K_{0}+1 \right)+r \quad \text{with $0 \leq r \leq K_{0}$},
\]
provided that $I \geq K_{0} \left( K_{0}+1 \right)/2$. Then
\[
\Theta \left( K_{0},I \right) = \sum_{n=0}^{K_{0}-1} n(n+1) +\sum_{n=K_{0}}^{A-1} n(K_{0}+1)+rA.
\]

Here
\begin{align*}
& \sum_{n=0}^{K_{0}-1} n(n+1) +\sum_{n=K_{0}}^{A-1} n(K_{0}+1) \\
= & \frac{(K_{0}-1) K_{0}(2K_{0}-1)}{6} + \frac{(K_{0}-1) K_{0}}{2}
+ \frac{K_{0}+1}{2} \left( A(A-1)-K_{0}(K_{0}-1)\right) \\
= & \frac{(K_{0}-1) K_{0} (2K_{0}+2)}{6} + \frac{K_{0}+1}{2}A(A-1)
- \frac{(K_{0}-1) K_{0} (K_{0}+1)}{2} \\ 
= & \frac{K_{0}+1}{2} \left(A(A-1) - \frac{K_{0}(K_{0}-1)}{3} \right)
\end{align*}
and we get
\begin{equation}
\label{eq:theta1}
\Theta \left( K_{0},I \right)
= \frac{K_{0}+1}{2}\left(A(A-1) - \frac{K_{0}(K_{0}-1)}{3} \right)+rA.
\end{equation}

We can write
\[
A = \frac{K_{0}}{2} + \frac{I-r}{K_{0}+1}.
\]

So using \eqref{eq:theta1} and then this expression for $A$ in terms of $r$,
we have
\[
\frac{\partial \Theta}{\partial r}
= \frac{K_{0}+1}{2} (2A-1) \frac{\partial A}{\partial r} + A 
+ r \frac{\partial A}{\partial r} = -\frac{2A-1}{2}+A-\frac{r}{K_{0}+1}
=\frac{1}{2}-\frac{r}{K_{0}+1},
\]
which shows that the minimum of $\Theta$ is reached either for $r=0$ or
$r=K_{0}$. It is easy to verify that $\Theta$ takes the same value for
$r=0$ and $r=K_{0}+1$ (which is indeed out of the range of $r$), this implies
that the minimum is reached for $r=0$. It follows that
\begin{align*}
\frac{2\Theta (K_{0},I)}{K_{0}+1}
& \geq \left( \frac{K_{0}}{2} + \frac{I}{K_{0}+1} \right) \left( \frac{K_{0}}{2} + \frac{I}{K_{0}+1}-1 \right)
- \frac{K_{0}(K_{0}-1)}{3} \\
& = \frac{K_{0}^{2}}{4} + \frac{I^{2}}{(K_{0}+1)^{2}} + \frac{K_{0} I}{K_{0}+1}
- \frac{K_{0}}{2} - \frac{I}{K_{0}+1}
- \frac{K_{0}^{2}}{3} + \frac{K_{0}}{3} \\
& = \frac{I^{2}}{(K_{0}+1)^{2}} + \frac{(K_{0} -1)I}{K_{0}+1} - \frac{K_{0}^{2}}{12}
- \frac{K_{0}}{6} \\
& = \left( \frac{I}{K_{0}+1} \right)^{2}
\left( 1+ \frac{(K_{0}-1)( K_{0}+1)}{I} - \frac{K_{0}(K_{0}+2)(K_{0}+1)^{2}}{12 I^{2}} \right).
\end{align*}

This proves the lemma.
%
\end{proof}

\begin{lem}
\label{lem:rho-exp-UB}
Let $K$, $L$ and $N$ be as in Subsection~$\ref{subsect:notation}$ with the additional
assumptions that $K \geq 3$ and $L \geq 5$. Also let $0 \leq I \leq N$ be an integer
and $\Theta \left( K_{0}, I \right)$ be as defined in $\eqref{eq:Theta-defn}$. Then
\[
KL(N-I)+\Theta \left( K-1, I \right)
\geq \frac{N^{2}}{2K} \left( 1+\frac{2}{L}-\frac{6}{KL}-\frac{1}{3L^{2}} \right).
\]
\end{lem}

\begin{proof}
Suppose that $I \leq N/2$. Then
\[
KL(N-I) \geq \frac{KLN}{2}=\frac{N^{2}}{K+1}=\frac{3}{2}\frac{N^{2}}{2K},
\]
since $K \geq 3$. Since $L \geq 5$, we have
$1+2/L-6/(KL)-1/\left( 3L^{2} \right)<1+2/L<3/2$. Since $\Theta(K-1,I) \geq 0$,
the result follows in this case.

We now consider $I> N/2$. This and $L \geq 5$ implies that $I \geq (5/4)K^{2}$,
so we can apply Lemma~\ref{lem:Theta} with $K_{0}=K-1$ to get
\[
KL(N-I)+ \Theta (K-1,I)
\geq KL(N-I)+ \frac{I^{2}}{2K}
\left( 1+ \frac{(K-2)K}{I} - \frac{(K-1)(K+1)K^{2}}{12 I^{2}} \right).
\]

The derivative of the right-hand side with respect to $I$ is
\[
\frac{2I-2K^{2}L+K^{2}-2K}{2K}.
\]

This is linear in $I$ and the coefficient of $I$ is positive, so once the derivative
is positive, it remains positive for all larger values of $I$.
It equals $0$ at $I=K^{2}L-K^{2}/2+K$.
We can write
\[
K^{2}L-K^{2}/2+K
=K(K+1)L/2 + \left( K^{2}/2-K/2 \right)(L-1)+K/2.
\]

For $K \geq 2$, we have $K^{2}/2-K/2 \geq 1$ and since $L \geq 1$, we have
$\left( K^{2}/2-K/2 \right)(L-1)+K/2 \geq 1/2$. So this critical value of $I$
is larger than $N$. Hence
the minimum value of the above lower bound for $KL(N-I)+\Theta(K-1,I)$ occurs
at $I=N$. Thus
\[
KL(N-I) + \Theta (K-1,I)
\geq \frac{N^{2}}{2K} \left( 1+\frac{2}{L}-\frac{6}{KL}-\frac{1}{3L^{2}} \right)
+ \frac{2K+18L}{3K(K+1)L^{2}},
\]
where the equality was obtained by using Maple. This implies that the lemma holds
for all $I$.
\end{proof}

\begin{lem}
\label{lem:Stir}
{\rm (a)} Let $K>1$ be an integer, then
\begin{equation}
\label{eq:k-factorial-bnd}
\log \left( \prod_{k=1}^{K-1}(k!)^{K-k}\right)^{12/(K(K-1)(K+1))}
\geq 2\log(K) - 11/3.
\end{equation}

\noindent
{\rm (b)} With $b$, $d_{1}$, $d_{2}$, $K$, $R$, $S$ and $T$ as
defined in Subsection~$\ref{subsect:notation}$, we have
\begin{align*}
\log b \leq & \log \frac{(R-1)b_{3}+(T-1)b_{1}}{2d_{1}}+\log \frac{(S-1)b_{2}+(T-1)b_{3}}{2d_{2}} \\
            & -2\log (K) + 11/3.
\end{align*}
\end{lem}

\begin{proof}
Our proof is a variant of the proof of Lemme~8 of \cite{LMN}, which itself is
based on the proof of Lemma~9 in \cite{L3}.

From the inequality $k! \geq (k/e)^{k}$, we have 
\begin{align*}
\sum_{k=0}^{K-1} (K-k) \log (k!)
& \geq \sum_{k=1}^{K-1} (K-k)k (\log (k) - 1) \\ 
&   =  \sum_{k=1}^{K-1} (K-k)k \log(k) - \sum_{k=0}^{K-1} (K-k)k.  
\end{align*}

The last term is easily shown to be $K(K+1)(K-1)/6$, so that 
\[
\sum_{k=1}^{K-1} (K-k) \log k!  
\geq \sum_{k=1}^{K-1} k(K-k) \log k - \frac{K(K+1)(K-1)}{6}.
\]

We now estimate the remaining sum, which we break into two sums: 
\[
K \sum_{k=1}^{K-1} k \log k - \sum_{k=1}^{K-1} k^{2} \log k.
\]

We use the Euler-Maclaurin summation formula to estimate these sums. We shall
use the formulation of the Euler-Maclaurin summation formula in equation~(7.2.4)
p.~303 of \cite{D} with $r=1$:
\[
f(1)+\cdots+f(n)=\int_{1}^{n} f(x) dx + \frac{f(1)+f(n)}{2} + \frac{f'(n)+f'(1)}{12}
+R_{1},
\]
where
\[
R_{1} \leq \frac{1}{2\pi^{2}} \int_{1}^{n} \left| f^{(3)}(x) \right| dx.
\]

From this point onward in the proof, we use Maple extensively to
perform the integrations and algebraic manipulations.

In this way, with $n=K-1$ and $f(n)=x\log(x)$, we have $f'(x)=\log(x)+1$ and
$f^{(3)}(x)=-x^{-2}$, so $1/ \left( 2\pi^{2} \right) \int_{1}^{K-1} x^{-2} dx=(K-2)/ \left( 2\pi^{2}(K-1) \right)$
and
\[
\sum_{k=1}^{K-1} k \log k 
\geq
\int_{1}^{K-1} x\log(x) dx + \frac{(K-1)\log(K-1)}{2} + \frac{\log(K-1)+2}{12}
-\frac{1}{2\pi^{2}}.
\]

With $n=K-1$ and $f(n)=x^{2}\log(x)$, we have $f'(x)=2x\log(x)+x$ and
$f^{(3)}(x)=2x^{-1}$, so $1/ \left( 2\pi^{2} \right) \int_{1}^{K-1} 2x^{-1} dx=\log(K-1)/\pi^{2}$
and
\[
\sum_{k=1}^{K-1} k^{2} \log k 
\leq
\int_{1}^{K-1} x^{2}\log(x) dx + \frac{(K-1)^{2}\log(K-1)}{2} + \frac{2(K-1)\log(K-1)+K}{12}
+\frac{\log(K-1)}{2\pi^{2}}.
\]

Combining these two estimates, along with
\[
\int_{1}^{K-1} x\log(x) dx
=\frac{\log(K-1)}{2}K^{2}-\frac{K^{2}}{4}-K\log(K-1)+\frac{K}{2} +\frac{\log(K-1)}{2}
\]
and
\[
\int_{1}^{K-1} x^{2}\log(x) dx
=\frac{\log(K-1)}{3}K^{3}-\frac{K^{3}}{9}
-\log(K-1)K^{2}+\frac{K^{2}}{3}
+\log(K-1)K-\frac{K}{3}
-\frac{\log(K-1)}{3}+\frac{2}{9},
\]
we obtain 
\[
\sum_{k=1}^{K-1} k(K-k) \log k  
\geq \frac{\log(K-1)K^{3}}{6}-\frac{11K^{3}}{36}+\frac{K^{2}}{6}
-\frac{\log(K-1)K}{12} + \left( \frac{7}{12}-\frac{1}{2\pi^{2}} \right)K
-\frac{\log(K-1)}{2\pi^{2}}-\frac{2}{9}.
\]

Subtracting $\left( 2\log(K)-11/3 \right) K(K-1)(K+1)/12$, we obtain
\begin{equation}
\label{eq:factsum-LB}
\frac{\log(1-1/K)}{6}K^{3} + \frac{K^{2}}{6} + \frac{\log \left( K^{2}/(K-1) \right)}{12}K
+ \left( \frac{5}{18} -\frac{1}{2\pi^{2}} \right) K-\frac{\log(K-1)}{2\pi^{2}}-\frac{2}{9}.
\end{equation}

From the series expansion of $\log(1-x)$, we find that
$\log(1-1/K)>-1/K-1/\left( 2K^{2} \right) - 2/\left( 3K^{3} \right)$ for $K \geq 2$,
so \eqref{eq:factsum-LB} is larger than
\[
\frac{\log(K)}{12}K
+ \left( \frac{7}{36} -\frac{1}{2\pi^{2}} \right) K
-\frac{\log(K-1)}{2\pi^{2}}-\frac{1}{3}
\]
for $K \geq 2$.

This expression is positive for $K \geq 3$, since $3/12>1/\left( 2\pi^{2} \right)$
and $\left( \frac{7}{36} -\frac{1}{2\pi^{2}} \right)>1/3$.
So part~(a) holds for $K \geq 3$.

Part~(a) also holds for $K=2$, since the left-hand side of \eqref{eq:k-factorial-bnd}
is $0$ for $K=2$, while the right-hand side is $-2.28\ldots$.

\vspace*{3.0mm}

(b) Using the definitions of $b$, $\eta_{0}$, $\zeta_{0}$, $\beta_{1}$ and $\beta_{3}$,
we have
\[
b = \left( b_{3}'\frac{R-1}{2} + b_{1}'\frac{T-1}{2} \right)
\left( b_{3}'' \frac{S-1}{2} + b_{2}''\frac{T-1}{2} \right)
\left( \prod_{k=1}^{K-1}k! \right)^{-\frac{12}{K(K-1)(K+1)}}.
\]

Applying the relationships in \eqref{eq:bi-rels}, part~(b) follows immediately.
\end{proof}

\subsection{An upper bound for $\left| \Delta \right|$}

In this subsection, we prove the result below, Proposition~\ref{prop:M1}, an
upper bound for $\left| \Delta \right|$ (also see \cite[Proposition~12.5]{BMS1}).
We start with an estimate for the zero multiplicity of a certain function, the
determinant of a particular matrix, at $x=0$. We closely follow Section~7.2 of
\cite{W}.

Let $K$ and $N$ be positive integers, $\eta_{1},\ldots, \eta_{N}$,
$\zeta_{1},\ldots, \zeta_{N}$ elements of $\bbC$, $f_{1},\ldots,f_{N}$ analytic
functions in $\bbC$, $\theta_{1}$ and $\theta_{2}$ non-zero complex numbers and
$p_{1},\ldots,p_{N}$ polynomials in $\bbC \left[ z_{1},z_{2} \right]$ of total
degree at most $K$. We define, for $1 \leq i \leq N$,
\[
\phi_{i} \left( z_{1},z_{2} \right)
= p_{i} \left( z_{1},z_{2} \right)
f_{i} \left( \theta_{1}z_{1}+\theta_{2}z_{2} \right).
\]


Let $\cI$ be a subset of $\{ 1,\ldots, N \}$. We define an $N \times N$ matrix
with entries
\[
\Phi_{\cI}(x)_{i,j}
=
\left\{
\begin{array}{ll}
\phi_{i} \left( x\eta_{j}, x\zeta_{j} \right), & \text{if $i \in \cI$},\\
\delta_{i,j} \phi_{i} \left( x\eta_{j}, x\zeta_{j} \right), & \text{if $i \not\in \cI$},
\end{array}
\right.
\]
where $\delta_{i,j}$ are complex numbers and let $\Psi_{\cI}(x)=\det \left( \Phi_{\cI}(x) \right)$.

\begin{lem}
\label{lem:7.2}
The function $\Psi_{I}(x)$ has a zero at $x=0$ of multiplicity at least
$\Theta \left( K, |\cI| \right)$,
where $|\cI|$ is the number of elements in $\cI$.
\end{lem}

\begin{proof}
This is Lemma~7.2 of \cite{W} in the case of $n=2$, since the total degree
of each of the polynomials, $p_{1},\ldots,p_{N}$ is at most $K$.
\end{proof}

Returning to our specific situation here, let $K$, $L$, $N$, $R$, $S$ and $T$, along with the $r_{j}$'s, $s_{j}$'s and
$t_{j}$'s be as defined in Subsection~\ref{subsect:notation}.

Recalling our definition of the $\lambda_{i}$'s in \eqref{eq:lambda-defns}, we have
\[
\sum_{i=1}^{N}\lambda_{i}=\frac{K(K+1)}{2}\sum_{i=0}^{L-1} \left( i - (L-1)/2 \right)=0
\]
and the following slight variation of equation~(2.1) in \cite{Ben-C} (our $\Lambda'$
is slightly different from theirs)
\begin{equation}
\label{eq:lambdaPrime-rel}
\alpha_{1}^{\lambda_{i} r_{j}}\alpha_{2}^{\lambda_{i} s_{j}}\alpha_{3}^{\lambda_{i} t_{j}}
= \alpha_{1}^{\lambda_{i}(r_{j}+t_{j}\beta_{1})}
\alpha_{2}^{\lambda_{i}(s_{j}+t_{j}\beta_{2})}e^{\lambda_{i} t_{j}\Lambda/b_{3}} 
= \alpha_{1}^{\lambda_{i}(r_{j}+t_{j}\beta_{1})}
\alpha_{2}^{\lambda_{i}(s_{j}+t_{j}\beta_{2})}(1+\theta_{i,j}\Lambda'),
\end{equation}
where
\[
\theta_{i,j}= \frac{ e^{\lambda_{i} t_{j}\Lambda/b_{3}}-1}{\Lambda'}
\]
and
\begin{equation}
\label{eq:lambda-prime-defn}
\Lambda'=\left| \Lambda \right| \cdot
    \frac{LTe^{LT \left| \Lambda \right|/(2b_{3})}}{2b_{3}}.
\end{equation}

Let
\begin{equation}
\label{eq:phi-defn}
\phi_{i}(\eta,\zeta)
= \frac{{b_{3}'}^{k_{i}}{b_{3}''}^{m_{i}}}{k_{i}!\,m_{i}!} \eta^{k_{i}}\zeta^{m_{i}}
\alpha_{1}^{\lambda_{i}\eta} \alpha_{2}^{\lambda_{i}\zeta},
\end{equation}
for any $i=1,\ldots, N$, and
\[
\Phi_{\cI} (x)_{i,j}=\begin{cases}
\phi_{i} \left( x\eta_{j}, x\zeta_{j} \right),             & \text{if $i \in \cI$,} \\
\theta_{i,j}\phi_{i} \left( x\eta_{j}, x\zeta_{j} \right), & \text{if $i \not\in \cI$,}
\end{cases}
\]
for any subset $\cI$ of $\cN=\{1,\ldots, N \}$ and $j=1,\ldots, N$.

In our notation before Lemma~\ref{lem:7.2}, here we put
$p_{i} \left( z_{1}, z_{2} \right)
=\frac{{b_{3}'}^{k_{i}}{b_{3}''}^{m_{i}}}{k_{i}!\,m_{i}!} z_{1}^{k_{i}}z_{2}^{m_{i}}$,
$f_{i}(z)=\exp \left( \lambda_{i} z \right)$
and $\theta_{i,j}=\delta_{i,j}$.
Hence we can write
$\alpha_{1}^{\lambda_{i}z_{1}}\alpha_{2}^{\lambda_{i}z_{2}}
=\exp \left( \lambda_{i} \left( \log \left( \alpha_{1} \right) z_{1}
+ \log \left( \alpha_{2} \right) z_{2} \right) \right)
=f_{i}\left( \theta_{1}z_{1}+\theta_{2}z_{2} \right)$ with
$\theta_{1}=\log \left( \alpha_{1} \right)$ and
$\theta_{2}=\log \left( \alpha_{2} \right)$.

We put $\cM_{\cI} = \left( \Phi_{\cI}(1)_{i,j} \right)$,
$\Psi_{\cI}(x)=\det \left( \Phi_{\cI} (x)_{i,j} \right)$,
\begin{equation}
\label{eq:deltaI-defn}
\Delta_{\cI}=\Psi_{\cI}(1)
\end{equation}
and
\[
J_{\cI} = \ord_{x=0} \left( \Psi_{\cI}(x) \right).
\]

\begin{prop}
\label{prop:M1}
Suppose $K$ and $L$ are two integers satisfying $K \geq 3$ and $L \geq 5$. If
\begin{equation}
\label{eq:star} 
\Lambda' < \rho^{-KL}
\end{equation}
holds for some real number $\rho \geq 2$, then
\begin{align*}
    \log \left| \Delta \right|
    < & \sum_{i=1}^{3} M_{i} \log \left| \alpha_{i} \right| + \rho \sum_{i=1}^{3} G_{i} \left| \log \alpha_{i} \right|
        +\log(N!)+N\log 2+ \frac{N(K-1)}{3} \log b \\
      & -\frac{N^{2}}{2K} \left( 1-\frac{2}{3L}-\frac{2}{3KL}-\frac{1}{3L^{2}}-\frac{16}{3K^{2}L} \right)
    \log \rho + 0.001.
\end{align*}
\end{prop}

\begin{proof}
We start by proving that $\left| \theta_{i,j} \right| \leq 1$.

Since $b_{3}$, $L$ and $\left| \Lambda \right|$ are all positive, $0 \leq t_{j} \leq T$,
and $\left| \lambda_{i} \right| \leq L/2$, we have
\[
\left| \theta_{i,j} \right| \leq \frac{e^{x}-1}{xe^{x}}
\hspace*{3.0mm} \text{ where } x= \frac{LT\left| \Lambda \right|}{2b_{3}}>0.
\]

Observe that $\left( e^{x}-1 \right)/\left( xe^{x} \right)$ is a decreasing function for
$x>0$, since its derivative is $\left( 1+x-e^{x} \right)/\left( x^{2}e^{x} \right)$.
By L'H\^{o}pital's rule, we find that $\lim_{x \rightarrow 0^{+}}
\left( e^{x}-1 \right)/\left( xe^{x} \right)
=\lim_{x \rightarrow 0^{+}} \left( 1+x \right)^{-1}=1$.
Hence,
\[
\left| \theta_{i,j} \right| \leq 1.
\]

Let
\[
\eta_{j}=r_{j}+t_{j}\beta_{1}-\eta_{0} \qquad \text{ and} \qquad \zeta_{j}=s_{j}+t_{j}\beta_{2}-\zeta_{0},
\]
so $\left| \eta_{j} \right| \leq \eta_{0}$ and $\left| \zeta_{j} \right| \leq \zeta_{0}$.
Since,
\[
\binom{r_{j}b_{3}'+t_{j}b_{1}'}{k_{i}}
=\binom{b_{3}' \left( \eta_{j}+\eta_{0} \right)}{k_{i}}
= \frac{{b_{3}'}^{k_{i}}}{k_{i}!} {\eta_{j}}^{k_{i}}+ \text{terms in $\eta_{j}$ of degree less than $k_{i}$},
\]
and similarly for $\binom{s_{j}b_{3}''+t_{j}b_{2}''}{m_{i}}$, using the multilinearity
of determinants we obtain the formula
\[
\Delta = \det \left( \frac{{b_{3}'}^{k_{i}}{b_{3}''}^{m_{i}}}{k_{i}!\, m_{i}!}
{\eta_{j}}^{k_{i}} {\zeta_{j}}^{m_{i}}{\alpha_{1}}^{\ell_{i} r_{j}}{\alpha_{2}}^{\ell_{i} s_{j}}{\alpha_{3}}^{\ell_{i} t_{j}}
\right).
\]

Combining this with \eqref{eq:lambdaPrime-rel}, along with the definitions of
$\lambda_{i}$, $M_{1}$, $M_{2}$ and $M_{3}$, it follows that
\[
\Delta = {\alpha_{1}}^{M_{1}} {\alpha_{2}}^{M_{2}} {\alpha_{3}}^{M_{3}}
\det \left( \frac{ {b_{3}'}^{k_{i}}{b_{3}''}^{m_{i}}}{k_{i}!\, m_{i}!}
\eta_{j}^{k_{i}} \zeta_{j}^{m_{i}}
\alpha_{1}^{\lambda_{i}(r_{j}+t_{j}\beta_{1})} \alpha_{2}^{\lambda_{i}(s_{j}+t_{j}\beta_{2})} 
\left( 1+\Lambda'\theta_{i,j} \right)
\right).
\]

Since $\sum_{i}\lambda_{i}=0$, we deduce from this and the definitions of $\eta_{j}$
and $\zeta_{j}$ that
\[
\Delta = \alpha_{1}^{M_{1}} \alpha_{2}^{M_{2}} \alpha_{3}^{M_{3}} 
\det \left( \frac{{b_{3}'}^{k_{i}}{b_{3}''}^{m_{i}}}{k_{i}!\, m_{i}!}
{\eta_{j}}^{k_{i}} {\zeta_{j}}^{m_{i}}
\alpha_{1}^{\lambda_{i}\eta_{j}} \alpha_{2}^{\lambda_{i}\zeta_{j}} 
\left( 1+\Lambda'\theta_{i,j} \right)
\right).
\]

Expanding this determinant, we obtain
\begin{equation}
\label{eq:delta-sum}
\Delta = {\alpha_{1}}^{M_{1}} {\alpha_{2}}^{M_{2}}{\alpha_{3}}^{M_{3}}
\sum_{\cI \subseteq \cN}(\Lambda')^{N-\left| \cI \right|}\Delta_{\cI},
\end{equation}
where $\cI$ runs over all subsets of $\cN=\{ 1,\ldots,N \}$ and $\Delta_{\cI}$
is defined in \eqref{eq:deltaI-defn}.

From Schwarz' Lemma (see, for example, Lemma~2.3 on page~37 of \cite{W}), we have
\begin{equation}
\label{eq:Psi1-ub}
\left| \Psi_{\cI}(1) \right| \leq \rho^{-J_{\cI}} \cdot \max_{\left| x \right| =\rho} \left| \Psi_{\cI}(x) \right|,
\end{equation}
recalling that $J_{\cI} = \ord_{x=0} \left( \Psi_{\cI}(x) \right)$.

Since $\left| \theta_{i,j} \right| \leq 1$, expanding the determinant $\Psi_{\cI}$
shows that
\[
\left| \Psi_{\cI}(x) \right|
\leq N! \max_{\sigma\in \fS (\cN)} \left| \prod_{i=1}^{N} \phi_{i} \left( x\eta_{\sigma(i)}, x\zeta_{\sigma(i)} \right) \right|,
\]
where $\fS(\cN)$ is the group of all permutations of $\cN$. For any $\sigma \in \fS(\cN)$
and any $x$ satisfying $|x| \leq \rho$, we also have
\[
\left| \prod_{i=1}^{N} \phi_{i} \left( x\eta_{\sigma(i)}, x\zeta_{\sigma(i)} \right) \right|
\leq \frac{ {b_{3}'}^{\sum k_{i}}{b_{3}''}^{\sum m_{i}}}{\prod k_{i}!\, \prod m_{i}!}
\left( \rho \eta_{0} \right)^{ \sum k_{i}} \left( \rho \zeta_{0} \right)^{ \sum m_{i}}
\left| \alpha_{1}^{\sum \lambda_{i}\eta_{\sigma(i)}x} \right|
\cdot \left| \alpha_{2}^{\sum \lambda_{i}\zeta_{\sigma(i)}x} \right|,
\]
since $\left| \eta_{j} \right| \leq \eta_{0}$ and $\left| \zeta_{j} \right| \leq \zeta_{0}$.

Note that all the sums and products on right-hand side are for $i=1,\ldots,N$.
This will also be the case for all sums and products that follow which have $i$
as the index, but without explicit lower and upper bounds on $i$.

Since $\left| \exp(z) \right| \leq \exp \left( \left| z \right| \right)$, it follows that
\begin{align}
\label{eq:Psi-ub}
\max_{\left| x \right| =\rho} \left|\Psi_{\cI} (x) \right|
\leq & N! \frac{{b_{3}'}^{\sum k_{i}}{b_{3}''}^{\sum m_{i}}}{\prod k_{i}!\, \prod m_{i}!}
\left( \rho \eta_{0} \right)^{\sum k_{i}} \left( \rho \zeta_{0} \right)^{\sum m_{i}} \\
& \times \max_{\sigma\in \fS (\cN)} \exp
\left\{ 
\rho \left(
\left| \sum \lambda_{i} \eta_{\sigma(i)} \right| \left| \log \alpha_{1} \right|
+ \left| \sum \lambda_{i} \zeta_{\sigma(i)} \right| \left| \log \alpha_{2} \right|
\right)
\right\}. \nonumber
\end{align}

Using the relation $\sum_{i=1}^{N} \lambda_{i}=0$, we get
\begin{align*}
\sum_{i=1}^{N} \lambda_{i} \eta_{\sigma(i)}
& = \sum_{i=1}^{N} \lambda_{i} \left( r_{\sigma(i)}+t_{\sigma(i)}\beta_{1} \right) \\
& = \sum_{i=1}^{N}\left( \ell_{i} - \frac{L-1}{2} \right)r_{\sigma(i)}
    + \beta_{1} \sum_{i=1}^{N}\left( \ell_{i} - \frac{L-1}{2} \right)t_{\sigma(i)} \\
& = \sum_{i=1}^{N} \ell_{i}r_{\sigma(i)} - M_{1}
    + \beta_{1} \sum_{i=1}^{N}\ell_{i}t_{\sigma(i)}- \beta_{1}M_{3}.
\end{align*}

Thus, from Lemma~\ref{lem:lemG},
\[
\left| \sum_{i=1}^{N} \lambda_{i} \eta_{\sigma(i)} \right| \leq G_{1}+\beta_{1} G_{3}.
\]

In a similar way,
\[
\left| \sum_{i=1}^{N} \lambda_{i} \zeta_{\sigma(i)} \right| \leq G_{2}+\beta_{2} G_{3}.
\]

Recalling that $b_{3} \left| \log \alpha_{3} \right|
=b_{1} \left| \log \alpha_{1} \right| + b_{2} \left| \log \alpha_{2} \right| \pm \left| \Lambda \right|$,
it follows that
\begin{align}
\label{eq:exp-sum-ub}
& \exp
\left\{ 
\rho 
\left(
   \left| \sum \lambda_{i} \eta_{\sigma(i)} \right| \left| \log \alpha_{1} \right|
   + \left| \sum \lambda_{i} \zeta_{\sigma(i)} \right| \left| \log \alpha_{2} \right|
\right)
\right\} \\
\leq &
\exp
\left\{ 
\rho
\left(
\left( G_{1}+\beta_{1} G_{3} \right) \left| \log \alpha_{1} \right|
+ \left( G_{2}+\beta_{2} G_{3} \right) \left| \log \alpha_{2} \right|
\right)
\right\} \nonumber \\
\leq &
\exp
\left\{ 
\rho \left(
   G_{1} \left| \log \alpha_{1} \right|
  +G_{2} \left| \log \alpha_{2} \right|
  +G_{3} \left( \left| \log \alpha_{3} \right| + \frac{\left| \Lambda \right|}{b_{3}} \right)
\right)
\right\}. \nonumber
\end{align}

Recalling \eqref{eq:star} and applying the definitions for the quantities that
arise, we have
\begin{align*}
\rho G_{3} \frac{\left| \Lambda \right|}{b_{3}}
&= \rho g \frac{NLT}{b_{3}} \frac{\left| \Lambda \right|}{2}
= \frac{\rho gK(K+1) L}{2} \frac{\Lambda'}{e^{LT\left| \Lambda \right|/(2b_{3})}}
\leq \frac{\rho gK(K+1) L \Lambda'}{2}
\leq \frac{\rho K(K+1) L \Lambda'}{8} \\
&< \frac{\rho K(K+1) L}{8\rho^{KL}}.
\end{align*}

By looking at the partial derivatives of this last expression with respect to
$\rho$, $K$ and $L$, we see that it is a non-increasing function in each of these
provided that $KL\log(\rho) \geq 2$ and $KL \geq 1$. These conditions hold for
$K \geq 3$, $L \geq 5$ and $\rho \geq 2$. For $K=3$, $L=5$ and $\rho=2$, we find
that $\rho K^{2} L/ \left( 4\rho^{KL} \right)<0.0005$. Hence
\begin{equation}
\label{eq:rhoG2-ub}
\rho G_{3} \frac{\left| \Lambda \right|}{b_{3}}<0.001.
\end{equation}

Combining, \eqref{eq:delta-sum}, \eqref{eq:Psi1-ub}, \eqref{eq:Psi-ub},
\eqref{eq:exp-sum-ub} and \eqref{eq:rhoG2-ub}, we find that
condition~\eqref{eq:star} implies the upper bound
\begin{align*}
\log \left| \Delta \right|
< & \sum_{i=1}^{3} M_{i} \log \left| \alpha_{i} \right| + \rho \sum_{i=1}^{3} G_{i} \left| \log \alpha_{i} \right|
    +\log(N!)+N\log(2)
    + \log(\rho) \sum_{i} \left( k_{i}+m_{i} \right) \\
  & + \log \left( \frac{\left( b_{3}'\eta_{0} \right)^{\sum k_{i}}}{\prod k_{i}!}
    \frac{\left( b_{3}''\zeta_{0} \right)^{\sum m_{i}}}{\prod m_{i}!}
    \max_{\cI \subseteq \cN} \frac{|\Lambda'|^{N-|\cI|}}{\rho^{J_{\cI}}} \right) + 0.001. \nonumber
\end{align*}

Under condition~\eqref{eq:star}, we have
\begin{equation}
\label{eq:lambdaP-UB}
\frac{|\Lambda'|^{N-|\cI|}}{\rho^{J_{\cI}}}
\leq \rho^{-KL(N-|\cI|)-J_{\cI}}
\end{equation}
(note that if $N=|\cI|$, then we need $\leq$ here, rather than the $<$ in
\eqref{eq:star}).

From Lemma~\ref{lem:7.2}, we obtain $J_{\cI} \geq \Theta \left( K-1, \left| \cI \right| \right)$.
Note that our matrix is not of exactly the same form as used in Lemma~\ref{lem:7.2},
as we have functions in the entries, $\Psi_{\cI}(x)_{i,j}$ when $i \not\in \cI$,
rather than complex numbers.
But since the $\phi_{i}$'s are the product of polynomials and analytic functions
we can write them as power series (some possibly truncated). Since $\Psi_{\cI}(x)$
is a determinant, it is multilinear,  these entries cannot
reduce $J_{\cI}$ (see the proof of Lemma~7.2 of \cite{W} for more details).

So applying equation~\eqref{eq:lambdaP-UB}, Lemma~\ref{lem:rho-exp-UB}
and using the relations
\begin{align*}
\sum_{i=1}^{N}\left( k_{i}+m_{i} \right)
&= L\sum_{k=0}^{K-1} \left( \sum_{m=0}^{K-1-k} k+m \right)
= L\sum_{k=0}^{K-1} \frac{(K-1+k)(K-k)}{2}
=\frac{KL(K+1)(K-1)}{3} \\
&= \frac{2N(K-1)}{3},
\end{align*}
we obtain
\begin{align*}
& \log(\rho) \sum_{i} \left( k_{i}+m_{i} \right)
+ \log \left( \max_{\cI \subseteq \cN} \frac{|\Lambda'|^{N-|\cI|}}{\rho^{J_{\cI}}} \right) \\
& \leq \log(\rho) \left( \frac{2N(K-1)}{3} -KL(N-|\cI|) - J_{\cI} \right) \\
& \leq \log(\rho) \left( \frac{2N(K-1)}{3} -KL(N-|\cI|) - \Theta \left( K-1, \left| \cI \right| \right) \right) \\
& \leq \log(\rho) \left( \frac{2N(K-1)}{3} - \frac{N^{2}}{2K} \left( 1+\frac{2}{L}-\frac{6}{KL}-\frac{1}{3L^{2}} \right) \right) \\
& = -\log(\rho) \frac{N^{2}}{2K} \left( 1-\frac{2}{3L}-\frac{2}{3KL}-\frac{1}{3L^{2}}-\frac{16}{3K^{2}L} \right).
\end{align*}

Also note that
\begin{equation}
\label{eq:ki-sum}
\sum_{i=1}^{N} k_{i} = L\sum_{k=0}^{K-1} (K-k)k=\frac{K(K-1)(K+1)L}{6}=\frac{N(K-1)}{3}.
\end{equation}

So using the definition of $b$ in \eqref{eq:b-defn}, we see that
\begin{align*}
b^{N(K-1)/3}
&= \left( b_{3}'\eta_{0} \right)^{N(K-1)/3} \left( b_{3}''\zeta_{0} \right)^{N(K-1)/3}
\left( \prod_{k=1}^{K-1}(k!)^{K-k} \right)^{-2L} \\
&= \frac{\left( b_{3}'\eta_{0} \right)^{\sum k_{i}}}
{\prod k_{i}!}
\frac{\left( b_{3}''\zeta_{0} \right)^{\sum m_{i}}}
{\prod m_{i}!}.
\end{align*}

This completes the proof of the proposition.
\end{proof}

\subsection{A lower bound for $\left| \Delta \right|$}

Liouville's inequality is the key tool that we need to obtain a lower bound for
$\left| \Delta \right|$. The version of Liouville inequality that we use is the same as in
\cite{LMN} (p.~298--299) (also see Exercises~3.3(a) and 3.5 on pages~106--107 of \cite{W}).

\begin{lem}
\label{lem:Liou}
Let $\alpha_{1}$, $\alpha_{2}$ and $\alpha_{3}$ be non-zero algebraic numbers and a polynomial
$f\in \bbZ \left[ X_{1},X_{2},X_{3} \right]$
such that $f \left( \alpha_{1},\alpha_{2},\alpha_{3} \right) \neq 0$, then
\[
\left| f \left( \alpha_{1},\alpha_{2},\alpha_{3} \right) \right|
\geq |f|^{-{\cD+1}} \left( \alpha_{1}^{*} \right)^{d_{1}} \left( \alpha_{2}^{*} \right)^{d_{2}} \left( \alpha_{3}^{*} \right)^{d_{3}}
\times \exp \left\{-\cD\left( d_{1} \h \left( \alpha_{1} \right) + d_{2} \h \left( \alpha_{2} \right) + d_{3} \h \left( \alpha_{3} \right) \right) \right\},
\]
where 
$\cD=\left[ \bbQ \left( \alpha_{1},\alpha_{2},\alpha_{3} \right) : \bbQ \right]\bigm/
\left[ \bbR \left( \alpha_{1},\alpha_{2},\alpha_{3} \right) : \bbR \right]$,
\[
d_{i}=\deg_{X_{i}}f, \ \ i=1,2,3,
\qquad \left| f \right|=\max \left\{ \left| f \left( z_{1},z_{2},z_{3} \right) \right| : \left| z_{i} \right| \leq 1, \ i=1,2,3 \right\},
\]
and $\h(\alpha)$ is the absolute logarithmic height of the algebraic number $\alpha$,
and $\alpha^{*}=\max\{ 1, \left| \alpha \right| \}$.
\end{lem}


Using Lemma~\ref{lem:Liou}, we get the following lemma -- also see
Proposition~12.6 in \cite{BMS1}.

\begin{prop}
\label{prop:M2}
If $\Delta \neq 0$, then
\begin{align*}
\log \left| \Delta \right|
\geq & - \frac{\cD-1}{2}N\log(N) +\sum_{i=1}^{3} \left( M_{i} +G_{i} \right)\log \left| \alpha_{i} \right|
       -2\cD \sum_{i=1}^{3} G_{i} \h \left( \alpha_{i} \right) \\
     & - \frac{\cD-1}{3}(K-1)N\log(b).
\end{align*}
\end{prop}

\begin{proof}
From \eqref{eq:delta-defn}, we have $\Delta=P\left( \alpha_{1},\alpha_{2},\alpha_{3} \right)$
where $P\in \bbZ \left[ X_{1},X_{2},X_{3} \right]$ is given by
\[
P \left( X_{1},X_{2},X_{3} \right) =
\sum_{\sigma\in \fS_{N}} {\rm sg}(\sigma)
\left( \prod_{i=1}^{N} \binom{r_{\sigma(i)} b_{3}'+t_{\sigma(i)} b_{1}'}{k_{i}}
\binom{s_{\sigma(i)} b_{3}''+t_{\sigma(i)} b_{2}''}{m_{i}} \right)
X_{1}^{n_{r,\sigma}}X_{2}^{n_{s,\sigma}}X_{3}^{n_{t,\sigma}},
\]
where ${\rm sg}(\sigma)$ is the signature of the permutation, $\sigma$,
\[
n_{r,\sigma}=\sum_{i=1}^{N} \ell_{i} r_{\sigma(i)}, \quad
n_{s,\sigma}=\sum_{i=1}^{N} \ell_{i} s_{\sigma(i)} \quad \text{ and } \quad
n_{t,\sigma}=\sum_{i=1}^{N} \ell_{i} t_{\sigma(i)}.
\]

By Lemma~\ref{lem:lemG},
\[
\left| \deg_{X_{i}} P - M_{i} \right| \leq G_{i}, \quad \text{for $i=1,2,3$.}
\]

Let
\[
V_{i} = \lfloor M_{i}+G_{i}\rfloor, \qquad U_{i} = \lceil M_{i}-G_{i}\rceil,
\quad \text{$i=1, 2, 3$,}
\]
then
\[
\Delta = \alpha_{1}^{V_{1}} \alpha_{2}^{V_{2}} \alpha_{3}^{V_{3}}
\widetilde{P} \left( \alpha_{1}^{-1},\alpha_{2}^{-1},\alpha_{3}^{-1} \right),
\]
where
\[
\deg_{X_{i}} \widetilde{P} \leq V_{i}-U_{i}, \quad \text{$i=1, 2, 3$.}
\]

By our Liouville estimate
\[
\log \left| \widetilde{P} \left( \alpha_{1}^{-1}, \alpha_{2}^{-1}, \alpha_{3}^{-1} \right) \right|
\geq -(\cD-1)\log \left| \widetilde{P} \right| - \cD\sum_{i=1}^{3} \left( V_{i}-U_{i} \right) \h \left( \alpha_{i} \right),
\]
recalling from our assumptions at the start of Section~\ref{sect:result} that
$\left| \alpha_{i} \right| \geq 1$, and hence $\left( \alpha_{i}^{-1} \right)^{*}=1$, for $i=1,2,3$.

Now we have to find an upper bound for $\left| \widetilde{P} \right|$ (or for $|P|$,
which is equal to $\left| \widetilde{P} \right|$). By the multilinearity of the
determinant, for all $\eta$, $\zeta \in \bbC$,
\[
P \left( z_{1},z_{2},z_{3} \right)
= \det \left(
\frac{\left( r_{j} b_{3}'+t_{j}b_{1}'-\eta \right)^{k_{i}}}{k_{i}!}
\frac{\left( s_{j} b_{3}''+t_{j}b_{2}''-\zeta \right)^{m_{i}}}{m_{i}!}
z_{1}^{\ell_{i} r_{j}} z_{2}^{\ell_{i} s_{j}} z_{3}^{\ell_{i} t_{j}}
\right).
\]

Choose
\[
\eta = \frac{(R-1)b_{3}'+(T-1)b_{1}'}{2}, \quad
\zeta = \frac{(S-1)b_{3}''+(T-1)b_{2}''}{2}.
\]

Notice that, for $1 \leq j \leq N$,
\[
\left| r_{j} b_{3}'+t_{j}b_{1}'-\eta \right|^{k_{i}}
\leq \left( \frac{(R-1)b_{3} +(T-1)b_{1}}{2d_{1}} \right)^{k_{i}}\!\!,
\quad
\left| s_{j} b_{3}''+t_{j}b_{2}''-\zeta \right|^{k_{i}}
\leq \left( \frac{(S-1)b_{3}+(T-1)b_{2}}{2d_{2}} \right)^{m_{i}}
\]
and recall from \eqref{eq:ki-sum} that
\[
\sum_{i=1}^{N}k_{i} = \sum_{i=1}^{N}m_{i} = \frac{N(K-1)}{3}.
\]

So Hadamard's inequality implies
\begin{align*}
|P| \leq
& N^{N/2} \left( \frac{(R-1)b_{3}+(T-1)b_{1}}{2d_{1}} \right)^{(K-1)N/3} 
\left( \frac{(S-1)b_{3}+(T-1)b_{2}}{2d_{2}} \right)^{(K-1)N/3} \\
& \times \left( \prod_{i=1}^{N} k_{i}! \right)^{-1}
         \left( \prod_{i=1}^{N} m_{i}! \right)^{-1}.
\end{align*} 

Recalling the definition of $b$, we get
\[
|P| \leq N^{N/2}b^{(K-1)N/3}.
\]

Collecting all the above estimates, we find
\[
\log \left| \Delta \right|
\geq -(\cD-1)\left( \log \left( N^{N/2} \right) + \frac{(K-1)N}{3}\log b \right)
 - \cD \sum_{i=1}^{3} \left( V_{i}-U_{i} \right) \h \left( \alpha_{i} \right)
 +\sum_{i=1}^{3} V_{i}\log \left| \alpha_{i} \right|.
\]

The inequalities $\cD\h \left( \alpha_{i} \right) \geq \log \left| \alpha_{i} \right| \geq 0$
imply
\[
V_{i}\log \left| \alpha_{i} \right| - \cD \left( V_{i}-U_{i} \right) \h \left( \alpha_{i} \right)
\geq \left( M_{i}+G_{i} \right)\log \left| \alpha_{i} \right| - 2\cD G_{i} \h \left( \alpha_{i} \right)
\]
and the result follows.
\end{proof}

\subsection{Synthesis}
Here we combine the upper and lower bounds for $\left| \Delta \right|$ that we obtained in
the two previous subsections.

\begin{prop}
\label{prop:M3}
With the previous notation, if $K \geq 3$, $L \geq 5$, $\rho \geq 2$, and if
$\Delta \neq 0$ then
\[
\Lambda' \geq \rho^{-KL}
\]
provided that
\[
\left( \frac{KL}{2} + \frac{L}{2} - 0.37K - 2 \right) \log \rho
\geq (\cD+1)\log N + gL \left( a_{1}R+a_{2}S+a_{3}T \right) + \frac{2\cD (K-1)}{3}\log b,
\]
where the $a_{i}$ are positive real numbers which satisfy
\[
a_{i} \geq \rho \left| \log \alpha_{i} \right| - \log \left| \alpha_{i} \right|
+ 2\cD \h \left( \alpha_{i} \right)
\qquad \text{for $i=1,2,3$.}
\]
\end{prop}


\begin{proof}
Under the hypotheses of the Propositions~\ref{prop:M1} and \ref{prop:M2} (which
include the hypothesis that $\Lambda'<\rho^{-KL}$ from \eqref{eq:star}),
we get
\begin{align*} 
& - \frac{\cD-1}{2}N\log(N) +\sum_{i=1}^{3} \left( M_{i} +G_{i} \right)\log \left| \alpha_{i} \right|
  -2\cD \sum_{i=1}^{3} G_{i} \h \left( \alpha_{i} \right)
  -\frac{\cD-1}{3}(K-1)N\log(b) \\
< & \sum_{i=1}^{3} M_{i} \log \left| \alpha_{i} \right|
    + \rho \sum_{i=1}^{3} G_{i} \left| \log \alpha_{i} \right|
    + \log(N!)+N\log 2 + \frac{N}{3}(K-1) \log b \\
  & -\frac{N^{2}}{2K} \left( 1-\frac{2}{3L}-\frac{2}{3KL} - \frac{1}{3L^{2}}-\frac{16}{3K^{2}L} \right)
    \log \rho + 0.001.
\end{align*}

After combining like terms, we obtain
\begin{align*}
& \frac{N^{2}}{2K} \left( 1-\frac{2}{3L}-\frac{2}{3KL} - \frac{1}{3L^{2}}-\frac{16}{3K^{2}L} \right) \log \rho \\
< & \frac{\cD-1}{2} N\log N + \sum_{i=1}^{3} G_{i} \left( \rho \left| \log \alpha_{i} \right|
    - \log \left| \alpha_{i} \right| + 2\cD \h \left( \alpha_{i} \right) \right)
    + \log(N!) + N\log (2) + \frac{K-1}{3} \cD N\log(b) + 0.001.
\end{align*}

Applying $N!<N(N/e)^{N}$ (which holds for $N \geq 7$), then dividing both sides
by $N/2$, it follows that
\begin{align*}
  & \left( \frac{KL}{2} + \frac{L}{2} - \left( \frac{1}{3}+\frac{1}{6L} \right)K - \frac{2}{3}-\frac{3}{K}-\frac{1}{6L}-\frac{8}{3K^{2}} \right) \log \rho \\
< & (\cD+1)\log N + (2/N)\sum_{i=1}^{3} G_{i} \left( \rho \left| \log \alpha_{i} \right|
    - \log \left| \alpha_{i} \right| + 2\cD \h \left( \alpha_{i} \right) \right)
    + \frac{2\log(N)}{N} - 2\log (e/2) \\
  & + \frac{2(K-1)\cD}{3} \log(b) + 0.002/N.
\end{align*}

For $K \geq 3$ and $L \geq 5$, we have $1/3+1/(6L)=0.366\ldots$ and
$2/3+3/K+1/(6L)+8/\left( 3K^{2} \right)=1.9962\ldots$, we have
\begin{align*}
\left( \frac{KL}{2} + \frac{L}{2} - 0.37K - 2 \right) \log \rho
< & (\cD+1)\log N + (2/N)\sum_{i=1}^{3} G_{i} \left( \rho \left| \log \alpha_{i} \right|
    - \log \left| \alpha_{i} \right| + 2\cD \h \left( \alpha_{i} \right) \right) \\
  & + \frac{2\log(N)}{N} - 2\log (e/2) + \frac{2(K-1)\cD}{3} \log(b) + 0.002/N.
\end{align*}

The proof now follows from $2\log(N)/N-2\log(e/2)+0.002/N<0$ for $N \geq 6$
and the definitions of the $G_{i}$'s in \eqref{eq:g-defn} and applying the contrapositive
to show that the assumption that $\Lambda'<p^{-KL}$ does not hold.
\end{proof}


\subsection{A zero lemma}

To use Proposition~\ref{prop:M2}, we need to find conditions under which our determinant
$\Delta$ is non-zero, a so-called {\it zero lemma}. We use a
zero lemma due to N. Gouillon (see \cite[Th\'{e}or\`{e}me~2.1]{G2}, which is a
refinement of Th\'{e}or\`{e}me~1 of \cite{G1}. In fact, in our formulation below,
we state Gouillon's result not just for $\bbC$, as he does, but for any algebraically
closed field of characteristic zero -- there are no changes required to his proof.
Also Gouillon's result applies to multiplicities. We ignore multiplicities of
the zeroes here.

Let $\bbK$ be an algebraically closed field of characteristic zero and let $d_{0}$
and $d_{1}$ be two non-negative integers which are not both zero.
We denote by $G$ the group $\bbK^{d_{0}} \times \left( \bbK^{\times} \right)^{d_{1}}$ 
The group law on $G$ will be written additively, hence its neutral element
is denoted by $\bzero_{G}$. When $\Sigma_{1},\ldots,\Sigma_{n}$ are finite subsets
of $G$, we define
\[
\Sigma_{1}+\cdots+\Sigma_{n}
= \left\{ \sigma_{1} + \cdots + \sigma_{n}: 
\sigma_{1} \in \Sigma_{1}, \ldots, \sigma_{n} \in \Sigma_{n} \right\}.
\]

\begin{prop}
\label{prop:zero-est}
Suppose that $K$ and $L$ are positive integers and that $\Sigma_{1}$, $\Sigma_{2}$
and $\Sigma_{3}$ are non-empty finite subsets of $\bbK^{2}\times \bbK^{\times}$
such that
\begin{equation}
\label{eq:zero-est-i}
\begin{cases}
\card \left\{ \lambda x_{1} + \mu x_{2} :
      \text{$\exists y \in \bbK^{\times}$ with $\left( x_{1},x_{2},y \right) \in \Sigma_{1}$} \right\} &>K,
      \quad \forall (\lambda, \mu ) \in \bbK^{2} \setminus \{(0,0)\},\\
\card \left\{ y :
      \text{$\exists \left( x_{1}, x_{2} \right) \in \bbK^{2}$ with $\left( x_{1},x_{2},y \right) \in \Sigma_{1}$} \right\} &>L,
\end{cases}
\end{equation}

\begin{equation}
\label{eq:zero-est-ii}
\begin{cases}
\card \left\{ \left( \lambda x_{1} + \mu x_{2},y \right) :
      \left( x_{1},x_{2},y \right) \in \Sigma_{2} \right\} &>2KL, \quad \forall (\lambda, \mu) \in \bbK^{2} \setminus \{(0,0)\}, \\
\card \left\{ \left( x_{1},x_{2} \right) :
      \text{$\exists y \in \bbK^{\times}$ with $\left( x_{1},x_{2},y \right) \in \Sigma_{2}$} \right\} &>K^{2},
\end{cases}
\end{equation}
and
\begin{equation}
\label{eq:zero-est-iii}
\card \left( \Sigma_{3} \right) > 3K^{2}L.
\end{equation}

Then the only polynomial $P \in \bbK \left[ X_{1}, X_{2},Y \right]$ of total
degree at most $K$ in $X_{1}$ and $X_{2}$ and of degree at most $L$ in $Y$ which
is zero on the set $\Sigma_{1}+\Sigma_{2}+\Sigma_{3}$ is the zero polynomial.
\end{prop}

The proof of this proposition is based on the following generalisation of a
special case of a result due to Gouillon.

\begin{lem}
\label{lem:zero-est-gouillon}
Let $K$ and $L$ be positive integers, $\bbK$ be an algebraically closed
field of characteristic zero and $\Sigma_{1}, \Sigma_{2}, \Sigma_{3}$ be non-empty
finite subsets of $\bbK^{2} \times \bbK^{\times}$.

Suppose that the following conditions are satisfied.

(1) For $j=1$ and $j=2$ and for all $\bbK$-subspaces, $W$, of $\bbK^{2}$ of
dimension at most $2-j$, we have
\[
\card \left( \frac{\Sigma_{j}+\left( W \times \bbK^{\times} \right)}{W \times \bbK^{\times}} \right)
> K^{j}.
\]

(2) For each of $j=1$, $j=2$ and $j=3$ and for all $\bbK$-subspaces, $W$, of $\bbK^{2}$ of
dimension at most $3-j$, we have
\[
\card \left( \frac{\Sigma_{j}+\left( W \times \{ 1 \} \right)}{W \times \{ 1 \}} \right)
> jK^{j-1}L.
\]

Then the only polynomial $P \in \bbK \left[ X_{1}, X_{2},Y \right]$ of total
degree at most $K$ in $X_{1}$ and $X_{2}$ and of degree at most $L$ in $Y$ which
is zero on the set $\Sigma_{1}+\Sigma_{2}+\Sigma_{3}$ is the zero polynomial.
\end{lem}

\begin{proof}
This is based on Th\'{e}or\`{e}me~2.1 in \cite{G2} in the special case of $m=2$ and
$T_{1}=\cdots=T_{m+1}=0$. The only difference is that he stated and proved his
result only for $\bbC$ in place of our $\bbK$. However, his proof only requires
that the field be algebraically closed and of characteristic $0$, rather than
requiring any additional properties of $\bbC$.

We have taken his $T_{1}=\cdots=T_{m+1}=0$ since we are only concerned
with the zeroes themselves, not their multiplicities. Also we have used the
notation of Waldschmidt \cite{W}, which is itself based on the notation of
Philippon \cite{Phil}, instead of Gouillon's similar, but not identical, notation.
\end{proof}

\begin{proof}[Proof of Proposition~\ref{prop:zero-est}]
We only show that case $j=1$ of Gouillon's condition~(1) follows from the conditions
in our proposition (the first part of condition~\eqref{eq:zero-est-i} of our proposition,
in particular), as the proofs of the others are very similar.

In this case, there exists a $\bbK$-subspace, $W$, of $\bbK^{2}$ of dimension
either $0$ or $1$.

If the dimension of $W$ is $0$, then
\begin{equation}
\label{eq:zero-est-condition}
\card \left( \frac{\Sigma_{1}+\left( W \times \bbK^{\times} \right)}{W \times \bbK^{\times}} \right)
=\card \left\{ \left( x_{1}, x_{2} \right) :
\text{$\exists y \in \bbK^{\times}$ with $\left( x_{1},x_{2},y \right) \in \Sigma_{1}$} \right\}.
\end{equation}

This is because
$\left( x_{1}, x_{2}, y \right) + \left( \{ (0,0) \} \times \bbK^{\times} \right)
=\left( x_{1}, x_{2}, 1 \right) + \left( \{ (0,0) \} \times \bbK^{\times} \right)$ for any
$\left( x_{1}, x_{2}, y \right) \in \Sigma_{1}$ and each coset,
$\left( x_{1}, x_{2}, 1 \right) + \left( \{ (0,0) \} \times \bbK^{\times} \right)$, is distinct.

The first part of our condition~\eqref{eq:zero-est-i} implies that the cardinality in \eqref{eq:zero-est-condition}
exceeds $K$.

If the dimension of $W$ is $1$, then this subspace
is
\[
\left\{ \left( x_{1}, x_{2} \right) \in \bbK^{2}: \lambda x_{1}+\mu x_{2} = 0 \right\}
\]
for some $(\lambda, \mu) \in \bbK^{2} \backslash \{ (0,0) \}$.

For any $(\lambda, \mu) \in \bbK^{2} \backslash \{ (0,0) \}$, there is a bijection
between this set and the set in the first part of condition~\eqref{eq:zero-est-i} of our
proposition (note that all $\left( x_{1}, x_{2}, y \right) \in \Sigma_{1}$ with
$x_{1}$ and $x_{2}$ fixed map to the same element in the set in Gouillon's
condition~(1) with $j=1$). So the first part of condition~\eqref{eq:zero-est-i} of our
proposition ensures that Gouillon's condition~(1) holds for $j=1$.

Continuing in a very similar way, we can show that the conditions in our proposition
imply that Gouillon's conditions hold. Hence our conclusion follows from his result.
\end{proof}

\begin{rem-nonum}
Equation~\eqref{eq:zero-est-condition} illustrates how the sets on the left-hand
sides of \eqref{eq:zero-est-i}--\eqref{eq:zero-est-iii} in Proposition~\ref{prop:zero-est}
arise. They are related to sets of classes of the form
$\left( \Sigma_{i}+H \right)/H$ for various algebraic subgroups, $H$, of $\bbK^{2} \times \bbK^{\times}$.
Such algebraic subgroups, $H$, are the obstruction subgroups introduced to the
study of zero estimates and multiplicity estimates by Philippon \cite{Phil}.

Also note that any algebraic subgroup of the product of an additive group by a
multiplicative group is a product of a subgroup of the additive group and
a subgroup of the multiplicative group.
\end{rem-nonum}

\begin{rem}
\label{rem:sigma-defn}
For $j=1$, $2$, $3$, we shall consider finite sets $\Sigma_{j}$ defined by
\begin{equation}
\label{eq:Sigma-j-defn}
\Sigma_{j} = \left\{ \left( r+t\beta_{1}, s+t\beta_{2}, \alpha_{1}^{r} \alpha_{2}^{s} \alpha_{3}^{t} \right)
    : 0 \leq r \leq R_{j}, 0 \leq s \leq S_{j}, 0 \leq t \leq T_{j} \right\},
\end{equation}
where $R_{j}$, $S_{j}$ and $T_{j}$ are positive integers, $\beta_{1}=b_{1}/b_{3}=b_{1}'/b_{3}'$
and $\beta_{2}=b_{2}/b_{3}=b_{2}''/b_{3}''$ are as in \eqref{eq:bi-rels}.
This choice corresponds to the entries of the arithmetical matrix used in the
definition of $\Delta$ in \eqref{eq:delta-defn}.
\end{rem}

\subsection{Degeneracies}

If the conditions in our zero lemma do not all hold, then there will be a linear
dependence relation over $\bbQ$ that the $b_{i}$'s in our linear form satisfy
(see conditions~\eqref{eq:C1} and ({\rm C2}) in Theorem~\ref{thm:main}).
We refer to such cases as degeneracies and present results in this subsection for
how we handle them.

\begin{rem}
\label{rem:wald}
Note that there is an alternative approach due to Waldschmidt for handling the
degenerate case
(see the discussion at the end of Section~7.1 of \cite[pp. 191--192]{W}). This
alternative approach is more efficient in its dependence on $b$ ($\log^{2}(b)$
rather than $\log^{8/3}(b)$ as in Subsection~\ref{subsect:degen}). This would
considerably simplify our treatment of the degenerate case as well as the statement
of Theorem~\ref{thm:main}. Our attempts to apply it have yielded larger constants,
and hence weaker results.
But Waldschmidt's approach certainly warrants further efforts.
\end{rem}

Concerning the group, $\bbC^{2} \times \bbC^{\times}$, the following elementary lemma is important.

\begin{lem}
\label{lem:condM1}
%
%
The following conditions are equivalent.

\noindent
{\rm (a)} The map
\[
\psi : \bbZ^{3} \to \bbC^{2} \times \bbC^{\times}, \quad (r,s,t)
\mapsto \left( r+\beta_{1}t, s+\beta_{3}t, \alpha_{1}^{r}\alpha_{2}^{s}\alpha_{3}^{t} \right)
\]
is not one-to-one $($not injective$)$.

\noindent
{\rm (b)} There exists some positive integer $m$ such that
\[
\alpha_{3}^{mb_{3}} = \alpha_{1}^{mb_{1}} \alpha_{2}^{mb_{2}}.
\]

\noindent
{\rm (c)} The number $\Lambda=b_{1}\log \alpha_{1}+b_{2}\log \alpha_{2}-b_{3}\log \alpha_{3}$
belongs to the set $i\pi \bbQ$.
\end{lem}

\begin{proof}
Clearly, without loss of generality, we may assume that $\gcd \left( b_{1}, b_{2}, b_{3} \right)=1$.

Recall our notation from \eqref{eq:bi-rels} with $d_{1}=\gcd \left( b_{1}, b_{3} \right)$
and $d_{2}=\gcd \left( b_{2}, b_{3} \right)$.
Since $\gcd \left( b_{1}, b_{2}, b_{3} \right)=1$, we have $\gcd \left( d_{1}, d_{2} \right)=1$.
Thus
\[
b_{3}=d_{1} d_{2} \widetilde{b_{3}} \ \text{(say)}, \quad b_{3}'= d_{2} \widetilde{b_{3}},
\quad b_{3}''=d_{1} \widetilde{b_{3}}.
\]

After these preliminaries, we prove the implication $(a) \Rightarrow (b)$.
Suppose that the map $\psi$ is not injective. Then there exist rational integers
$r$, $s$, $t$, not all zero, such that
\[
\psi(r,s,t)=(0,0,1).
\]

That is,
\[
r + t \beta_{1} = 0, \quad s + t \beta_{2} = 0,\quad \alpha_{1}^{r} \alpha_{2}^{s} \alpha_{3}^{t}=1.
\]

The first relation implies $r=-kb_{1}'$ for some rational integer, $k$. In fact,
we have $k=t/b_{3}'$. Thus $t=kb_{3}'=kd_{2}\widetilde{b_{3}}$. Similarly, from
the second relation we have $s=-\ell b_{2}''$, where $\ell=t/b_{3}''$, so
$t=\ell b_{3}''=\ell d_{1}\widetilde{b_{3}}$, for some rational integer $\ell$.
In particular, $kd_{2}=\ell d_{1}$, hence there exists $m\in \bbZ$ such that
$k=md_{1}$ and $\ell=md_{2}$. Thus
\[
r = - m b_{1}, \quad s = -m b_{2} \quad \text{and} \quad t = m b_{3}.
\]

Since at least one of $r$, $s$ and $t$ is non-zero, it follows that $m \neq 0$.
Thus the third relation gives
\[
\alpha_{3}^{mb_{3}} = \alpha_{1}^{mb_{1}} \alpha_{3}^{m b_{3}},
\]
as wanted.

Clearly, $(b)$ implies $(c)$.

To show that $(c)$ implies $(a)$, we suppose that $(c)$ holds, \ie. that 
$m\Lambda$ belongs to $2i\pi \bbZ$ for some positive rational integer $m$. Then
it is clear that $\psi \left( mb_{1},mb_{2},-mb_{3} \right)=(0,0,1)$, proving
that the map $\psi$ is not injective.
\end{proof}

\begin{lem}
\label{lem:condM2}
If $\alpha_{1}$, $\alpha_{2}$ and $\alpha_{3}$ are non-zero complex numbers such
that (for example) $\alpha_{1}$ and $\alpha_{2}$ are multiplicatively independent 
and $\alpha_{3} \neq 1$ is a root of unity, and if $\log \alpha_{j}$ is any determination
of the logarithm of $\alpha_{j}$ for $j=1$, $2$, $3$, then the numbers $\log \alpha_{1}$,
$\log \alpha_{2}$ and $\log \alpha_{3}$ are linearly independent over the rationals.

Furthermore, if $b_{1}$, $b_{2}$ and $b_{3}$ are rational integers with at least
one of $b_{1}$ and $b_{2}$ non-zero, then the number 
$b_{1}\log \alpha_{1}+b_{2}\log \alpha_{2}+b_{3}\log \alpha_{3}$ does not belong to the set
$i\pi \bbQ$.
\end{lem} 

\begin{proof}
Suppose that
\[
\Lambda = b_{1}\log \alpha_{1} + b_{2}\log \alpha_{2} - b_{3}\log \alpha_{3}=0
\]
where 
$b_{1}$, $b_{2}$ and $b_{3}$ are rational integers not all equal to zero.
Then $\alpha_{3}^{b_{3}}=\alpha_{1}^{b_{1}} \alpha_{2}^{b_{2}}$.
Assume that $\alpha_{3}^{d}=1$ with $d>1$, then 
$ \alpha_{2}^{db_{2}}=\alpha_{1}^{-db_{1}}$, which implies $b_{1}=b_{2}=0$
since $\alpha_{1}$ and $\alpha_{2}$ are multiplicatively independent.
Since we assumed that $b_{1}$, $b_{2}$ and $b_{3}$ are not all equal to zero,
it follows that $b_{3} \neq 0$ and so $\Lambda = b_{3}\log \alpha_{3}\neq 0$,
since $\alpha_{3} \neq 1$. This contradiction proves the first claim.

Noting that $\log \alpha_{3}=2\pi im/n$ with $n \nmid m$, the second claim
follows from the first one.
\end{proof}

The following very elementary lemma will be useful when investigating
conditions~\eqref{eq:zero-est-i} and \eqref{eq:zero-est-ii} of Proposition~\ref{prop:zero-est}.

\begin{lem}
\label{lem:condI}
Suppose that $b_{1}$, $b_{2}$ and $b_{3}$ are positive rational integers which
are coprime. Let $R$, $S$ and $T$ be positive integers and consider the set
\[
\widetilde{\Sigma} = \left\{ \left( r+tb_{1}/b_{3}, s+tb_{2}/b_{3} \right):
0 \leq r \leq R, \, 0 \leq s \leq S, \, 0 \leq t \leq T \right\}.
\]

Then
\[
\card \widetilde{\Sigma} = (R+1)(S+1)(T+1)
\]
unless
\[
b_{1} \leq R \quad \text{and} \quad b_{2} \leq S \quad \text{and} \quad b_{3} \leq T.
\]
\end{lem}

\begin{proof}
With the same notation as above, suppose that the map 
\[
\psi : \left\{ (r,s,t) : 0 \leq r \leq R, 0 \leq s \leq S, 0 \leq t \leq T \right\}
\to \widetilde{\Sigma},
\quad (r,s,t) \mapsto \left( r+\beta_{1} t, s+\beta_{2} t \right)
\]
is not injective. Then there exist two different triples of
rational integers $(r,s,t)$ and $(r',s',t')$, with $0 \leq r,\,r' \leq R$,
$0 \leq s,s' \leq S$ and $0 \leq t,t' \leq T$
such that $\psi(r,s,t)= \psi(r',s',t')$. That is,
\[
(r-r') + (t-t') \beta_{1} = 0 \quad \text{ and} \quad
(s-s') + (t-t') \beta_{3} = 0.
\]

As in the proof that (a) implies (b) for Lemma~\ref{lem:condM1}, these two
relations imply that
\[
r-r' = m b_{1}, \quad s-s' = m b_{2}, \quad s-s' = -m b_{3}.
\]

Thus $-R \leq mb_{1} \leq R$, $-S \leq mb_{2} \leq S$ and $-T \leq mb_{3} \leq T$.
Since $m$ is non-zero and the $b_{i}$'s are positive, the conclusion follows.
\end{proof}

The first subcondition of condition~\eqref{eq:zero-est-i} in Proposition~\ref{prop:zero-est}
is the most difficult to handle. For it, we will need the following lemmas, in
particular, Lemma~\ref{lem:M3}. These lemmas also bring some extra information
to Proposition~3.1.1 of \cite{Ben-C} (also see \cite[Ex~6.4, pp.~184--185]{W}).

\begin{lem}
\label{lem:M2}
Let $A$, $B$, $C$, $D$, $X>0$, $Y>0$ and $Z>0$ be rational integers with $\gcd(A,B,C)=1$
and $ABC \neq 0$. Put
\[
\Sigma = \left\{ (x,y,z) \in \bbZ^{3} : 0 \leq x \leq X, 0 \leq y \leq Y, 0 \leq z \leq Z \right\}
\]
and
\[
M = \card \left\{ (x,y,z) \in \Sigma : Ax+By+Cz=D \right\}.
\]

\noindent
{\rm (a)} We have
\[
M \leq \left( 1+\left\lfloor \frac{X}{\alpha} \right\rfloor \right)
       \left( 1+\left\lfloor \frac{Y}{|C|/\alpha} \right\rfloor \right)
    \quad \text{and} \quad
M \leq \left( 1+\left\lfloor \frac{X}{\alpha} \right\rfloor \right)
       \left( 1+\left\lfloor \frac{Z}{|B|/\alpha} \right\rfloor \right),
\]
where
\[
\alpha =\gcd(B,C).
\]

\noindent
{\rm (b)} If we suppose that
\[
M \geq \max\left\{X+Y+1,\, Y+Z+1, \, Z+X+1 \right\}
\]
then
\[
\left| A \right| \leq \frac{(Y+1)(Z+1)}{M - \max\{Y,Z\}}, \quad
\left| B \right| \leq \frac{(X+1)(Z+1)}{M - \max\{X,Z\}} \quad \text{and} \quad
\left| C \right| \leq \frac{(X+1)(Y+1)}{M - \max\{X,Y\}}.
\]
\end{lem}

\begin{rem}
\label{rem:cV1}
When we apply part~(b) of this lemma, we will assume that $M$ is (possibly) even
larger. Let $\cV = ((X+1)(Y+1)(Z+1))^{1/2}$ and suppose that $\chi$ is a positive
real number. We will assume that
\[
M \geq \max\left\{X+Y+1,\, Y+Z+1, \, Z+X+1, \chi \cV \right\}.
\]
\end{rem}

\begin{proof}
(a) Define
\[
\Pi=\left\{ (x,y,z) \in \bbC^{3} : Ax+By+Cz=D \right\}.
\]

If the image by the map $(x,y,z)\mapsto Ax+By+Cz$ of a point $(x,y,z)\in \bbZ^{3}$
belongs to the plane $\Pi$, then
\[
A x \equiv D \pmod \alpha,
\]
where $A$ and $\alpha$ are coprime since $\gcd(A,B,C)=1$. This shows that the
number of such $x$ which satisfy $0 \leq x \leq X$ is at most
$1+\left\lfloor X/\alpha \right\rfloor$.

Now let $x$ be fixed, with $0 \leq x \leq X$, and such that the images of two
distinct elements $(x,y,z)$ and $(x,y',z')$ of $\Sigma$ also belong to $\Pi$.
Then
\[
B(y'-y)=C(z-z'),
\]
where we suppose (as we may) that $y$ is minimal (then $y'>y$). Hence there
exists a positive integer $k$ such that
\[
y'-y = k (|C|/\alpha) \quad \text{and} \quad
z-z'= \pm k (|B|/\alpha).
\]

It follows that, for $x$ fixed, the number of $(x,y,z) \in \Sigma$ whose
image belongs to $\Pi$ is at most $1+\lfloor Y/(|C|/\alpha) \rfloor$. Hence
\begin{equation}
\label{eq:M-ub1}
M \leq \left(1+\left\lfloor \frac{X}{\alpha}     \right\rfloor \right)
       \left(1+\left\lfloor \frac{Y}{|C|/\alpha} \right\rfloor \right),
\end{equation}
which proves the first upper bound for $M$ in part~(a) of the lemma.

The proof of the second upper bound for $M$ is the same, except for fixed values
of $x$, we bound the number of possible $z$-coordinates rather than the number
of possible $y$-coordinates.

\vspace*{3.0mm}

(b) We start with the upper bound for $|C|$.

For $\xi \geq 1$, put
\[
f(\xi) = \left( 1+\frac{X}{\xi} \right) \left( 1+\frac{\xi Y}{|C|} \right).
\]

From equation~\eqref{eq:M-ub1}, it follows that
\[
M \leq f(\alpha).
\]

Clearly, $1 \leq \alpha \leq C$. Since $f''(\xi)=2X/\xi^{3}>0$, it follows that
$f(\xi)$ is convex and so
\[
M \leq f(\alpha) \leq \max \left\{ f(1), f(C) \right\}.
\]

If
\[
M \leq f(1)=1 + \frac{XY}{|C|} + X + \frac{Y}{|C|},
\quad \text{ then } \quad
|C| \leq \frac{Y(X+1)}{M -(X+1)}.
\]

If
\[
M \leq f(C)=1 + \frac{XY}{|C|} + \frac{X}{|C|} + Y,
\quad \text{ then } \quad
|C| \leq \frac{X(Y+1)}{M -(Y+1)}.
\]

Suppose finally that
\[
M \geq \max\{X+Y+1,Y+Z+1, Z+X+1\}.
\]

Since $M-X \geq Y+1$, we can write
\begin{align*}
\frac{Y(X+1)}{M-(X+1)}= \frac{XY+Y}{(M-X)(1-1/(M-X))}
& =    \frac{XY+Y}{M-X} \left( 1+\frac{1}{M-X}+\frac{1}{(M-X)^{2}}+\cdots \right) \\
& \leq \frac{XY+Y}{M-X} \left( 1+\frac{1}{Y+1}+\frac{1}{(Y+1)^{2}}+\cdots \right) \\
& =    \frac{XY+Y}{M-X}\frac{Y+1}{Y}=\frac{(X+1)(Y+1)}{M-X}.
\end{align*}

Similarly,
\[
\frac{X(Y+1)}{M-(Y+1)} \leq \frac{(X+1)(Y+1)}{M-Y}.
\]

Thus, we always have
\[
|C| \leq \frac{(X+1)(Y+1)}{M-\max\{X,Y\}}.
\]

The upper bounds for $|A|$ and $|B|$ are proved in the same way.
\end{proof}

\begin{lem}
\label{lem:plan}
Let $B$, $C$, $D$, $X>0$, $Y>0$ and $Z>0$ be rational integers with $\gcd(B,C)=1$
and $BC \neq 0$.

Put
\[
\Sigma = \left\{ (x,y,z) \in \bbZ^{3} : 0 \leq x \leq X, 0 \leq y \leq Y, 0 \leq z \leq Z \right\}
\]
and
\[
M = \card \left\{ (x,y,z) \in \Sigma : By+Cz=D \right\}.
\]

\noindent
{\rm (a)} We have
\[
M \leq (X+1)
\left(1+\left\lfloor \frac{ Y}{|C|} \right\rfloor \right)
\quad \text{and} \quad
M \leq (X+1)
\left(1+\left\lfloor \frac{ Z}{|B|} \right\rfloor \right).
\]

\noindent
{\rm (b)} Moreover, if we suppose that
\[
M \geq \max\{ X+Y+1, X+Z+1\},
\]
then
\[
\left| B \right| \leq \frac{(X+1)(Z+1)}{M-X} \quad \text{and} \quad
\left| C \right| \leq \frac{(X+1)(Y+1)}{M-X}.
\]
\end{lem}

\begin{rem-nonum}
As with Lemma~\ref{lem:M2}(b) and noted in Remark~\ref{rem:cV1}, when we apply
part~(b) of this lemma, we will assume that $M$ is (possibly) even
larger. Let $\cV = ((X+1)(Y+1)(Z+1))^{1/2}$ and suppose that $\chi$ is a positive
real number. We will assume that
\[
M \geq \max\left\{X+Y+1,\, Y+Z+1, \, Z+X+1, \chi \cV \right\}.
\]
\end{rem-nonum}

\begin{proof}
The proof is similar to that of Lemma~\ref{lem:M2}, but simpler.

(a) Define the plane
\[
\Pi = \left\{ (x,y,z) \in \bbC^{3} : By+Cz=D \right\}
\]
and consider the map $(x,y,z)\mapsto By+Cz$ defined on $\bbC^{3}$.

Let $x$ be fixed with $0 \leq x \leq X$ and such that the images of two
distinct points $(x,y,z)$ and $(x,y',z')$ in $\Sigma$ belong to $\Pi$. Then
\[
B(y'-y)=C(z-z'),
\]
where we suppose (as we may) that $y$ is minimal (then $y'>y$). Hence there
exists a positive integer $k$ such that
\[
y'-y = k|C| \quad \text{and} \quad z-z'= \pm k|B|.
\]

Since $y'-y=C(z-z')/B$ and $\gcd(B,C)=1$, it must be the case that
$B|(z-z')$. This is why $k$ is an integer.

It follows that, for $x$ fixed, the number of $(x,y,z) \in \Sigma$ whose
image belongs to $\Pi$ is at most $1+\lfloor Y/|C| \rfloor$. Hence
\begin{equation}
\label{eq:M-ub2}
M \leq (1+X) \left(1+\left\lfloor \frac{Y}{|C|} \right\rfloor \right),
\end{equation}
which proves the first upper bound for $M$ in the lemma.

The proof of the second upper bound for $M$ is the same, except for fixed values
of $x$, we bound the number of possible $z$-coordinates rather than the number
of possible $y$-coordinates.

\vspace*{3.0mm}

(b) We turn now to the upper bounds for $|B|$ and $|C|$, starting with the
upper bound for $|C|$.

From equation~\eqref{eq:M-ub2}, it follows that
\[
M \leq \left( 1+X \right) \left( 1+\frac{Y}{|C|} \right).
\]

Thus
\[
|C| \leq \frac{Y(1+X)}{M-1-X}.
\]

Suppose now
\[
M \geq \max\{ X+Y+1,X+Z+1 \}.
\]

As we saw in the proof of Lemma~\ref{lem:M2}, $M \geq X+Y+1$ implies that
\[
|C| \leq \frac{Y(1+X)}{M-1-X} \leq \frac{(X+1)(Y+1)}{M-X},
\]
as required.

The remaining upper bound for $|B|$ at the end of the lemma is proved
in the same way.
\end{proof}

\begin{lem}
\label{lem:M3}
Let $R_{1}$, $S_{1}$ and $T_{1}$ be positive integers and consider the set
\[
\widetilde{\Sigma}_{1} = \left\{ \left( x_{1},x_{2} \right)
= \left( r+t\beta_{1}, s+t\beta_{2} \right) :
0 \leq r \leq R_{1}, 0 \leq s \leq S_{1}, 0 \leq t \leq T_{1} \right\},
\]
where $\beta_{1}=b_{1}/b_{3}$ and $\beta_{2}=b_{2}/b_{3}$ with $b_{1}$, $b_{2}$ and $b_{3}$
coprime non-zero rational integers, and assume that
\[
\card \widetilde{\Sigma}_{1} = \left( R_{1}+1 \right) \left( S_{1}+1 \right) \left( T_{1}+1 \right).
\]

Put
\[
\cV = \left( \left( R_{1}+1 \right) \left( S_{1}+1 \right) \left( T_{1}+1 \right) \right)^{1/2}.
\]

For any $(\lambda ,\mu) \in \bbC^{2} \setminus \{(0,0)\}$ and any complex number $c$,
let $M_{c}$ be the number of elements
$\left( x_{1},x_{2} \right) \in \widetilde{\Sigma}_{1}$ such that
$\lambda x_{1} + \mu x_{2} = c$.

\noindent
{\rm (a)}
Let $\chi$ be a positive real number. If
\begin{equation}
\label{eq:*}
M_{c} < \cM := \max \left\{ R_{1}+S_{1}+1, S_{1}+T_{1}+1, R_{1}+T_{1}+1, \chi \cV \right\}
\end{equation}
does not hold, then there exist rational integers $u_{1}$, $u_{2}$ and $u_{3}$, not all zero, such that
\[
u_{1}b_{1}+u_{2}b_{2}+u_{3}b_{3}=0,
\]
with $\gcd \left( u_{1}, u_{2}, u_{3} \right)=1$ and
\[
\left| u_{1} \right|
\leq \frac{(S_{1}+1)( T_{1}+1)}{\cM-\max\{S_{1},T_{1}\} }, \qquad
\left| u_{2} \right|
\leq \frac{(R_{1}+1)( T_{1}+1)}{\cM-\max\{R_{1},T_{1}\} } \quad \text{and} \quad
\left| u_{3} \right|
\leq \frac{(R_{1}+1)( S_{1}+1)}{\cM-\max\{R_{1},S_{1}\}}.
\]

\vspace*{1.0mm}

\noindent
{\rm (b)} If the upper bound \eqref{eq:*} for $M_{c}$ holds then, for all
$(\lambda,\mu) \in \bbC^{2} \setminus \{(0,0)\}$, we have
\[
\card \left\{ \lambda x_{1}+\mu x_{2} : \left( x_{1},x_{2} \right) \in \widetilde{\Sigma}_{1} \right\}
\geq \frac{(R_{1}+1)(S_{1}+1)(T_{1}+1)}
          {\max \left\{ R_{1}+S_{1}+1, S_{1}+T_{1}+1, R_{1}+ T_{1}+1, \chi \cV \right\}}.
\]
\end{lem}

\begin{rem-nonum}
The introduction of $\chi \cV$ here turns out to be very helpful to us. In many
cases, $\chi \cV$ is much larger than the other terms in the definition of $\cM$
here. So its use here gives us much smaller upper bounds on the sizes of the
$u_{i}$'s. This gives us better results from the kit.
\end{rem-nonum}

\begin{proof}
(a) Suppose that \eqref{eq:*} does not hold for some triple $(\lambda,\mu,c)$.
Let $c$ be a complex number such that $M_{c}$ is maximal and consider the
associated values of $\lambda$ and $\mu$.
We distinguish the following possibilities for $\mu$ and $\lambda$.

$\bullet$
$\mu=0$: suppose that $\left( x_{1}, x_{2} \right) \in \widetilde{\Sigma}_{1}$
satisfies $\lambda x_{1}+\mu x_{2}=\lambda \left( r+t\beta_{1} \right)=c$.
So $b_{3}r+b_{1}t=cb_{3}/\lambda$ for some integers $0 \leq r \leq R_{1}$ and
$0 \leq t \leq T_{1}$ (since $\mu=0$ here and also $(\mu, \lambda) \neq (0,0)$,
we have $\lambda \neq 0$).

We will now apply Lemma~\ref{lem:plan}. Let $(X,Y,Z)$ there be $\left( S_{1}, R_{1}, T_{1} \right)$
and $(B,C,D)$ there be $\left( b_{3}/d_{1},b_{1}/d_{1},cb_{3}/ \left(\lambda d_{1} \right) \right)$,
where $d_{1}=\gcd \left( b_{1}, b_{3} \right)$.
Taking $r$ and $t$ here as $y$ and $z$, respectively, in the definition of $M$
in Lemma~\ref{lem:plan}, the equation $By+Cz=D$ in the definition of $M$
becomes our $\left( b_{3}/d_{1} \right)r + \left( b_{1}/d_{1} \right)t
=cb_{3}/\left( d_{1} \lambda \right)$.

Using the map $\sigma:\Sigma \rightarrow \widetilde{\Sigma}_{1}$ defined by
$\sigma: (s,r,t) \mapsto \left( r+t\beta_{1}, s+t\beta_{2} \right)$, we show that
the cardinalities of $\Sigma$ and $\widetilde{\Sigma}_{1}$ are equal.
The map is clearly surjective. Suppose that
\[
\sigma \left( s_{1}, r_{1}, t_{1} \right)
=\left( r_{1}+t_{1}\beta_{1}, s_{1}+t_{1}\beta_{2} \right)
=\left( r_{2}+t_{2}\beta_{1}, s_{2}+t_{2}\beta_{2} \right)
=\sigma \left( s_{2}, r_{2}, t_{2} \right).
\]
Then $\left( r_{1}-r_{2} \right) + \left( t_{1}-t_{2} \right) \beta_{2}
=\left( s_{1}-s_{2} \right) + \left( t_{1}-t_{2} \right) \beta_{2}=0$, so
$r_{1}-r_{2}=s_{1}-s_{2}$. In this case, we can write $r_{1}=r_{2}+k$ and 
$s_{1}=s_{2}+k$. Thus
$\left( r_{2}+k+t_{1}\beta_{1}, s_{2}+k+t_{1}\beta_{2} \right)
=\left( r_{2}+t_{2}\beta_{1}, s_{2}+t_{2}\beta_{2} \right)$, which can only
happen if $k=0$. This proves that $\sigma$ is injective too. Hence the cardinalities
of $\Sigma$ and $\widetilde{\Sigma}_{1}$ are equal

Therefore, since \eqref{eq:*} does not hold, the inequality for $M$ in
Lemma~\ref{lem:plan}(b) holds and we have
\[
\left| b_{3}/d_{1} \right| =|B| \leq \frac{\left( S_{1}+1 \right) \left( T_{1}+1 \right)}{M_{c}-S_{1}}
\leq \frac{\left( S_{1}+1 \right) \left( T_{1}+1 \right)}{\cM-S_{1}}
\]
and
\[
\left| b_{1}/d_{1} \right|=|C| \leq \frac{\left( S_{1}+1 \right) \left( R_{1}+1 \right)}{M_{c}-S_{1}}
\leq \frac{\left( S_{1}+1 \right) \left( R_{1}+1 \right)}{\cM-S_{1}}.
\]

We now use this information to obtain the linear relation we want between the
$b_{i}$'s. We have the trivial relationship $\left( b_{1}/d_{1} \right)b_{3}
-b_{1} \left( b_{3}/d_{1} \right)=0$, so we can let $u_{1}=-b_{3}/d_{1}$,
$u_{2}=0$ and $u_{3}=b_{1}/d_{1}$. The upper bounds above on
$\left| b_{3}/d_{1} \right|$ and $\left| b_{1}/d_{1} \right|$ establish our lemma
in this case.

\vspace*{1.0mm}

Now we assume $\mu \neq 0$ and, to simplify the notation, we take $\mu=1$.

$\bullet$ $\lambda=0$: by the same argument as for $\mu=0$, we have $b_{3} \left( \lambda x_{1}+\mu x_{2} \right)
=b_{3} \mu x_{2}=b_{3} \left( s+t\beta_{2} \right)
=b_{3}s+tb_{2}=b_{3}c$ for some $\left( x_{1}, x_{2} \right) \in \widetilde{\Sigma}_{1}$.
Here we apply Lemma~\ref{lem:plan} with
$\left( R_{1},S_{1},T_{1} \right)$ for $(X,Y,Z)$ and
$\left( b_{3}/d_{1},b_{2}/d_{1},b_{3}c/d_{1} \right)$ for $(B,C,D)$, where
$d_{1}=\gcd \left( b_{2},b_{3} \right)$. As in the case of $\mu=0$, Lemma~\ref{lem:plan}(b)
gives us
\[
\left| b_{3}/d_{1} \right|
=|B| \leq \frac{\left( R_{1}+1 \right) \left( T_{1}+1 \right)}{\cM-R_{1}}
\quad \text{and} \quad
\left| b_{2}/d_{1} \right|
=|C| \leq \frac{\left( R_{1}+1 \right) \left( S_{1}+1 \right)}{\cM-S_{1}}.
\]

As in the case of $\mu=0$, we have the relationship $u_{1}b_{1}+u_{2}b_{2}+u_{3}b_{3}=0$
with $u_{1}=0$, $u_{2}=b_{3}/d_{1}$ and $u_{3}=-b_{2}/d_{1}$.

\vspace*{1.0mm}

It remains to consider $\mu\lambda \neq 0$. We do so with two cases.

$\bullet$ $\lambda b_{1}+b_{2}=0$: we proceed in the same way as in the case of
$\lambda=0$. We have
$\lambda x_{1}+\mu x_{2}=-b_{2}/b_{1} \left( r+t\beta_{1} \right)+ s+t\beta_{2}=c$
(recalling that we take $\mu=1$).
Expanding this and simplifying it, we obtain $-b_{2}r+b_{1}s=cb_{1}$, so we use
Lemma~\ref{lem:plan} with $\left( 0, -b_{2}/d_{2},b_{1}/d_{2},cb_{1}/d_{2} \right)$
for $(A,B,C,D)$, $(t,r,s)$ for $(x,y,z)$ and $\left( T_{1}, R_{1}, S_{1} \right)$
for $\left( X,Y,Z \right)$, where $d_{2}=\gcd \left( b_{1},b_{2} \right)$. Here
\[
\left| b_{2}/d_{2} \right|
=|B| \leq \frac{\left( S_{1}+1 \right) \left( T_{1}+1 \right)}{\cM-T_{1}}
\quad \text{and} \quad
\left| b_{1}/d_{2} \right|
=|C| \leq \frac{\left( T_{1}+1 \right) \left( R_{1}+1 \right)}{\cM-T_{1}}.
\]

Notice the denominators here differ from those for the case of $\lambda=0$. This
explains why we need the $\max$ in our upper bounds in the lemma.

The desired relationship, $u_{1}b_{1}+u_{2}b_{2}+u_{3}b_{3}=0$, holds if we take
$u_{1}=b_{2}/d_{2}$, $u_{2}=-b_{1}/d_{2}$ and $u_{3}=0$.

\vspace*{1.0mm}

$\bullet$ $\lambda \mu \left( \lambda b_{1}+b_{2} \right) \neq 0$: we will show
that the desired relationship between the $b_{i}$'s holds here too. To proceed, we put
\[
E_{1} = \left\{ (r,s,t) \in \bbZ^{3} : 0 \leq r \leq R_{1}, 0 \leq s \leq S_{1}, 0 \leq t \leq T_{1} \right\}.
\]

Since $M_{c}>T_{1}+1$ (by our assumption that \eqref{eq:*} does not hold),
there exist two distinct triples $\left( r_{1},s_{1},t_{0} \right)$ and
$\left( r_{1}',s_{1}',t_{0} \right) \in E_{1}$ such that
\[
\lambda \left( r_{1}+\beta_{1}t_{0} \right) + \left( s_{1}+\beta_{2}t_{0} \right)
= \lambda \left( r_{1}'+\beta_{1}t_{0} \right) + \left( s_{1}'+\beta_{2}t_{0} \right),
\]
recalling our assumption (stated just before considering the case $\lambda=0$)
that $\mu=1$. This gives us a trivial linear relation between the $b_{i}$'s,
but it does tell us that $\lambda \left( r_{1}'-r_{1} \right)=s_{1}-s_{1}'$.
Since $\lambda \neq 0$ and at least one of $r_{1} \neq r_{1}'$ or $s_{1} \neq s_{1}'$
holds, it follows that both $r_{1} \neq r_{1}'$ and $s_{1} \neq s_{1}'$ hold.
Put 
$r_{1}''=\left( r_{1}'-r_{1} \right) / \gcd \left( r_{1}-r_{1}',s_{1}-s_{1}' \right)$
and $s_{1}''=\left( s_{1}-s_{1}' \right) / \gcd \left( r_{1}-r_{1}',s_{1}-s_{1}' \right)$,
then $\lambda = s_{1}''/r_{1}''$.

We now use this information about $\lambda$ to obtain a non-trivial linear
relation between the $b_{i}$'s whose coefficients we can bound.

We have
$\lambda x_{1}+\mu x_{2}=s_{1}''/r_{1}'' \left( r+t\beta_{1} \right)+ s+t\beta_{2}=c$
(recalling that we take $\mu=1$).
Expanding this and simplifying it, we obtain
\[
s_{1}''b_{3}r + \left( s_{1}''b_{1}+r_{1}''b_{2} \right)t + r_{1}''b_{3}s=r_{1}''b_{3}c,
\]
so we use
Lemma~\ref{lem:M2}(b) with $\left( s_{1}''b_{3}/\delta_{1}, \left( s_{1}''b_{1}+r_{1}''b_{2} \right)/\delta_{1}, r_{1}''b_{3}/\delta_{1}, r_{1}''b_{3}c/\delta_{1}\right)$
for $(A,B,C,D)$, $(r,t,s)$ for $(x,y,z)$ and $\left( R_{1}, T_{1}, S_{1} \right)$
for $\left( X,Y,Z \right)$, where 
\[
\delta_{1}=\gcd \left( s_{1}''b_{3}, s_{1}''b_{1}+r_{1}''b_{2}, r_{1}''b_{3} \right)=\gcd \left(b_{3}, s_{1}''b_{1}+r_{1}''b_{2} \right)
\]
since $r_{1}''$ and $s_{1}''$ are coprime.
Here
\[
\left| s_{1}''b_{3}/\delta_{1} \right| = \left| A \right| \leq \frac{\left( Y+1 \right) \left( Z+1 \right)}{\cM-\max \{ Y,Z \}}
\leq \frac{\left( S_{1}+1 \right) \left( T_{1}+1 \right)}{\cM-\max \{ S_{1}, T_{1} \}},
\]
\[
\left| \left( s_{1}''b_{1}+r_{1}''b_{2} \right)/\delta_{1} \right|
= \left| B \right| \leq \frac{\left( X+1 \right) \left( Z+1 \right)}{\cM-\max \{ X,Z \}}
\leq \frac{\left( R_{1}+1 \right) \left( S_{1}+1 \right)}{\cM-\max \{ R_{1}, S_{1} \}}
\]
and
\[
\left| r_{1}''b_{3}/\delta_{1} \right| = \left| C \right| \leq \frac{\left( X+1 \right) \left( Y+1 \right)}{\cM-\max \{ X,Y \}}
\leq \frac{\left( R_{1}+1 \right) \left( T_{1}+1 \right)}{\cM-\max \{ R_{1}, T_{1} \}}.
\]
Since $\delta_{1}$ divides $ s_{1}''b_{1}+r_{1}''b_{2}$, we have $ s_{1}''b_{1}+r_{1}''b_{2}=k_{1}\delta_{1}$.
Multiplying this by $b_{3}/\delta_{1}$, we get a linear relation
\[
u_{1}b_{1} + u_{2}b_{2} + u_{3}b_{3}=0
\]
with $u_{1}= s_{1}''b_{3}/\delta_{1}$, $u_{2}=r_{1}''b_{3}/\delta_{1}$
and $u_{3}=-\left( s_{1}''b_{1}+r_{1}''b_{2} \right)/\delta_{1}$. Thus
\begin{align*}
\left| u_{1} \right| &= \left| s_{1}''b_{3}/\delta_{1} \right|
\leq \frac{\left( S_{1}+1 \right) \left( T_{1}+1 \right)}{\cM-\max \{ S_{1}, T_{1} \}}, \\
\left| u_{2} \right| &= \left| r_{1}''b_{3}/\delta_{1} \right|
\leq \frac{\left( R_{1}+1 \right) \left( T_{1}+1 \right)}{\cM-\max \{ R_{1}, T_{1} \}}
\quad \text{and} \\
\left| u_{3} \right| &= \left| \left( s_{1}''b_{1}+r_{1}''b_{2} \right)/\delta_{1} \right|
\leq \frac{\left( R_{1}+1 \right) \left( S_{1}+1 \right)}{\cM-\max \{ R_{1}, S_{1} \}}.
\end{align*}

\vspace*{1.0mm}

(b) For $(\lambda,\mu) \in \bbC^{2}\setminus \{(0,0)\}$, we consider the cardinality
\[
N=\card \left\{ \lambda x_{1}+\mu x_{2} \, :\, \left( x_{1},x_{2} \right) \in \widetilde{\Sigma}_{1} \right\}.
\]

Putting $M=\max_{c \in \bbC} M_{c}$, 
we clearly have $N \geq \card \left( \widetilde{\Sigma}_{1} \right) /M$,
so part~(b) of the lemma follows from the assumption in the lemma that
\[
\card \widetilde{\Sigma}_{1} = \left( R_{1}+1 \right) \left( S_{1}+1 \right) \left( T_{1}+1 \right)
\]
and the assumption in part~(b) that $M \leq \cM$.
\end{proof}

\section{Proof of Main Result}
\label{sect:proof}

We start by showing that we can apply our zero lemma, Proposition~\ref{prop:zero-est},
to $\Delta$, so that we have $\Delta \neq 0$. This will allow us to use
Proposition~\ref{prop:M2} to obtain a lower bound for $\left| \Delta \right|$.

If the $N=K(K+1)L/2$ rows of the matrix used to define the interpolation determinant,
$\Delta$, in \eqref{eq:delta-defn} are linearly dependent, then there exists a
polynomial, $P \left( X_{1}, X_{2}, Y \right)$, not exactly zero, with
$P \left( r+t\beta_{1}, s+t\beta_{2}, \alpha_{1}^{r} \alpha_{2}^{s} \alpha_{3}^{t} \right)=0$
for all triples $(r,s,t)$ with $0 \leq r <R$, $0 \leq s < S$ and $0 \leq t<T$.
Since this polynomial arises from a linear combination of the rows, the maximum
exponent of $r+t\beta_{1}$ plus the maximum exponent of $s+t\beta_{2}$ is at most
$K-1$ and the maximum exponent of $\alpha_{1}^{r} \alpha_{2}^{s} \alpha_{3}^{t}$
is at most $L-1$, $\deg_{\uX}(P) \leq K-1$ and $\deg_{Y}(P) \leq L-1$.

Using the definition of the $\Sigma_{j}$'s in \eqref{eq:Sigma-j-defn}, along with
the lower bounds for $R$, $S$ and $T$ in \eqref{eq:RST-lb}, we use that the set
of all such triples $(r,s,t)$ contains $\Sigma_{1}+\Sigma_{2}+\Sigma_{3}$. Therefore,
if conditions~\eqref{eq:zero-est-i}, \eqref{eq:zero-est-ii} and \eqref{eq:zero-est-iii} in Proposition~\ref{prop:zero-est}
hold, then we find that $P \left( X_{1}, X_{2}, Y \right)$ is the zero polynomial.
This contradiction shows that the $N=K(K+1)L/2$ rows of the matrix used to define
the interpolation determinant, $\Delta$, in \eqref{eq:delta-defn} are not linearly
dependent and hence the interpolation determinant, $\Delta$, is not zero.

Thus, if we can show that conditions~\eqref{eq:thm41-i}--\eqref{eq:thm41-v} in
the theorem imply conditions~\eqref{eq:zero-est-i}, \eqref{eq:zero-est-ii} and
\eqref{eq:zero-est-iii} in Proposition~\ref{prop:zero-est}
(unless conditions~\eqref{eq:C1} or \eqref{eq:C2} hold),
then by Proposition~\ref{prop:M3}, the lower bound for $\Lambda'$ in the theorem
will hold (again, unless conditions~\eqref{eq:C1} or \eqref{eq:C2} hold).

\vspace*{3.0mm}

Condition~\eqref{eq:zero-est-i} of Proposition~\ref{prop:zero-est} has two subconditions.
The first subcondition is
\begin{equation}
\label{eq:cond-i.1}
\card \left\{ \lambda x_{1} + \mu x_{2}\, : \,
    \left( x_{1},x_{2},y \right) \in \Sigma_{1} \right\} >K, \quad \forall (\lambda, \mu )\neq(0,0).
\end{equation}

Recalling the definition of $\Sigma_{1}$ in \eqref{eq:Sigma-j-defn} and of
$\widetilde{\Sigma_{1}}$ in Lemma~\ref{lem:condI}, we have
\[
\card \left\{ \lambda x_{1} + \mu x_{2}\, : \,
    \left( x_{1},x_{2},y \right) \in \Sigma_{1} \right\}
= \card \left\{ \lambda x_{1} + \mu x_{2}\, : \,
    \left( x_{1},x_{2} \right) \in \widetilde{\Sigma}_{1} \right\}.
\]

By Lemma~\ref{lem:condI}, we find that
$\card \widetilde{\Sigma}_{1} = \left( R_{1}+1 \right) \left( S_{1}+1 \right) \left( T_{1}+1 \right)$
holds unless condition~\eqref{eq:C1} holds. So we may now assume that
$\card \widetilde{\Sigma}_{1} = \left( R_{1}+1 \right) \left( S_{1}+1 \right) \left( T_{1}+1 \right)$
holds. Thus, by Lemma~\ref{lem:M3}(b), condition~\eqref{eq:thm41-i} of the theorem implies that
\[
\card \left\{ \lambda x_{1} + \mu x_{2}\, : \,
    \left( x_{1},x_{2} \right) \in \widetilde{\Sigma}_{1} \right\} >K, \quad \forall (\lambda, \mu )\neq(0,0).
\]
holds, unless the condition in Lemma~\ref{lem:M3}(a) holds. This condition in
Lemma~\ref{lem:M3}(a) gives rise to condition~\eqref{eq:C2}.

The second subcondition of condition~\eqref{eq:zero-est-i} of Proposition~\ref{prop:zero-est} is
\begin{equation}
\label{eq:cond-i.2}
\card \left\{ y\, : \, \left( x_{1},x_{2},y \right) \in \Sigma_{1} \right\}>L.
\end{equation}

Condition~\eqref{eq:thm41-ii} in this theorem implies that this subcondition holds.

So we have shown that condition~\eqref{eq:zero-est-i} of Proposition~\ref{prop:zero-est}
follows from conditions~\eqref{eq:thm41-i} and \eqref{eq:thm41-ii} in this theorem,
provided that conditions~\eqref{eq:C1} and \eqref{eq:C2} do not hold.

\vspace*{3.0mm}

We now consider condition~\eqref{eq:zero-est-ii} of Proposition~\ref{prop:zero-est}.

It is also divided into two
subconditions. We replace the first one by the stronger condition
\begin{equation}
\label{eq:cond-ii.1}
\card \left\{ y : \left( x_{1},x_{2},y \right) \in \Sigma_{2}\right\} > 2KL.
\end{equation}

Condition~\eqref{eq:thm41-iii} in this theorem implies that this subcondition holds.

The second subcondition of condition~\eqref{eq:zero-est-ii} of Proposition~\ref{prop:zero-est} is
\begin{equation}
\label{eq:cond-ii.2}
\card \left\{ \left( x_{1},x_{2} \right) : 
    \left( x_{1},x_{2},y \right) \in \Sigma_{2} \right\} > K^{2}.
\end{equation}

By Lemma~\ref{lem:condI}, $\card \left\{ \left( x_{1},x_{2} \right) \, : \, \left( x_{1},x_{2},y \right) \in \Sigma_{2} \right\}
=\left( R_{2}+1 \right) \left( S_{2}+1 \right) \left( T_{2}+1 \right)$ holds
unless condition~\eqref{eq:C1} holds.
So condition~\eqref{eq:thm41-iv} in this theorem implies that this subcondition
holds unless condition~\eqref{eq:C1} holds.

\vspace*{3.0mm}

Condition~\eqref{eq:zero-est-iii} of Proposition~\ref{prop:zero-est} is that
$\card \Sigma_{3} > 3K^{2}L$. From the definition of $w$, if $\Lambda \in i\pi \bbQ$,
then $\Lambda=i\pi 2p/q$ where $p \neq 0$ and $0<|q| \leq w$. So, from the
assumption in this theorem that $0 < \left| \Lambda \right| < 2\pi/w$, it follows
that $\Lambda \not\in i\pi \bbQ$. Thus the map in Lemma~\ref{lem:condM1}(a) is
injective, so hypothesis~\eqref{eq:thm41-v} of the
theorem implies condition~\eqref{eq:zero-est-iii} of Proposition~\ref{prop:zero-est}
holds. This finishes the proof. 


\section{How to use Theorem~\ref{thm:main}}
\label{sect:how-to}

We will first consider the multiplicative group generated by the three algebraic
numbers $\alpha_{1}$, $\alpha_{2}$ and $\alpha_{3}$, which we will denote by $\cG$.

\subsection{About the multiplicative group $\cG$}

In practical examples, generally the following condition holds:
\begin{equation}
\label{hypo:M}
\begin{cases}
\text{either $\alpha_{1}$, $\alpha_{2}$ and $\alpha_{3}$ are multiplicatively independent, or} \\
\text{two of them are multiplicatively independent and the third is a root of unity $\neq 1$.}
\end{cases}
\end{equation}

We now use hypothesis~\eqref{hypo:M}, which is clearly stronger than the standard
hypothesis that the multiplicative group $\cG$ is of rank at least two. We also
notice that the order in $\bbC^{\times}$ of a root of unity ${} \neq 1$ is at
least equal to $2$, thus the condition~\eqref{eq:cond-i.2} is satisfied if
\begin{equation}
\label{eq:C.i.2}
\frac{2(R_{1}+1)(S_{1}+1)(T_{1}+1)}{W_{1}+1} >L,
\end{equation}
where
\[
W_{1}=\begin{cases}
       R_{1}, & \text{if $\alpha_{1}$ is a root of unity}, \\
       S_{1}, & \text{if $\alpha_{2}$ is a root of unity}, \\
       T_{1}, & \text{if $\alpha_{3}$ is a root of unity}, \\
       1,     & \text{otherwise},
      \end{cases}
\]
and recalling the definition of the $\Sigma_{j}$'s in \eqref{eq:Sigma-j-defn}.
But see also the first remark after \eqref{eq:C.ii.1} below.


In the same way, we see that to satisfy the condition~\eqref{eq:cond-ii.1} it is
enough to suppose that (when condition~\eqref{hypo:M} holds)
\begin{equation}
\label{eq:C.ii.1}
\frac{ \left( R_{2}+1 \right) \left( S_{2}+1 \right) \left( T_{2}+1 \right)}{W_{2}+1} >KL,
\end{equation}
where $W_{2}$ is defined by
\[
W_{2}=\begin{cases}
R_{2}, & \text{if $\alpha_{1}$ is a root of unity,} \\
S_{2}, & \text{if $\alpha_{2}$ is a root of unity,} \\
T_{2}, & \text{if $\alpha_{3}$ is a root of unity,} \\
1, & \text{otherwise}.
\end{cases}
\]

\begin{rem-nonum}
When (for example) $\alpha_{3}$ is a root of unity of order $\nu$,
condition~\eqref{eq:C.i.2} above can be replaced by
\[
\nu \left( R_{1}+1 \right) \left( S_{1}+1 \right) >L 
\]
(provided $T_{1} \geq \nu -1$)
and condition~\eqref{eq:C.ii.1}
can be replaced by
\[
\nu \left( R_{2}+1 \right) \left( S_{2}+1 \right)>KL
\]
(provided $T_{2} \geq \nu -1$).
\end{rem-nonum}

\begin{rem-nonum}
Under a weaker condition, one can obtain similar (but slightly weaker) conclusions
(see, for instance, \cite[Ex.~7.5, p.~229]{W}). 
\end{rem-nonum}

\subsection{The choice of parameters}
\label{subsec:param-choice}

Here we assume that condition~\eqref{hypo:M} holds, then by Lemma~\ref{lem:condM2}
above we know that $\Lambda \not\in i\pi \bbQ$.

To apply Theorem~\ref{thm:main}, we consider an integer $L \geq 5$ and real
parameters $m>0$, $\rho \geq 2$ and $\chi>0$. Note that having chosen $\rho$,
we can set the values of the $a_{i}$'s too.

Now we put
\begin{equation}
\label{eq:k-value}
K = \lfloor m L a_{1} a_{2} a_{3}\rfloor.
\end{equation}

The reason for this choice of $K$ is as follows. The main term on the left-hand
side of equation~\eqref{eq:o} is $KL\log(\rho)/2$, so it must be larger than
$\cD(K-1)\log(b)$. This suggests that we let $L=O \left( \cD\log(b)/\log(\rho) \right)$.
Thus our lower bound for $\log \left| \Lambda \right|$, which is $-\log(\rho)KL$,
is $O \left( a_{1}a_{2}a_{3}\cD^{2} \log^{2}(b)/\log(\rho) \right)$. This is our
desired form and consistent with the bounds for linear forms in two logs that we
obtain from this same technique (see, for example, \cite{LMN, L5}).

We will also assume that
\[
m \geq 1 \quad \text{and} \quad \Omega := a_{1} a_{2} a_{3} \geq 2.
\]

We define
\begin{align}
\label{eq:RST-values}
R_{1}    &= \lfloor c_{1} a_{2} a_{3} \rfloor,
& S_{1}  &= \lfloor c_{1} a_{1} a_{3} \rfloor,
& T_{1}  &= \lfloor c_{1} a_{1} a_{2} \rfloor, \nonumber \\
R_{2}    &= \lfloor c_{2} a_{2} a_{3} \rfloor,
& S_{2}  &= \lfloor c_{2} a_{1} a_{3} \rfloor,
& T_{2}  &= \lfloor c_{2} a_{1} a_{2} \rfloor, \\
R_{3}    &= \lfloor c_{3} a_{2} a_{3} \rfloor,
& S_{3}  &= \lfloor c_{3} a_{1} a_{3} \rfloor,
& T_{3}  &= \lfloor c_{3} a_{1} a_{2} \rfloor, \nonumber
\end{align}
where the parameters $c_{1}$, $c_{2}$ and $c_{3}$ will be chosen so that conditions
\eqref{eq:thm41-i} through \eqref{eq:thm41-v} of Theorem~\ref{thm:main} are satisfied.
The motivation for this choice of these quantities is so that all three terms in
$a_{1}R+a_{2}S+a_{3}T$ on the right-hand side of equation~\eqref{eq:o} are roughly
the same size,
$O \left( a_{1}a_{2}a_{3} \right)$, and so that the $gL \left( a_{1}R+a_{2}S+a_{3}T \right)$
term on the right-hand side of \eqref{eq:o} is roughly the same size as the other
main term on the right-hand side of \eqref{eq:o}, $\cD(K-1)\log b$.

We first consider condition~\eqref{eq:thm41-i} of Theorem~\ref{thm:main}.
Recalling that $\cV = \left( \left( R_{1}+1 \right) \left( S_{1}+1 \right) \left( T_{1}+1 \right) \right)^{1/2}$,
we see that $\left( R_{1}+1 \right) \left( S_{1}+1 \right) \left( T_{1}+1 \right)
> K \chi \cV$ holds, if
$\left( c_{1}^{3} \left (a_{1} a_{2} a_{3} \right)^{2} \right)^{1/2} \geq \chi m a_{1} a_{2} a_{3} L$.
I.e., $c_{1} \geq (\chi m L)^{2/3}$.

Next we establish conditions for
\[
\left( R_{1}+1 \right) \left( S_{1}+1 \right) \left( T_{1}+1 \right)
> K \cdot \max \left\{ R_{1}+S_{1}+1, S_{1}+T_{1}+1, R_{1}+T_{1}+1 \right\}
\]
to hold.
We consider the special case $a_{1}\le a_{2} \le a_{3}$ (the other cases are the
same), then $T_{1} \le S_{1} \le R_{1}$ and we want to show that
\[
\left( R_{1}+1 \right) \left( S_{1}+1 \right) \left( T_{1}+1 \right) > 
K \left( R_{1}+S_{1}+1 \right).
\]

Using the expressions for these quantities, this inequality will hold if
\[
\left( R_{1}+1 \right) c_{1}^{2}a_{1}^{2}a_{2}a_{3}
> mLa_{1}a_{2}a_{3} \left( R_{1}+1+c_{1}a_{1}a_{3} \right)
\]
holds.
If $ax>bx+c$ with $a,b,c>0$ holds for $x=x_{0}$, then it holds for all $x \geq x_{0}$.
So it suffices to show that
$c_{1}^{3}a_{1}^{2}a_{2}^{2}a_{3}^{2}
\geq mLa_{1}a_{2}a_{3} \left( c_{1}a_{2}a_{3}+c_{1}a_{1}a_{3} \right)$ holds.
This will hold if
$c_{1}^{2}a_{1}^{2}a_{2}^{2}a_{3}^{2} \geq 2mL\Omega^{2} \left( a_{1}^{-1}+a_{2}^{-1} \right)$
holds. That is, when
$c_{1}^{2} \geq \left( a_{1}^{-1}+a_{2}^{-1} \right) mL$ holds. In the general case, the wanted condition holds if
\[
c_{1}^{2} \geq \left( a^{-1}+a'^{-1} \right) mL, \quad \text{ where} \ a=\min \left\{ a_{1},a_{2},a_{3} \right\}
\text{ and }
\ a'=\min \left( \left\{ a_{1}, a_{2}, a_{3} \right\} \setminus \{a\} \right).
\]

Condition~\eqref{eq:thm41-ii} of Theorem~\ref{thm:main} holds when
$2c_{1}^{2} a_{1} a_{2} a_{3} \cdot \min \left\{ a_{1}, a_{2}, a_{3} \right\}
=2c_{1}^{2} \Omega a > L$,
provided that $c_{1}>2^{1/3}$ (since $\Omega \geq 2$). This inequality arises from the second part of
\eqref{hypo:M}, with the factor of $2$ on the left-hand side coming from the
fact that the order of the root of unity is at least $2$. The condition that
$c_{1}>2^{1/3}$ ensures that condition~\eqref{eq:thm41-ii} also holds when the
first part of \eqref{hypo:M} holds.

Thus, since we suppose $m \geq 1$ and also $\Omega \geq 2$, we can take
\begin{equation}
\label{eq:c1-value}
c_{1} = \max \left\{ 2^{1/3}, (\chi mL)^{2/3}, \left( \frac{2mL}{a} \right)^{1/2} \right\}.
\end{equation}

Our treatment of condition~\eqref{eq:thm41-iii} of Theorem~\ref{thm:main} is very
similar to that for condition~\eqref{eq:thm41-ii}. We want $2c_{2}^{2}\Omega>2KL$.
Thus $c_{2} = \sqrt{m/a}\, L$.

To satisfy condition~\eqref{eq:thm41-iv} of Theorem~\ref{thm:main}, we need
\[
\left( R_{2}+1 \right) \left( S_{2}+1 \right) \left( T_{2}+1 \right)
>K^{2}.
\]
 
Using our expressions above, this will hold if $c_{2}^{3}>m^{2}L^{2}$.

Combining these two expressions for $c_{2}$, we require
\begin{equation}
\label{eq:c2-value1}
c_{2} = \max \left\{ (mL)^{2/3}, \sqrt{m/a}\, L \right\}.
\end{equation}
Note we do not require $c_{2} \geq 2^{1/3}$ here explicitly, since $m \geq 1$
and $L \geq 5$ ensures that $(mL)^{2/3}>2^{1/3}$.

Finally, because of the hypothesis in \eqref{hypo:M}, we have $\Lambda \not\in i\pi \bbQ$
by Lemma~\ref{lem:condM2}. So, by Lemma~\ref{lem:condM1}, condition~\eqref{eq:thm41-v}
of Theorem~\ref{thm:main} holds for
\begin{equation}
\label{eq:c3-value}
c_{3} = \left( 3m^{2} \right)^{1/3}L.
\end{equation}

\begin{rem-nonum}
When $\alpha_{1}$, $\alpha_{2}$, $\alpha_{3}$ are multiplicatively independent
then it is enough to take $c_{1}$ and $c_{3}$ as above and
\begin{equation}
\label{eq:c2-value2}
c_{2} = (mL)^{2/3}.
\end{equation}
\end{rem-nonum}

\subsection{The degenerate case}
\label{subsect:degen}

In this subsection, we present some informal arguments for what happens in the
degenerate case. We obtain
\[
\log \left| \Lambda \right| \gg -a_{1}a_{2} a_{3} \min \left\{ a_{1}, a_{2}, a_{3} \right\} (\cD \log B)^{8/3}.
\]

\begin{rem-nonum}
It is this worse dependence on $\log B$ than in the non-degenerate case that
leads to the degenerate case having an impact on the results obtained in practice.
Fortunately, it is the constants that are important and our estimates should
lead to good results when compared to published previously ones (e.g., \cite{Matveev}).
See the examples in the next section for evidence of this.
\end{rem-nonum}

From condition~\eqref{eq:C1} in Theorem~\ref{thm:main}, we have
\[
b_{1} \leq \max \left\{ R_{1},R_{2} \right\}, \quad
b_{2} \leq \max \left\{ S_{1},S_{2} \right\} \quad \text{and } \quad
b_{3} \leq \max \left\{ T_{1},T_{2} \right\}.
\]

We now focus our attention on condition~\eqref{eq:C2}.
In the remainder of this subsection we put $\chi=1$. We have
\[
u_{1}b_{1}+u_{2}b_{2}+u_{3}b_{3}=0,
\]
with
\[
\left| u_{1} \right| \leq \frac{(S_{1}+1)(T_{1}+1)}{\cM-\max \{ S_{1}, T_{1} \}},
\quad \left| u_{2} \right| \leq \frac{(R_{1}+1)(T_{1}+1)}{\cM-\max \{ R_{1}, T_{1} \}}
\quad \text{and} \quad
\left| u_{3} \right| \leq \frac{(R_{1}+1)(S_{1}+1)}{\cM-\max \{ R_{1}, S_{1} \}}
\]
where
\[
\cM = \max \left\{ R_{1}+S_{1}+1, S_{1}+T_{1}+1, R_{1}+T_{1}+1, \chi \cV \right\}.
\]

This essentially implies that
\[
\left| u_{1} \right| \leq \sqrt{c_{1}} a_{1}/\chi, \quad
\left| u_{2} \right| \leq \sqrt{c_{1}} a_{2}/\chi \quad \text{and} \quad
\left| u_{3} \right| \leq \sqrt{c_{1}} a_{3}/\chi,
\]
since $R_{1}\approx c_{1} a_{2} a_{3}$, 
$S_{1}\approx c_{1} a_{1} a_{3}$, $T_{1}\approx c_{1} a_{1} a_{2}$ and
typically $\cM=\chi \cV \approx \chi c_{1}^{3/2} a_{1} a_{2}a_{3}$.

Suppose we eliminate $b_{1}$. Then
\[
u_{1} \Lambda = u_{1}b_{1}\log \alpha_{1} + u_{1}b_{2}\log \alpha_{2}
+ u_{1} b_{3} \log \alpha_{3}
= b_{2} \left( -u_{2} \log \alpha_{1}+u_{1}\log \alpha_{2} \right)
+ b_{3} \left( u_{1}\log \alpha_{3}-u_{3}\log \alpha_{1} \right).
\]
Applying \cite{LMN} to this linear form in two logs we get
\[
- \log \left| \Lambda \right|
\ll \left( \left| u_{1} \right| a_{2} + \left| u_{2} \right| a_{1} \right)
\left( \left| u_{1} \right| a_{3} + \left| u_{3} \right| a_{1} \right) \cD^{2}\log^{2} B,
\]
where (being somewhat pessimistic)
$B = \max \left\{ \left| b_{1} \right|, \left| b_{2} \right|, \left| b_{3} \right| \right\}$,
and the implied constant is an absolute constant. Using the upper
bounds for the $\left| u_{1} \right|$'s, we get
\[
- \log \left| \Lambda \right| \ll 
\left( \sqrt{c_{1}} a_{1} a_{2}/\chi \right) \left( \sqrt{c_{1}} a_{1} a_{3}/\chi \right)
\cD^{2}\log^{2} B 
\ll a_{1}^{2} a_{2} a_{3} L^{2/3}\cD^{2}\log^{2} B/\chi^{2},
\]
since we have $c_{1} \ll L^{2/3}$. Recalling that $L=O \left( \cD \log B \right)$,
we get
\[
-\log \left| \Lambda \right| \ll a_{1}^{2} a_{2} a_{3} (\cD\log B)^{8/3},
\]
where the implied constant is again absolute.

In the two remaining cases, where we eliminate $b_{2}$ or $b_{3}$, the argument
is identical and we obtain similar results:
\[
-\log \left| \Lambda \right|
\ll a_{1}a_{2}a_{3}a_{i}(\cD\log B)^{8/3},
\]
where we eliminate $b_{i}$. This suggests eliminating $b_{i}$ where $a_{i}=
\min \left\{ a_{1}, a_{2}, a_{3} \right\}$. This choice works best in our
examples below too.

Of course, one could use \cite{Matveev} (or any $\log B$ type lower bound) instead
of \cite{LMN}. This would lead to a lower bound for $\log \left| \Lambda \right|$
with $(\log B)^{5/3}$ instead of $(\log B)^{8/3}$. However, it would also lead
to a much larger constant and it is that constant that is more important than the
dependence on $B$ for our use here.

\section{Examples}
\label{sect:eg}

To demonstrate how to use our kit, we give two examples here, revisiting the
linear forms in three logs that arose in \cite{BMS1} and \cite{BMS2}.

These examples also provide comparison for readers. In both \cite{BMS1} and
\cite{BMS2}, the authors used earlier
versions of our kit due to the first author (see Section~12 of \cite{BMS1} and
Section~14 of \cite{BMS2}). In the first example \cite{BMS1}, the authors showed
that if $y^{p}=F_{n}$, then $p<197 \cdot 10^{6}$. Here we
obtain $p<18 \cdot 10^{6}$, roughly $11$ times smaller than the bound in \cite{BMS1}.
For the second example, we improve the upper bound in \cite{BMS2} as well as
correct mistakes in \cite{BMS2}.

We start with the following sharpening of Lemma~2.2 of \cite{PW} that we will
use throughout this section and in our code. In fact, it is explicit in their
proof. Roughly speaking, it removes the factor of $2^{h}$ from their result,
yielding bounds very close to the actual largest solution.

\begin{lem}
\label{lem:PW}
Let $a \geq 0$, $h \geq 1$ and $b> \left( 1/h \right)^{h}$ be real numbers
and let $x \in \bbR$ be the largest solution of $x=a+b \left( \log x \right)^{h}$.
Put $c=hb^{1/h}$. Then,
\[
x<\left( c \log c + \frac{\log c}{\log(c)-1} \left( a^{1/h} + c \log \log c \right) \right)^{h}.
\]
\end{lem}

\begin{proof}
This is the inequality on the second-last line of the proof of Lemma~2.2 of \cite{PW}
with a weaker condition on $b$, so we reprove their lemma to justify this weaker
condition.

Since $h \geq 1$, we know that $\left( z_{1}+z_{2} \right)^{1/h} \leq z_{1}^{1/h}+z_{2}^{1/h}$
for any positive real numbers, $z_{1}$ and $z_{2}$. Applying this to our expression
for $x$, we obtain
\[
x^{1/h} \leq a^{1/h} + c \log \left( x^{1/h} \right),
\]
where $c=hb^{1/h}$, provided $a>0$, $c>0$ and $x>1$.
Put $x^{1/h}=(1+y)c \log(c)$. We also have $c \log \left( x^{1/h} \right)
\leq x^{1/h}$ under these conditions. Hence
$c\log(c)+c\log \log \left( x^{1/h} \right) \leq x^{1/h}$.
So as long as $x>e^{h}$ and $\log(c)>0$, we have $y>0$ above.

Thus
\begin{align*}
(1+y)c \log (c) &= x^{1/h} \leq a^{1/h} + c \log (1+y) + c \log(c) + c \log \log (c) \\
& \leq a^{1/h} + cy + c \log(c) + c \log \log (c).
\end{align*}

Hence
\[
yc \left( \log(c)-1 \right) < a^{1/h} + c \log \log (c).
\]

The upper bound for $x$ in our lemma now follows, as in the proof of Lemma~2.2
of \cite{PW} except that the condition $c>e^{2}$ is not needed here.
\end{proof}

\subsection{Example~1: $y^{p}=F_{n}$}

\begin{thm}
\label{thm:eg1}
If $y^{p}=F_{n}$ has a solution for an odd prime $p$ and $y>1$, then
\begin{equation}
\label{eq:eg1}
p<18 \cdot 10^{6}.
\end{equation}
\end{thm}

\begin{proof}
Following Section~13 of \cite{BMS1}, we suppose that $y^{p}=F_{n}$. From Proposition~10.1
of \cite{BMS1}, we have
\begin{equation}
\label{eq:y-LB1}
\log y > 10^{20}.
\end{equation}

Here we will suppose that $p>10 \cdot 10^{6}$, rather than $p>2 \cdot 10^{8}$ in \cite{BMS1}.
The reason for this weaker bound on $p$ is to accommodate the improved upper bound we obtain
here. We will also use the principal branch of the logarithm throughout the proof.

\vspace*{1.0mm}

\noindent
\underline{Step~(1): Linear form definition and upper bound}\\

We now define the linear form in logs we will use and obtain an upper bound for
it.

In Section~13 of \cite{BMS1}, on page~1013, the authors consider
\[
\Lambda=n\log(\omega) - \log \sqrt{5}-p \log(y),
\]
which they rewrite as
\[
\Lambda=p \log \left( \omega^{k}/y \right) - q\log(\omega) - \log \sqrt{5}.
\]
Notice that $-\Lambda$ is in the form we consider in \eqref{eq:lfl-form}.

Here $\omega=\left( 1+\sqrt{5} \right)/2$ and $n=kp-q$ with $0 \leq q < p$.
Note that if $q=0$, then $\Lambda$ is a linear form in two logs and we obtain
a much better upper bound on $p$.

They also state (see the start of the proof of Proposition~11.1 on page~1000 or
the start of Section~13 on page~1013) that
\begin{equation}
\label{eq:lfl1-UB}
\log \left| \Lambda \right| < -2p\log(y)+1.
\end{equation}

\vspace*{1.0mm}

\noindent
\underline{Step~(2): Matveev}\\

In the notation of Theorem~\ref{thm:Mat}, we have $D=2$, $\alpha_{1}=\omega^{k}/y$,
$\alpha_{2}=\omega$, $\alpha_{3}=\sqrt{5}$, $b_{1}=p$, $b_{2}=-q>-p$ and $b_{3}=1$.

Recall that
$A_{j} \geq \max \left\{ D \h \left( \alpha_{j} \right), \left| \log \alpha_{j} \right| \right\}$.
Thus, we can take $A_{2}=\log(\omega)$ and $A_{3}=\log(5)$. For $A_{1}$, we need
a little more work.

From the first expression above for $\Lambda$ and \eqref{eq:lfl1-UB}, we have
\[
\frac{\log \sqrt{5}}{p}-\frac{e}{py^{2p}}
< \frac{n}{p}\log(\omega)-\log(y)
< \frac{\log \sqrt{5}}{p}-\frac{e}{py^{2p}}.
\]

Applying $n=kp-q$ and using $0 \leq q \leq p-1$, we have
\begin{align*}
0<\frac{\log \sqrt{5}}{p}-\frac{e}{py^{2p}}
\leq \frac{q}{p}\log(\omega)+\frac{\log \sqrt{5}}{p}-\frac{e}{py^{2p}}
&<k\log(\omega)-\log(y)
< \frac{q}{p}\log(\omega)+\frac{\log \sqrt{5}}{p}+\frac{e}{py^{2p}} \\
&\leq \log(\omega)-\frac{1}{p}\log(\omega)+\frac{\log \sqrt{5}}{p}+\frac{e}{py^{2p}}
\leq \log(\omega)+\frac{1}{3p}.
\end{align*}

Hence
\begin{equation}
\label{eq:a1-UB}
\left| \log \left( \omega^{k}/y \right) \right|
=\log \left| \omega^{k}/y \right| < \log(\omega)+10^{-6},
\end{equation}
since $p>10 \cdot 10^{6}$.

The conjugate of $\omega^{k}/y$ is $\omega^{-k}/y<1$, so
$\h \left( \omega^{k}/y \right)=(2\log(y)+k\log(\omega)-\log(y))/2=(k/2)\log(\omega)+(1/2)\log(y)$
(the $2\log(y)$ is because we need a factor of $y^{2}$ to clear the denominator
in the minimal polynomial of $\omega^{k}/y$). From \eqref{eq:a1-UB}, we have
$k\log(\omega)<\log(\omega)+\log(y)+10^{-6}$, so
\begin{equation}
\label{eq:hgtA1-UB}
\h \left( \alpha_{1} \right)=\h \left( \omega^{k}/y \right)<(1/2)\log(\omega)+\log(y)+10^{-6}
\end{equation}
and $A_{1}=2\h \left( \alpha_{1} \right) < 2\log(y)+0.4813$.
Thus $\max \left\{ \left| b_{j} \right| A_{j}/A_{1}: 1 \leq j \leq 3 \right\}=p$
and we can take $B=p$.

Applying Matveev's theorem (Theorem~\ref{thm:Mat} above) with $\chi=1$ and the
above quantities gives
\begin{align*}
\log \left| \Lambda \right| & > -\frac{5\cdot 16^{5}}{6} \cdot e^{3} \cdot 9 (3e/2)
\cdot \left( 26.25+\log(4\log(2e)) \right) \cdot 4 \cdot \left( 2\log(y)+0.4813 \right)
\cdot \log \omega \\
& \hspace*{5.0mm} \cdot \log(5) \cdot \log \left( 3ep\log(2e) \right) \\
&>
- \left(
7.10 \cdot 10^{10} + 2.71 \cdot 10^{10}\log(p) + 2.96 \cdot 10^{11}\log(y)
+1.13 \cdot 10^{11} \log(y)\log(p)
\right).
\end{align*}

Combining this lower bound for $\log \left| \Lambda \right|$ with the upper
bound in \eqref{eq:lfl1-UB}, and dividing by $2\log(y)$, we obtain
\[
1.476 \cdot 10^{11} + 5.62 \cdot 10^{10}\log(p)>p,
\]
using \eqref{eq:y-LB1}.

Applying Lemma~\ref{lem:PW} with $a=1.476 \cdot 10^{11}$,
$b=5.62 \cdot 10^{10}$, $h=1$ and $x=p$, so $c=hb^{1/h}=b$ and
\begin{equation}
\label{eq:eg1-nUB1}
p < b\log(b)+\frac{\log(b)}{\log(b)-1}(a+b\log(\log(b)))
<1.74 \cdot 10^{12}.
\end{equation}

The reason we take this step is because we first need an upper bound on
$p$ to control simultaneously the condition in \eqref{eq:o} and the degenerate
cases in our main theorem.

\vspace*{1.0mm}

\noindent
\underline{Step~(3): Non-degenerate case}\\

Here we apply Theorem~\ref{thm:main} to reduce our bound on $p$.

So that our linear form is in the form \eqref{eq:lfl-form}, we set
\[
\alpha_{1}=\omega, \quad
\alpha_{2}=\sqrt{5}, \quad \alpha_{3}=\omega^{k}/y
\quad b_{1}=q, \quad b_{2}=1
\hspace*{3.0mm} \text{ and } \hspace*{3.0mm}
b_{3}=p
\]
and in what follows (Steps~(3) and (4)), put
\[
\Lambda = b_{1} \log \alpha_{1}+b_{2}\log \alpha_{2}-b_{3}\log \alpha_{3}
=q \log \left( \omega \right) + 1 \cdot \log \left( \sqrt{5} \right)
-p \log \left( \omega^{k}/y \right).
\]
This is $-1$ times the $\Lambda$ considered above in Steps~(1) and (2).

Recall that we take
\[
a_{i} \geq \rho \left| \log \alpha_{i} \right|
-\log \left| \alpha_{i} \right| +2\cD \h \left( \alpha_{i} \right)
\]
and here $\cD=2$.

We have $\h(\omega)=\log(\omega)/2$, so we can take
$a_{1}=(\rho+1)\log(\omega)$.

Similarly, $\h \left( \sqrt{5} \right)=\log \left( \sqrt{5} \right)$, so
$a_{2}=(\rho+3)\log \left( \sqrt{5} \right)$.

In Step~(2), we saw that $\log \left| \alpha_{3} \right|=\log \left| \alpha_{3} \right|$
(recall that $\alpha_{3}$ here was denoted by $\alpha_{1}$ there),
so $\rho \left| \log \alpha_{3} \right|-\log \left| \alpha_{3} \right|
=(\rho-1) \log \left| \alpha_{3} \right|$. Applying \eqref{eq:a1-UB}, we obtain
\[
\rho \left| \log \alpha_{3} \right|-\log \left| \alpha_{3} \right|
< (\rho-1) \log \omega + (\rho-1) 10^{-6}.
\]

Combining this with \eqref{eq:hgtA1-UB}, we can take
\[
a_{3} = (\rho-1)\log(\omega)+2\log(\omega)+4\log(y)+(\rho+3) 10^{-6}
=(\rho+1)\log(\omega)+4\log(y)+(\rho+3) 10^{-6}.
\]

To apply Theorem~\ref{thm:main}, we need to select values for all the parameters
there. I.e., the positive rational integers $K$, $L$, $R$, $R_{1}$, $R_{2}$,
$R_{3}$, $S$, $S_{1}$, $S_{2}$, $S_{3}$, $T$, $T_{1}$, $T_{2}$ and $T_{3}$,
along with the real numbers $\rho$ and $\chi$.

We use the work in Section~\ref{sect:how-to} to reduce the amount of choice
involved here.

From \eqref{eq:k-value}, we see that $K$ depends on $a_{1}$, $a_{2}$, $a_{3}$,
$L$ and a real number $m \geq 1$.

From \eqref{eq:RST-values}, we see that the $R_{i}$'s, $S_{i}$'s and $T_{i}$'s
depend on $a_{1}$, $a_{2}$, $a_{3}$ and three positive real parameters $c_{1}$,
$c_{2}$ and $c_{3}$. Furthermore, we put $R=R_{1}+R_{2}+R_{3}+1$,
$S=S_{1}+S_{2}+S_{3}+1$ and $T=T_{1}+T_{2}+T_{3}+1$.

From \eqref{eq:c1-value}, \eqref{eq:c2-value1}, \eqref{eq:c3-value} and
\eqref{eq:c2-value2}, we have values for $c_{1}$, $c_{2}$ and $c_{3}$ in terms of
$m$, $L$, $a_{1}$, $a_{2}$, $a_{3}$ and $\chi$. For our linear form, this just
leaves $m$, $L$, $\rho$ and $\chi$ as unspecified parameters.

To apply Theorem~\ref{thm:main}, we do a brute force search. To minimise the
effect of the degenerate case we will use Theorem~2 of \cite{L5}. But this also
involves a search to obtain the best results, so we do not want to do such an
additional search for every choice of $m$, $L$, $\rho$ and $\chi$ that we consider.
Instead we do the degenerate case only once for each value of $\chi$.

For each of $20$ equidistributed values of $\chi$
satisfying $0.5 \leq \chi \leq 1.5$, we proceed as follows. First, we search
over integer values of $L$ with $100 \leq L \leq 200$,
20 values of each of $m$ and $\rho$ evenly distributed with
$4 \leq m \leq 9$ and $7 \leq \rho \leq 12$ that lead to \eqref{eq:o} being
satisfied and so that $KL\log(\rho)$ is as small as possible. With such a minimal
choice of parameters for Step~3 for each value of $\chi$, we find the associated
bound for Step~4 (the degenerate case) for this choice of parameters. The choice
of $\chi$ that leads to the best bound for both Step~3 and Step~4 is the one we
use.

There is nothing special about using 20 such values. It was only chosen to give
a good balance between speed and finding small admissible values of $KL\log(\rho)$.
The ranges on the parameters were found by experimentation.

This search led to the choice
\[
\chi=0.75, \quad L=167, \quad m=6 \quad \text{ and } \quad \rho=10.
\]

We have 
\[
K= \lfloor Lm a_{1} a_{2} a_{3} \rfloor
= \lfloor 221,945\log(y) \rfloor.
\]

Since $a=a_{1}$ and $a'=a_{2}$, we put
\[ 
c_{1}=82.65\ldots, \quad
c_{2}=100.13\ldots, \quad c_{3}=795.28\ldots.
\]

Using these values and the values of the $R_{i}$'s in \eqref{eq:RST-values}, we get 
\[
R_{1}=\lfloor c_{1} a_{2}a_{3} \rfloor =\lfloor 3458.9\log(y) \rfloor, \quad
R_{2}=\lfloor c_{2} a_{2}a_{3} \rfloor =\lfloor 4190.2\log(y) \rfloor,
\]
and
\[ 
R_{3}=\lfloor c_{3}a_{2}a_{3}\rfloor = \lfloor 33280\log(y) \rfloor.
\]

Further 
\[
S_{1}=\lfloor c_{1} a_{1}a_{3} \rfloor = \lfloor 1750.2\log(y) \rfloor, \quad
S_{2}=\lfloor c_{2} a_{1}a_{3} \rfloor = \lfloor 2120.2\log(y) \rfloor, \quad 
S_{3}=\lfloor c_{3} a_{1}a_{3} \rfloor = \lfloor 16839\log(y) \rfloor
\]
and finally
\[
T_{1}=\lfloor c_{1} a_{1}a_{2} \rfloor = 4577, \quad
T_{2}=\lfloor c_{2} a_{1}a_{2} \rfloor = 5544, 
\]
and
\[
T_{3}=\lfloor c_{3}a_{1}a_{2} \rfloor = 44,039.
\]

With $\cV = \left( \left( R_{1}+1 \right) \left( S_{1}+1 \right) \left( T_{1}+1 \right) \right)^{1/2}$,
we have $\chi \cV > 124,000\log(y)$, while
$5210\log(y) > R_{1}+S_{1}+1=\max \left\{ R_{1}+S_{1}+1,\, S_{1}+T_{1}+1,\,R_{1}+T_{1}+1 \right\}$,
since $\log(y)>10^{20}$, so $\cM=\chi \cV$.

With these choices, along with our lower bound
for $y$ and upper bound for $p$, we also find that
\begin{align*}
\log \left( b_{3}'\eta_{0} \right)
&< \log \left( (20465\log(y) + 27080)p \right)
<\log \log (y) + 38.11 \quad \text{ and} \\
\log \left( b_{3}''\zeta_{0} \right)
&< \log \left( (10355\log(y) - 1/2)p + 27080 \right)
< \log \log (y) + 37.43.
\end{align*}

Combining these estimates with Lemma~\ref{lem:Stir}(a) and our expression above
for $K$, we obtain
\[
\log (b')<54.58.
\]

As seen in Subsection~\ref{subsec:param-choice}, these choices imply that the
conditions~\eqref{eq:thm41-i}--\eqref{eq:thm41-v} of Theorem~\ref{thm:main} hold.
Moreover, the above choices have been made so that condition \eqref{eq:o} holds.

Thus we have 
\[
\log \left| \Lambda \right| \geq -KL\log \rho-\log(KL)>-8.535 \cdot 10^{7} \log(y).
\]

Combining this with the upper bound from \eqref{eq:lfl1-UB}, we get
\[
p<42.68 \cdot 10^{6}.
\]

\vspace*{1.0mm}

\noindent
\underline{Step~(4): Degenerate case}\\

Under condition~\eqref{eq:C1} of Theorem~\ref{thm:main}, we obtain
\[
p=b_{3} \leq \max \left\{ T_{1},T_{2} \right\} < 5600,
\]
which is excluded since we assume $p>10 \cdot 10^{6}$.

So we now consider condition~\eqref{eq:C2} of Theorem~\ref{thm:main}, where we have
\[
u_{1}b_{1} + u_{2}b_{2}+u_{3}b_{3}=u_{1}q + u_{2} + u_{3}p=0
\]
with $\gcd \left( u_{1}, u_{2}, u_{3} \right) = 1$.

We put
\[
U_{1}:=\frac{\left( S_{1}+1 \right) \left( T_{1}+1 \right)}
    {\cM-\max \{ S_{1}, T_{1} \}},
\hspace*{2.0mm}
U_{2}:=\frac{(R_{1}+1)( T_{1}+1) }
    {\cM-\max \{ R_{1}, T_{1} \}}
\hspace*{2.0mm} \text{ and } \hspace*{2.0mm}
U_{3}:=\frac{\left( R_{1}+1 \right) \left( S_{1}+1 \right)}
    {\cM-\max \{ R_{1}, S_{1} \}}.
\]

From the values of the relevant quantities in Step~(3) and $\log(y)>10^{20}$, we obtain
\[
\left| u_{1} \right| \leq \lfloor U_{1} \rfloor = 65, \hspace*{3.0mm}
\left| u_{2} \right| \leq U_{2} = 130 \hspace*{3.0mm} \text{ and }
\left| u_{3} \right| \leq \lfloor U_{3} \rfloor < 49.87\log(y).
\]

We will use this linear relation between the $b_{i}$'s to reduce the linear form,
$\Lambda$, to one in two logarithms. Let us make a remark here about how we
choose which $b_{i}$ to eliminate.

\begin{rem-nonum}
We can only eliminate a $b_{i}$ with $U_{i}$ bounded above by a constant.
Trying to eliminate a $b_{i}$ with $U_{i}$ depending on some parameter (like $U_{3}$
here depending on $\log y$) leads to both the quantities $a_{1}$ and $a_{2}$ in
Theorem~2 of \cite{L5} depending on that parameter, so we do not get an
absolute upper bound on the quantity we are interested in (i.e., $p$ here).

Here this means that we eliminate either $b_{1}=q$ or $b_{2}=1$.
Since $a_{1}<a_{2}$ here, our heuristic argument in Subsection~\ref{subsect:degen}
above suggests that we eliminate $b_{1}$.
\end{rem-nonum}

In our Pari/GP code, we tried eliminating both possibilities ($b_{1}$ and $b_{2}$)
and the best upper bound for $p$ comes from eliminating $b_{1}=q$. As noted
above, this is in keeping with our heuristic argument in Subsection~\ref{subsect:degen}.
So we consider $u_{1} \Lambda$:
\begin{align*}
u_{1} \Lambda &= u_{1}q\log(\omega) 
+u_{1}\log \left( \sqrt{5} \right)
- u_{1}p\log \left( \omega^{k}/y \right) \\
&= -\left( u_{2}q+u_{3}p \right)\log \left( \omega \right)
+u_{1} \log \left( \sqrt{5} \right)
- u_{1}p\log \left( \omega^{k}/y \right) \\
&= \log \left( \sqrt{5}^{u_{1}} \cdot \omega^{-u_{2}} \right)
- p \log \left( \left( \omega^{k}/y \right)^{u_{1}} \cdot \omega^{u_{3}} \right).
\end{align*}

We will use Theorem~2 in \cite{L5} to obtain lower bounds for this linear form.

We put $\alpha_{1}'=\sqrt{5}^{u_{1}} \cdot \omega^{-u_{2}}$,
$\alpha_{2}'=\left( \omega^{k}/y \right)^{u_{1}} \cdot \omega^{u_{3}}$,
$b_{1}=1$ and $b_{2}=p$.
We use $\alpha_{1}'$ and $\alpha_{2}'$ here for $\alpha_{1}$ and $\alpha_{2}$ in
\cite{L5} in order not to confuse it with our $\alpha_{1}$
and $\alpha_{2}$ above. As mentioned above, using Laurent's Theorem~2 requires
a search, here for the quantities that he labels as $\varrho$ (which plays the
analogous role for linear forms in two logs as our $\rho$) and $\mu$. Once
again, we do a brute force search over $20$ equidistributed values of each
parameter with $7 \leq \varrho \leq 11$ and $0.5 \leq \mu \leq 0.7$.
In this way, we take
\[
\varrho=10, \qquad \mu=0.61, \qquad a_{1}=1368.2
\qquad \text{ and } \qquad
a_{2}=524\log(y).
\]

We have
\[
\frac{b_{1}}{a_{2}}+\frac{b_{2}}{a_{1}}
<0.00074p.
\]
So $2\log(p)-9.373<h<\log(p)-2\log(p)-9.372$. Thus
\[
\log \left| \Lambda \right|
>423,900 \left( \log(p)-4.687 \right)^{2} \log(y).
\]

Combining this with the upper bound for $\Lambda$ in \eqref{eq:lfl1-UB}, we get
\[
-423,900 \left( \log(p)-4.687 \right)^{2} \log(y) < -2p\log(y)+\log \left| e \right|.
\]

Dividing both sides by $2\log(y)$, using $\log(y)>10^{20}$ and
again applying Lemma~\ref{lem:PW} with $a=\log(e)/\left( 2 \cdot 10^{20}\exp(4.687) \right)<10^{-6}$,
$b=423,900/ \left( 2 \exp(4.687) \right)$, $h=2$ and $x=p/\exp(4.687)$, we get
$c<88.41$ and
\[
p<34.86 \cdot 10^{6}.
\]

But we also have to consider the case that we cannot eliminate $b_{1}$.
This is the case when $u_{1}=0$. We proceed in the same way as we just did, but
now eliminate $b_{2}$, since $u_{2}$ is bounded above by a constant. Doing so
gives us the upper bound $p<39 \cdot 10^{6}$.


Combining this with the result of Step~(3), we have proved that $p \leq 42.68 \cdot 10^{6}$.

\vspace*{1.0mm}

\noindent
\underline{Step~(5): Iteration of Steps~(3) and (4)}\\

As in \cite{BMS1}, we repeated Steps~(3) and (4) a second time to obtain the improved
upper bound $p \leq 19.4 \cdot 10^{6}$.

We repeat this same search a third time with this further improved upper bound for $p$
to obtain $p<17.92 \cdot 10^{6}$.

\begin{center}
\begin{tabular}{|c|c|c|c|c|c|c|c|c|} \hline
iteration & \text{initial upper bound for $p$} & $L$ & $m$ & $\rho$ & $\chi$ & $\varrho$ & $\mu$ & \text{new upper bound for $p$} \\ \hline
$1$ & $1.8  \cdot 10^{12}$ & $167$ &    $6$ &   $10$ &  $0.75$ &  $10$ & $0.61$ & $43 \cdot 10^{6}$ \\ \hline
$2$ &   $43 \cdot  10^{6}$ & $105$ & $7.25$ & $9.75$ &  $1.03$ &  $10$ & $0.61$ & $19.4 \cdot 10^{6}$ \\ \hline
$3$ & $19.4 \cdot  10^{6}$ & $104$ &  $7.4$ &  $9.4$ &  $1.06$ & $9.8$ & $0.61$ & $17.92 \cdot 10^{6}$ \\ \hline
\end{tabular}
\end{center}

The three iterations took 180, 187 and 70 seconds on a Windows laptop with
an Intel i7-9750H 2.60GHz CPU and 16Gb of RAM.

The third iteration gives us the upper bound for $p$ stated in the theorem.
\end{proof}

From the table, one can see that little improvement is obtained
after the second iteration.

If one could ignore the degenerate case, as we conjecture should be possible,
and only consider the inequality
\eqref{eq:o} for the non-degenerate case, then one would obtain $p<12.4 \cdot 10^{6}$
instead. So we are within $50\%$ of the best possible result that our transcendence
argument can provide. Our kit should always provide such proximity to the optimal result
when considering the real
case for our linear forms in logs (as described in Subsection~\ref{subsect:conventions}).

\subsection{Example~2: $x^{2} + 7 = y^{p}$}

This is the case $D=7$ examined in detail in Section~15 of \cite{BMS2}.
There the authors claimed that $p<130 \cdot 10^{6}$. Our work here suggests
that the best possible bound they could have obtained was $p<156 \cdot 10^{6}$.
While our result here is over $6$ times smaller than this, our improvement here is
not as large as for the previous example. The reason is because in \cite{BMS2},
the zero estimate of Laurent \cite{L4}, given in Appendix~A below, was used. This
was an improvement over the zero estimate used in \cite{BMS1}.

So we take the opportunity here to correct the handling of $D=7$ in that paper.
In addition to the above, not all of the $R_{i}$'s, $S_{i}$'s and $T_{i}$'s can
be constants as stated in Section~15 of \cite{BMS2}. A dependence on $\log(y)$
is required. See our correct choice of these parameters in Step~(3) below.

One last note about our result here. The upper bound for $p$ is the best possible
one, given our inequality \eqref{eq:o} for the non-degenerate case. The degenerate
case does not adversely affect the results we obtain here.
This turns out to always happen when, as here, we are considering the imaginary
case for our linear forms in logs (as described in Subsection~\ref{subsect:conventions}).

\begin{thm}
\label{thm:eg2}
If $x^{2}+7=y^{p}$ has a solution for a prime $p \geq 3$ with $x, y \in \bbZ$,
then
\begin{equation}
\label{eq:eg2}
p<25 \cdot 10^{6}.
\end{equation}
\end{thm}

\begin{proof}
We will assume that $p>20 \cdot 10^{6}$ and use the modular lower bound for $y$
in equation~(14) of \cite{BMS2}:
\begin{equation}
\label{eq:y-LB2}
y \geq \left( \sqrt{p} \, - 1 \right)^{2}>19.9 \cdot 10^{6}.
\end{equation}

We will use the principal branch of the logarithm throughout the proof.

\vspace*{1.0mm}

\noindent
\underline{Step~(1): Linear form definition and upper bound}\\

In Section~15 of \cite{BMS2}, on page~56, the authors consider
\[
\Lambda=2\log \left( \varepsilon_{1} \overline{\alpha_{0}}/\alpha_{0} \right)
+ p \log \left( \varepsilon_{2} \overline{\gamma}/\gamma \right)
+ iq\pi,
\]
for some rational integer $q$ with $|q|<p$,
$\varepsilon_{1}, \varepsilon_{2}= \pm 1$,
$\alpha_{0}= \left( 1+\sqrt{-7} \right)/2$ and $\gamma$ is an algebraic integer
in $\bbQ \left( \sqrt{-7} \right)$ with norm $y$ such that
\[
\left( \frac{x-\sqrt{-7}}{x+\sqrt{-7}} \right)^{k}
= \left( \overline{\alpha_{0}}/\alpha_{0} \right)^{\kappa} \left( \pm \overline{\gamma}/\gamma \right)^{p}.
\]

This expression comes from Lemma~13.1 of \cite{BMS2} and its proof since
$\bbQ \left( \sqrt{-7} \right)$ has class number $1$, so $k_{0}=1$ there.
As a result, their $\kappa=2$ and $k=1$. They assert in the proof of their
Lemma~13.4 that this value of $\alpha_{0}$ is valid.

From their Lemma~13.3, we have
\begin{equation}
\label{eq:lfl2-UB}
\log \left| \Lambda \right| < -\frac{p}{2}\log(y)+\log \left( 2.2\sqrt{7} \right),
\end{equation}
since $D_{1}=1$ and $D_{2}=7$.

This is the case~(I) linear form that they consider there.

\vspace*{1.0mm}

\noindent
\underline{Step~(2): Matveev}\\
In the notation of Theorem~\ref{thm:Mat}, we have
$\alpha_{1}=\varepsilon_{2} \overline{\gamma}/\gamma$,
$\alpha_{2}=\varepsilon_{1} \overline{\alpha_{0}}/\alpha_{0}$,
$\alpha_{3}=-1$,
$b_{1}=p$, $b_{2}=2$ and
$b_{3}=q$. So $D=\chi=2$.

Note that we have swapped the $\alpha_{1}$ term with the $\alpha_{2}$ term here
with those in the case~(I) linear form in \cite{BMS2}.
This will result in $A_{1}$ being the largest of the $A_{i}$'s,
Doing so lets us take $B=p$ in Theorem~\ref{thm:Mat}.

Recall that
$A_{j} \geq \max \left\{ D \h \left( \alpha_{j} \right), \left| \log \alpha_{j} \right| \right\}$.
Since the norm of $\gamma$ is $y$, we have $\h \left( \gamma \right)=\log(y)/2$
and since $\alpha_{1}$ is on the unit circle, by our choice of $\varepsilon_{2}$,
we have $\left| \log \alpha_{1} \right|<\pi/2$. Thus, we can take
$A_{1}=\log(y)$, since $y>20 \cdot 10^{6}$ (by \eqref{eq:y-LB2}).

Similarly, for $A_{2}$, we have $d=2$ by Lemma~13.1 and Table~4 of \cite{BMS2}.
So from their Lemma~13.1, $\h \left( \alpha_{2} \right) = \log(2)/2$.
Also, $\left| \log \alpha_{2} \right|=0.722734\ldots$, so we can take
$A_{2}=0.73$.

Lastly, we can take $A_{3}=\pi$.

Applying Matveev's theorem (Theorem~\ref{thm:Mat} above) with the above quantities
gives
\begin{align*}
\log \left| \Lambda \right| & > -\frac{5\cdot 16^{5}}{6 \cdot 2} e^{3} (7+2\cdot 2) (3e/2)^{2}
\cdot \left( 26.25+\log \left( 2^{2}\log(2e) \right) \right) \cdot 2^{2} \log(y) 0.73 \pi
\log \left( 1.5e \cdot 2 p\log(2e) \right) \\
&>
- 4.11 \cdot 10^{11}\log(y)\log(13.81p).
\end{align*}

Combining this lower bound for $\log \left| \Lambda \right|$ with the upper
bound in \eqref{eq:lfl2-UB}, and dividing by $\log(y)/2$, we obtain
\[
8.21 \cdot 10^{11}\log(p)+2.16 \cdot 10^{12}
>8.21 \cdot 10^{11}\log(13.81p)+2\log \left( 2.2\sqrt{7} \right)/\log(y)>p,
\]
using \eqref{eq:y-LB2}.

Applying Lemma~\ref{lem:PW} with $a=2.16 \cdot 10^{12}$,
$b=8.21 \cdot 10^{11}$, $h=1$ and $x=p$, so $c=hb^{1/h}=b$ and
\begin{equation}
\label{eq:eg2-nUB1}
p < b\log(b)+\frac{\log(b)}{\log(b)-1}(a+b\log(\log(b)))
<2.76 \cdot 10^{13}.
\end{equation}

\vspace*{1.0mm}

\noindent
\underline{Step~(3): Non-degenerate case}\\
Here we apply Theorem~\ref{thm:main} to reduce our bound on $p$.

Recall that we take
$a_{i} \geq \rho \left| \log \alpha_{i} \right|
-\log \left| \alpha_{i} \right| +2\cD \h \left( \alpha_{i} \right)$
and here $\cD=1$.

Using the values of $\h \left( \alpha_{i} \right)$ and
$\left| \log \alpha_{i} \right|$ that we found in Step~(2),
we can take $a_{1}=\rho\pi/2+\log(y)$, $a_{2}=0.723\rho+\log(2)$ and
$a_{3}=\rho\pi$.

To apply Theorem~\ref{thm:main}, we do a brute force search in the same way as
we did in the first example. For each of $20$ equidistributed values of $\chi$
satisfying $0.04 \leq \chi \leq 0.24$, we proceed as follows. First, we search
over integer values of $L$ with $30 \leq L \leq 200$,
20 values of each of $m$ and $\rho$ evenly distributed with
$10 \leq m \leq 30$ and $3 \leq \rho \leq 13$ that lead to \eqref{eq:o} being
satisfied and so that $KL\log(\rho)$ is as small as possible. With such a minimal
choice of parameters for Step~3 for each value of $\chi$, we find the associated
bound for Step~4 (the degenerate case) for this choice of parameters. The choice
of $\chi$ that leads to the best bound for both Step~3 and Step~4 is the one we
use.

This search led to the choice
\[
\chi=0.08, \quad L=106, \quad m=21.0 \quad \text{ and } \quad \rho=5.5.
\]

Since $\lfloor Lm a_{1} a_{2} a_{3} \rfloor < \lfloor 300,476\log(y) \rfloor$, we put
\[
K= \lfloor 231,600\log(y) \rfloor.
\]

We have $a=a_{2}$ and $a'=a_{3}$ and put
\[ 
c_{1}=33.46\ldots, \quad
c_{2}=243.59\ldots, \quad c_{3}=1163.65\ldots.
\]

Using these values and the values of the $R_{i}$'s in \eqref{eq:RST-values}, we get 
\[
R_{1}=\lfloor c_{1} a_{2}a_{3} \rfloor =2299, \quad
R_{2}=\lfloor c_{2} a_{2}a_{3} \rfloor =16,737,
\]
and
\[ 
R_{3}=\lfloor c_{3}a_{2}a_{3}\rfloor = 79,953.
\]

Further,
\[
S_{1}=\lfloor c_{1} a_{1}a_{3} \rfloor = \lfloor 876\log y \rfloor, \quad
S_{2}=\lfloor c_{2} a_{1}a_{3} \rfloor = \lfloor 6373\log y \rfloor, \quad 
S_{3}=\lfloor c_{3} a_{1}a_{3} \rfloor = \lfloor 30440\log y \rfloor
\]
and finally
\[
T_{1}=\lfloor c_{1} a_{1}a_{2} \rfloor = \lfloor 202\log y \rfloor, \quad
T_{2}=\lfloor c_{2} a_{1}a_{2} \rfloor = \lfloor 1467\log y \rfloor,
\]
and
\[
T_{3}=\lfloor c_{3}a_{1}a_{2} \rfloor = \lfloor 7006\log y \rfloor.
\]

With $\cV = \left( \left( R_{1}+1 \right) \left( S_{1}+1 \right) \left( T_{1}+1 \right) \right)^{1/2}$,
we have $\chi \cV > 1611\log(y)$, while
$1100\log(y) > S_{1}+T_{1}+1=\max \left\{ R_{1}+S_{1}+1,\, S_{1}+T_{1}+1,\,R_{1}+T_{1}+1 \right\}$,
since $\log(y)>17.4$, so $\cM=\chi \cV$.

With these choices, along with our lower bound
for $y$ and upper bound for $n$, we also find that
\begin{align*}
\log \left( b_{3}'\eta_{0} \right)
&< \log \left( \left( 4337\log(y) + 49,500 \right) n \right)
<\log \log (y) + 39.85 \quad \text{ and} \\
\log \left( b_{3}''\zeta_{0} \right)
&< \log \left( \left( 8680 + 18,850n \right) \log(y) \right)
< \log \log (y) + 40.8.
\end{align*}

Combining these estimates with Lemma~\ref{lem:Stir}(a) and our expression above
for $K$, we obtain
\[
\log (b')<59.6.
\]

As seen above, these choices imply that the conditions~\eqref{eq:thm41-i}--\eqref{eq:thm41-v}
of Theorem~\ref{thm:main} hold. Moreover, the above choices have been made so
that condition \eqref{eq:o} holds.

Thus we have 
\[
\log \left| \Lambda \right| \geq -KL\log \rho-\log(KL)>-4.185 \cdot 10^{7} \log(y).
\]

Combining this with the upper bound from \eqref{eq:lfl1-UB}, we get
\[
p<83.69 \cdot 10^{6}.
\]

\vspace*{1.0mm}

\noindent
\underline{Step~(4): Degenerate case}\\
Under condition~\eqref{eq:C1} of Theorem~\ref{thm:main}, we obtain
\[
p=b_{1} \leq \max \left\{ R_{1},R_{2} \right\} < 16,800,
\]
which is excluded since we assume $p>20 \cdot 10^{6}$.

So we now consider condition~\eqref{eq:C2} of Theorem~\ref{thm:main}, where we have
\[
u_{1}b_{1} + u_{2}b_{2}+u_{3}b_{3}=u_{1}q + u_{2} + u_{3}p=0
\]
with $\gcd \left( u_{1}, u_{2}, u_{3} \right) = 1$.

We put
\[
U_{1}:=\frac{\left( S_{1}+1 \right) \left( T_{1}+1 \right)}
    {\cM-\max \{ S_{1}, T_{1} \}},
\hspace*{2.0mm}
U_{2}:=\frac{(R_{1}+1)( T_{1}+1) }
    {\cM-\max \{ R_{1}, T_{1} \}}
\hspace*{2.0mm} \text{ and } \hspace*{2.0mm}
U_{3}:=\frac{\left( R_{1}+1 \right) \left( S_{1}+1 \right)}
    {\cM-\max \{ R_{1}, S_{1} \}}.
\]

From the values of the relevant quantities in Step~(3) and $\log(y)>17.4$, we obtain
\[
\left| u_{1} \right| \leq U_{1} <239.64\log(y), \hspace*{3.0mm}
\left| u_{2} \right| \leq \lfloor U_{2} \rfloor = 328 \hspace*{3.0mm} \text{ and }
\left| u_{3} \right| \leq \lfloor U_{3} \rfloor = 2735.
\]

Here we use this linear relation between the $b_{i}$'s to reduce the linear form,
$\Lambda$, to one in two logarithms by eliminating $b_{2}$:
\begin{align*}
u_{2} \Lambda &=
2u_{2}\log \left( \varepsilon_{1} \overline{\alpha_{0}}/\alpha_{0} \right)
+ u_{2}p \log \left( \varepsilon_{2} \overline{\gamma}/\gamma \right)
+ u_{2}q \log(-1) \\
&= -\left( u_{1}p+u_{3}q \right)\log \left( \varepsilon_{1} \overline{\alpha_{0}}/\alpha_{0} \right)
+ u_{2}p \log \left( \varepsilon_{2} \overline{\gamma}/\gamma \right)
+ u_{2}q \log(-1) \\
&= p\log \left( \left( \varepsilon_{2} \overline{\gamma}/\gamma \right)^{u_{2}} \cdot \left( \varepsilon_{1} \overline{\alpha_{0}}/\alpha_{0} \right)^{-u_{1}} \right)
- q \log \left( \left( \varepsilon_{1} \overline{\alpha_{0}}/\alpha_{0} \right)^{u_{3}} \cdot (-1)^{-u_{2}} \right).
\end{align*}

So we put $\alpha_{1}=\left( \varepsilon_{2} \overline{\gamma}/\gamma \right)^{u_{2}} \cdot \left( \varepsilon_{1} \overline{\alpha_{0}}/\alpha_{0} \right)^{-u_{1}}$,
$\alpha_{2}=\left( \varepsilon_{1} \overline{\alpha_{0}}/\alpha_{0} \right)^{u_{3}} \cdot (-1)^{-u_{2}}$,
$b_{1}=p$ and $b_{2}=q$ in Theorem~2 of \cite{L5}. In the same way as in Example~1,
we take
\[
\varrho=180, \qquad \mu=0.61, \qquad a_{1}=495.2\log(y)+565.5
\qquad \text{ and } \qquad
a_{2}=2461.3.
\]

We have
\[
\frac{b_{1}}{a_{2}}+\frac{b_{2}}{a_{1}}
<0.00052p,
\]
since $D=1$, $\log(y)>17.4$ and $p>20 \cdot 10^{6}$.
So $\log(p)-4.431<h<\log(p)-4.185$. Thus
\[
\log \left| \Lambda \right|
>28,100 \left( \log(p)-4.185 \right)^{2} \log(y).
\]

Combining this with the upper bound for $\Lambda$ in \eqref{eq:lfl2-UB}, we get
\[
-28,100 \left( \log(p)-4.185 \right)^{2} \log(y)
< -(p/2)\log(y)+\log \left| 2.2\sqrt{7} \right|.
\]

Dividing both sides by $-(1/2)\log(y)$ and
again applying Lemma~\ref{lem:PW} with
\[
a=\frac{\log \left(2.2 \sqrt{7} \right)}{(1/2) \log \left( 19.9 \cdot 10^{6} \right) \exp(4.185)}
<0.0032,
\]
$b=28,100/(0.5\exp(4.185))$, $h=2$ and $x=p/\exp(4.185)$, we get $c<58.46$ and
\[
p<\exp(4.185) \cdot 347^{2}<79.2 \cdot 10^{6}.
\]

Similarly, when we consider the possibility that $u_{2}=0$, we find that
$p<54.2 \cdot 10^{6}$.

Combining this with the result of Step~(3), we have proved that $p<84 \cdot 10^{6}$.

\vspace*{1.0mm}

\noindent
\underline{Step~(5): Iteration of Steps~(3) and (4)}\\
As in \cite{BMS2}, we repeated Steps~(3) and (4) a second time using the improved
upper bound $p < 84 \cdot 10^{6}$.

\begin{center}
\begin{tabular}{|c|c|c|c|c|c|c|c|c|} \hline
iteration & \text{initial upper bound for $p$} & $L$ & $m$ & $\rho$ & $\chi$ & $\varrho$ & $\mu$ & \text{new upper bound for $p$} \\ \hline
$1$ & $2.76 \cdot 10^{13}$ & $106$ & $21.0$ &  $5.5$ &  $0.08$ & $180$ & $0.61$ & $84 \cdot 10^{6}$ \\ \hline
$2$ &   $84 \cdot 10^{6}$  &  $59$ & $18.0$ &  $6.0$ &  $0.1$  & $180$ & $0.61$ & $29 \cdot 10^{6}$ \\ \hline
$3$ &   $29 \cdot 10^{6}$  &  $59$ & $18.0$ & $5.75$ &  $0.1$  & $180$ & $0.61$ & $25.4 \cdot 10^{6}$ \\ \hline
$4$ & $25.5 \cdot 10^{6}$  &  $57$ & $19.0$ & $5.75$ &  $0.1$  & $180$ & $0.61$ & $24.94 \cdot 10^{6}$ \\ \hline
\end{tabular}
\end{center}

The four iterations took 191, 188, 103 and 104 seconds on a Windows laptop with
an Intel i7-9750H 2.60GHz CPU and 16Gb of RAM.

The fourth iteration gives us the upper bound for $p$ stated in the theorem.
\end{proof}

\appendix

\section{A Zero Estimate by Michel Laurent}

We revisit the original argument due to Masser \cite{A-Mas}, establishing zero
lemmas in algebraic commutative groups. Starting with a hypersurface, his approach
is based on the construction of complete intersections in successive codimensions
$2$, $3$, \ldots, using subsets of points $\Sigma_{1}$, $\Sigma_{2}$, $\ldots$ as
translation operators. Compared with subsequent works, see \cite{A-Phil} for instance,
the process enables us to control efficiently the possible degeneracies at each
step of the construction. We take advantage of this feature to minimise the size
of the sets $\Sigma_{1}$, $\Sigma_{2}$ and $\Sigma_{3}$ occurring in the following
proposition.

\begin{prop}
\label{prop:laurent-zero-est}
Let $\bbK$ be an algebraically closed field of characteristic $0$.
Let $K_{1}$, $K_{2}$ and $L$ be non-negative integers and let $\Sigma_{1}$, $\Sigma_{2}$ and
$\Sigma_{3}$ be finite subsets of the group $G = \bbK^{2} \times \bbK^{\times}$
$($whose composition law is written additively$)$. Assume that $\Sigma_{1}$,
$\Sigma_{2}$ and $\Sigma_{3}$ contain the origin $(0,0,1)$ of $G$ and that
\begin{equation}
\label{eq:A-i}
\begin{cases}
\card \left\{ ax_{1}+bx_{2} :
\text{$\exists y \in \bbK^{\times}$ with $\left( x_{1},x_{2},y \right) \in \Sigma_{1}$} \right\}
&> \max \left\{ K_{1}, K_{2} \right\},
\hspace*{1.0mm} \forall (a,b)\in \bbK^{2}\setminus \{(0,0)\}, \\
\card \left\{ y :
\text{$\exists \left( x_{1}, x_{2} \right) \in \bbK^{2}$ with $\left( x_{1},x_{2},y \right) \in \Sigma_{1}$} \right\}
&> L,
\end{cases}
\end{equation}

\begin{equation}
\label{eq:A-ii}
\begin{cases}
\card \left\{ \left( ax_{1}+bx_{2},y \right) : \left( x_{1},x_{2},y \right) \in \Sigma_{2} \right\}
&> 2\max \left\{ K_{1}, K_{2} \right\}L,
\hspace*{1.0mm} \forall (a,b) \in \bbK^{2}\setminus \{(0,0)\}, \\
\card \left\{ \left( x_{1},x_{2} \right) :
\text{$\exists y \in \bbK^{\times}$ with $\left( x_{1},x_{2},y \right) \in \Sigma_{2}$} \right\}
&> 2K_{1}K_{2},
\end{cases}
\end{equation}
and
\begin{equation}
\label{eq:A-iii}
\card \Sigma_{3} > 6K_{1}K_{2}L.
\end{equation}

Let $s$ be a non-zero polynomial of $\bbK \left[ X_{1},X_{2},Y \right]$, whose
partial degrees in the variables $X_{1},X_{2}$ and $Y$ are bounded by $K_{1}$,
$K_{2}$ and $L$, respectively. Then $s$ does not vanish identically on the set
$\Sigma_{1}+\Sigma_{2}+\Sigma_{3}$.
\end{prop}

Notice that a similar result has been obtained by Gouillon \cite{A-G} for polynomials
$s$ of total degree in $X_{1}$ and $X_{2}$ bounded by $2\max \left\{ K_{1}, K_{2} \right\}$,
with a constant $12$
instead of $6$ in the above main condition~\eqref{eq:A-iii} and where $\bbK=\bbC$.

\subsection{Geometrical preliminaries}
\label{sect:A-1}

We embed naturally the group $G$ in the product
\[
\bP = \bP^{1}(\bbK) \times \bP^{1}(\bbK) \times \bP^{1}(\bbK).
\]

For any closed irreducible subvarieties $V \subseteq \bP$ of codimension $0\leq r\leq 3$,
and any triple of integers $(a,b,c)$ with 
\[
a \in \{0,1\}, b \in \{0,1\}, c \in \{0,1\} \text{ and } a+b+c=r,
\]
we define the multidegrees $\delta_{a,b,c}(V)$ as the intersection degree
\[
\delta_{a,b,c}(V)= \card \left\{ V\ \cap \pi_{1}^{-1} \left( L_{a} \right) \cap \pi_{2}^{-1} \left( L_{b} \right) \cap \pi_{3}^{-1} \left( L_{c} \right) \right\},
\]
where $L_{a}$, $L_{b}$ and $L_{c}$ stand for generic linear subvarieties in
$\bP^{1}(\bbK)$ with respective dimensions $a$, $b$ and $c$ (thus $L_{1} =\bP^{1}(\bbK)$
and $L_{0}$ is a point) and where the maps $\pi_{j} : \bP \rightarrow \bP^{1}(\bbK)$
denote the three canonical projections. We also extend to cycles (meaning formal
linear combinations with integer coefficients of closed irreducible subvarieties
of codimension $r$ in $\bP$) the above definition of the multidegrees
$\delta_{a,b,c}$. Let $Z$ be a cycle of codimension $r \leq 2$ in $\bP$ and let
$s \in \bbK \left[ X_{1},U_{1};X_{2},U_{2};Y,V \right]$ be a non-zero polynomial
which is homogeneous of respective degrees $D_{X_{1}}$, $D_{X_{2}}$, $D_{Y}$ in
each of
the three pairs of variables $\left( X_{1},U_{1} \right)$, $\left( X_{2},U_{2} \right)$
and $(Y,V)$. Assume that $s$ does not vanish identically on each component of $Z$. 
Then Bezout's Theorem gives us the multidegrees of the intersection cycle $Z \cdot (s)$
of codimension $r+1$ in $\bP$. For any $a$, $b$ and $c$ as above with $a+b+c=r+1$,
we have the equalities:
\begin{equation}
\label{eq:A-0}
\delta_{a,b,c}(Z \cdot (s))= D_{X_{1}}\delta_{a-1,b,c}(Z)+D_{X_{2}}\delta_{a,b-1,c}(Z)+D_{Y}\delta_{a,b,c-1}(Z),
\end{equation}
where the multidegrees $\delta$ appearing on the right-hand side are understood
to be zero whenever the indices $a-1$ or $b-1$ or $c-1$ are negative.

Now the above Bezout equalities on $\bP$ induce upper bounds on $G$ in the following way.
For any irreducible subvarieties $V\subseteq G$, we denote by 
$\delta_{a,b,c}(V)$ the corresponding multidegree $\delta_{a,b,c} \left( \overline{V} \right)$
of its Zariski closure $\overline{V}$ in $\bP$, and if $Z$ is any cycle in $G$,
that is to say some formal linear combination of irreducible subvarieties of $G$
of the same codimension, we define $\delta_{a,b,c}(Z)$ by linearity. 

Let $s_{1}$, $s_{2}$ and $s_{3}$ be three non-zero polynomials of $\bbK \left[ X_{1},X_{2},Y \right]$
with partial degrees in $X_{1}$, $X_{2}$ and $Y$ respectively bounded by $K_{1}$,
$K_{2}$ and $L$. Denote by $Z_{1} = \left( s_{1} \right)$ the (eventually null)
divisor of the zeroes of $s_{1}$ on $G$ and assume that $s_{2}$ does not vanish
identically on any component of $Z_{1}$. Let $Z_{2}=Z_{1} \cdot \left( s_{2} \right)$
be the (eventually null) intersection cycle on $G$ of codimension $2$. Assume again
that $s_{3}$ does not vanish identically on any component of $Z_{2}$ and put
$Z_{3}=Z_{2} \cdot \left( s_{3} \right)$.
Notice that our assumptions mean equivalently that the sequence $\left( s_{1},s_{2},s_{3} \right)$
is a regular sequence in the local ring of any common zero of $s_{1}$, $s_{2}$
and $s_{3}$
on $G$. Then the above trihomogeneous version of Bezout's theorem in equation~\eqref{eq:A-0}
implies inductively the upper bounds for the multidegrees of the intersection cycles
$Z_{1}$, $Z_{2}$ and $Z_{3}$:
\begin{equation}
\label{eq:A-1}
\delta_{1,0,0} \left( Z_{1} \right) \leq K_{1}, \quad
\delta_{0,1,0} \left( Z_{1} \right) \leq K_{2},
\quad \delta_{0,0,1} \left( Z_{1} \right) \leq L,
\end{equation}
\begin{equation}
\label{eq:A-2}
\delta_{1,1,0} \left( Z_{2} \right) \leq 2K_{1}K_{2}, \quad
\delta_{0,1,1} \left( Z_{2} \right) \leq 2K_{2}L, \quad
\delta_{1,0,1} \left( Z_{2} \right) \leq 2K_{1}L \quad \text{and}
\end{equation}
\begin{equation}
\label{eq:A-3}
\quad \delta_{1,1,1} \left( Z_{3} \right) \leq 6K_{1}K_{2}L .
\end{equation}

\subsection{Proof of Proposition~\ref{prop:laurent-zero-est}}

Suppose on the contrary that there exists a non-zero polynomial $s\in \bbK \left[ X_{1},X_{2},Y \right]$
with partial degrees in $X_{1}$, $X_{2}$ and $Y$ bounded by $K_{1}$, $K_{2}$ and
$L$ and vanishing on $\Sigma_{1}+\Sigma_{2}+\Sigma_{3}$. Then we plan to construct
polynomials $s_{1}$, $s_{2}$ and $s_{3}$ as in Section~\ref{sect:A-1} and vanishing
moreover respectively on the subsets $\Sigma_{1}+\Sigma_{2}+\Sigma_{3}$,
$\Sigma_{2}+\Sigma_{3}$ and $\Sigma_{3}$. Since 
\[
\delta_{1,1,1} \left( Z_{3} \right) \geq \card \Sigma_{3},
\]
the assumption~\eqref{eq:A-iii} of the proposition will contradict equation~\eqref{eq:A-3}.

We start with $s_{1}=s$. Notice that the cycle $Z_{1}=\left( s_{1} \right)$ is
non-zero since the points $\Sigma_{1}+\Sigma_{2}+\Sigma_{3}$ are contained in
its support. 

Let us construct $s_{2}$. Observe first that for any component $V$ of $Z_{1}$,
there exists a translated variety $g+V$, for some $g\in \Sigma_{1}$, which is
not a component of $Z_{1}$. Otherwise by equation~\eqref{eq:A-1}, we should have the
upper bounds
\begin{align*}
\card \left( \Sigma_{1}/H \right) \delta_{1,0,0}(V) \leq \delta_{1,0,0} \left( Z_{1} \right) & \leq K_{1}, \\
\card \left( \Sigma_{1}/H \right) \delta_{0,1,0}(V) \leq \delta_{0,1,0} \left( Z_{1} \right) & \leq K_{2} \quad \text{and} \\
\card \left( \Sigma_{1}/H \right) \delta_{0,0,1}(V) \leq \delta_{0,0,1} \left( Z_{1} \right) & \leq L,
\end{align*}
where $H= \{ g\in G: g+V=V \}$ is the stabiliser of $V$. Clearly $H$ is an algebraic
subgroup of $G$ and $\dim H \leq 2$.

When $H=W\times \bbK^{\times}$, where $W$ is
either $\{0\}$ or a line $aX_{1}+bX_{2}=0$ in $\bbK^{2}$, at least one of the
degrees $\delta_{1,0,0}(V)$ or $\delta_{0,1,0}(V)$ is positive and we get a
contradiction with the first lower bound of \eqref{eq:A-i}.

When $H=W\times\mu$,
with a finite multiplicative group $\mu$, then $\delta_{0,0,1}(V) \geq \card (\mu)$,
and we deduce from the last upper bound
\begin{align*}
\card \left\{ y :
\text{$\exists \left( x_{1}, x_{2} \right) \in \bbK^{2}$ with $\left( x_{1},x_{2},y \right) \in \Sigma_{1}$} \right\}
&\leq \card \left( \Sigma_{1}/(W \times \{1\}) \right) \\
&\leq \card \left( \Sigma_{1}/(W \times \mu) \right) \card (\mu) \leq L,
\end{align*}
which contradicts the second lower bound of \eqref{eq:A-i}.

Therefore, for some
$g \in \Sigma_{1}$, the translated polynomial $s_{1} \circ \tau_{g}$ does not
vanish identically on $V$. Now a generic linear combination $s_{2}$ of the
polynomials $s_{1} \circ \tau_{g}, g \in \Sigma_{1}$ has the required properties.

We construct $s_{3}$ in a similar way, proving first that for any component $V$
of $Z_{2} =Z_{1} \cdot \left( s_{2} \right)$, the translated varieties $g+V, g \in \Sigma_{2}$,
are not all components of $Z_{2}$. Otherwise we should deduce from \eqref{eq:A-2}
the upper bounds
\begin{align}
\label{eq:sigma2-bnd1}
\card \left( \Sigma_{2}/H \right) \delta_{1,1,0}(V) \leq \delta_{1,1,0} \left( Z_{2} \right) & \leq 2K_{1}K_{2},\\
\card \left( \Sigma_{2}/H \right) \delta_{1,0,1}(V) \leq \delta_{1,0,1} \left( Z_{2} \right) & \leq 2K_{1}L \quad \text{and} \nonumber\\
\card \left( \Sigma_{2}/H \right) \delta_{0,1,1}(V) \leq \delta_{0,1,1} \left( Z_{2} \right) & \leq 2K_{2}L, \nonumber
\end{align}
where $H= \{g\in G : g+V=V\}$ is again the stabiliser of $V$. Now $\dim H \leq 1$.
When $H=\{0\} \times \bbK^{\times}$, the curve $V$ is some line $\left( u,v,\bbK^{\times} \right)$
and $\delta_{1,1,0}(V)=1$. Then the first upper bound in \eqref{eq:sigma2-bnd1}
contradicts the second lower bound of \eqref{eq:A-ii}.

Suppose now that $H = W \times \mu$, where $\mu$ is
a finite multiplicative group and $W$ is either $\{0\}$ or a line $aX_{1}+bX_{2}=0$.
The projection $\pi_{1}\times\pi_{2}$ restricted to $V$ is then a finite map on
to its image in $\bbK^{2}$ of degree $\geq \card (\mu)$. Then at least one of the
multidegrees $\delta_{1,0,1}(V)$ or $\delta_{0,1,1}(V)$ is $\geq \card (\mu)$.
Thus we find the upper bounds
\begin{align*}
\card \left\{ \left( ax_{1}+bx_{2},y \right) : \left( x_{1}, x_{2}, y \right) \in \Sigma_{2} \right\}
& \leq \card \left( \Sigma_{2}/(W \times \{1\}) \right) \\
& \leq \card \left( \Sigma_{2}/(W \times \mu) \right) \card (\mu)
\leq 2\max \left\{ K_{1}, K_{2} \right\}L,
\end{align*}
which contradict the first lower bound of \eqref{eq:A-ii}.

Finally, we take for $s_{3}$ a
generic linear combination of the polynomials $s_{1} \circ \tau_{g}$ and $s_{2}\circ\tau_{g}$,
for $g\in \Sigma_{2}$.


\begin{thebibliography}{}
\bibitem{Al}
Y. M. Aleksentsev,
\emph{The Hilbert polynomial and linear forms in the logarithms of algebraic numbers},
Izv. Math. {\bf 72} (2008), 1063--1110.

\bibitem{Ba1}
A. Baker,
\emph{Linear forms in the logarithms of algebraic numbers. I},
Mathematika {\bf 12} (1966), 204--216.

\bibitem{BD}
A. Baker, H. Davenport,
\emph{The equations $3x^{2}-2=y^{2}$ and $8x^{2}-7=z^{2}$},
Quart. J. Math. Oxford Ser. {\bf (2) 20} (1969), 129--137.

\bibitem{Ben-C}
C. D. Bennett, J. Blass, A. M. W. Glass, D. B. Meronk, R. P. Steiner,
\emph{Linear forms in the logarithms of three positive rational numbers},
J. Th\'{e}or. Nombres Bordeaux {\bf 9} (1997), 97--136.

\bibitem{BGMP}
M.A. Bennett, K. Gy\H{o}ry, Mignotte, \'{A}. Pint\'{e}r,
\emph{Binomial Thue equations and polynomial powers},
Comp. Math. {\bf 142} (2006), 1103--1121.

\bibitem{BDMS}
M.A. Bennett, S. Dahmen, Mignotte, S. Siksek,
\emph{Shifted powers in binary recurrence sequences},
Math. Proc. Camb. Phil. Soc. {\bf 158} (2015), 305--329.

\bibitem{BH}
Yu. Bilu, G. Hanrot,
\emph{Solving Thue Equations of High Degree},
J. Number Theory {\bf 60} (1996), 373--392.

\bibitem{Bug}
Y. Bugeaud,
\emph{Linear Forms in Logarithms and Applications},
European Mathematical Society, Zurich, 2018.

\bibitem{BMS1}
Y. Bugeaud, M. Mignotte, S. Siksek,
\emph{Classical and modular approaches to exponential Diophantine equations I. Fibonacci and Lucas perfect powers},
Ann. Math. {\bf 163} (2006), 969--1018.

\bibitem{BMS2}
Y. Bugeaud, M. Mignotte, S. Siksek,
\emph{Classical and Modular Approaches to Exponential Diophantine Equations II. The Lebesgue--Nagell Equation},
Comp. Math. {\bf 142} (2006), 31--62.

\bibitem{D}
J. Dieudonn\'{e},
\emph{Calcul infinit\'{e}simal} (2nd ed),
Hermann, Paris, 1980.

\bibitem{DP}
A. Dujella, A. Peth\H{o},
\emph{A generalization of a theorem of Baker and Davenport},
Quart. J. Math. Oxford Ser. {\bf (2) 49} (1998), 291--306.


\bibitem{G1}
N. Gouillon,
\emph{Un lemme de z\'{e}ros},
Comptes Rendus Acad. Sci. Paris, Ser.~I, {\bf 335} (2002), 167--170.

\bibitem{G2}
N. Gouillon,
\emph{Minorations explicites de formes lin\'{e}aires en deux logarithmes},
Th\`{e}se de Docteur de l'universit\'{e} de la M\'{e}diterran\'{e}e - Aix-Marseille II, (2003)
\url{https://tel.archives-ouvertes.fr/tel-00003964}.


\bibitem{L1}
M. Laurent,
\emph{Sur quelques r\'{e}sultats r\'{e}cents de transcendance},
Ast\'{e}risque {\bf 198--200} (1991), 209--230.

\bibitem{L2}
M. Laurent,
\emph{Hauteurs de matrices d'interpolation},
Approximations diophantiennes et nombres transcendants, Luminy (1990), 
ed.~P. Philippon, de Gruyter (1992), 215--238.

\bibitem{L3}
M. Laurent,
\emph{Linear forms in two logarithms and interpolation determinants},
Acta Arith. {\bf 66} (1994), 181--199.

\bibitem{L4}
M. Laurent, Personal communication to M. Mignotte, Nov. 2003.

\bibitem{L5}
M. Laurent,
\emph{Linear forms in two logarithms and interpolation determinants II},
Acta Arith. {\bf 133} (2008), 325--348.

\bibitem{LMN}
M. Laurent, M. Mignotte, Y. Nesterenko,
\emph{Formes lin\'{e}aires en deux logarithmes et d\'{e}terminants d'interpolation},
J. Number Theory {\bf 55} (1995), 285--321.

\bibitem{LLL}
A.K. Lenstra, H.W. Lenstra Jr., L. Lov\'{a}sz,
\emph{Factoring polynomials with rational coefficients},
Math. Ann. {\bf 261} (1982), 515--534.

\bibitem{Matveev}
E. M. Matveev,
\emph{An explicit lower bound for a homogeneous rational linear form in logarithms of algebraic numbers. II},
Izv. Ross. Akad. Nauk Ser. Mat. {\bf 64} (2000), 125--180.
English transl. in Izv. Math. {\bf 64} (2000), 1217--1269.

\bibitem{Pari}
The PARI~Group,
PARI/GP version \texttt{2.14.0}, Univ. Bordeaux, 2021,
\url{http://pari.math.u-bordeaux.fr/}.

\bibitem{Pet}
A. Peth\H{o},
\emph{Perfect powers in second order linear recurrences},
J. Number Theory {\bf 15} (1982), 5--13.

\bibitem{PW}
A. Peth\H{o}, B.M.M. de Weger,
\emph{Products of Prime Powers in Binary Recurrence Sequences Part I:
The hyperbolic Case, with an Application to the generalized Ramanujan-Nagell equation},
Math. Comp. {\bf 47} (1986), 713--727.

\bibitem{Phil}
P. Philippon, 
{\em Lemmes de z\'{e}ros dans les groupes alg\'{e}briques commutatifs},
Bull. Soc. Math. France, 114 (1987), 355--383. Errata et addenda, id., 115 (1987), 397--398.

\bibitem{RS}
J. B. Rosser, L. Schoenfeld,
\emph{Approximate Formulas for Some Functions of Prime Numbers},
Ill. J. Math. {\bf 6} (1962), 64--94.

\bibitem{SS}
T. N. Shorey, C. L. Stewart,
\emph{On the Diophantine equation $ax^{2t}+bx^{t}y+cy^{2}=d$ and pure powers in recurrence sequences},
Math. Scand. {\bf 52} (1983), 24--36.

\bibitem{Tij}
R. Tijdeman,
\emph{On the equation of Catalan}, Acta Arith. {\bf 29} (1976), 197--209.

\bibitem{TW}
N. Tzanakis, B.~M.~M. de Weger,
\emph{On the practical solution of the Thue equation},
J. Number Theory {\bf 31} (1989), 99--132.

\bibitem{W}
M. Waldschmidt,
\emph{Diophantine Approximation on Linear Algebraic Groups},
Springer, Berlin, 2000.
\end{thebibliography}

\begin{thebibliography}{}
\bibitem{A-G} N. Gouillon,
{\em Un lemme de z\'{e}ros},
Comptes Rendus Acad. Sci. Paris, Ser.~I, {\bf 335} (2002), 167--170.

\bibitem{A-Mas} D. W. Masser,
{\em On polynomials and exponential polynomials in several variables},
Invent. Math. {\bf 63} (1981), 81--95.

\bibitem{A-Phil} P. Philippon, 
{\em Lemmes de z\'{e}ros dans les groupes alg\'{e}briques commutatifs},
Bull. Soc. Math. France, 114 (1987), 355--383. Errata et addenda, id., 115 (1987), 397--398.
\end{thebibliography}
\end{document}